\documentclass[a4paper]{article}
\usepackage[utf8]{inputenc}
\usepackage{geometry}
\usepackage[T1]{fontenc}
\usepackage{lmodern}
\usepackage{array}
\usepackage{graphicx}
\usepackage{amssymb, amsfonts, amsthm, amsmath}
\usepackage[english]{babel}
\usepackage{url}

\theoremstyle{plain}
\newtheorem{theorem}{Theorem}[section]
\newtheorem{corollary}[theorem]{Corollary}
\newtheorem{lemma}[theorem]{Lemma}
\newtheorem{proposition}[theorem]{Proposition}

\theoremstyle{definition}

\theoremstyle{remark}
\newtheorem{remark}[theorem]{Remark}

\newcommand{\N}{\mathbb{N}}
\newcommand{\Z}{\mathbb{Z}}
\newcommand{\R}{\mathbb{R}}

\numberwithin{equation}{section}

\newcommand{\ind}[1]{\mathbf{1}_{\left\{#1\right\}}}
\newcommand{\floor}[1]{{\left\lfloor #1 \right\rfloor}}
\newcommand{\ceil}[1]{{\left\lceil #1 \right\rceil}}
\newcommand{\norme}[1]{{\left\Vert #1 \right\Vert}}

\DeclareMathOperator{\E}{\mathbf{E}}

\renewcommand{\P}{\mathbf{P}}
\DeclareMathOperator{\Var}{\mathbf{V}\mathrm{ar}}

\newcommand{\calF}{\mathcal{F}}

\newcommand{\calL}{\mathcal{L}}

\newcommand{\calR}{\mathcal{R}}

\newcommand{\conv}[2]{{\underset{#1\to #2}{\longrightarrow}}}

\newcommand{\T}{\mathbf{T}}
\newcommand{\calC}{\mathcal{C}}
\newcommand{\calD}{\mathcal{D}}
\renewcommand{\tilde}[1]{\widetilde{#1}}
\newcommand{\norm}[1]{\left\Vert #1 \right\Vert}
\renewcommand{\hat}[1]{\widehat{#1}}
\renewcommand{\bar}[1]{\overline{#1}}
\renewcommand{\phi}{\varphi}
\newcommand{\calA}{\mathcal{A}}
\newcommand{\calB}{\mathcal{B}}
\newcommand{\calW}{\mathcal{W}}

\newcommand{\crochet}[1]{\left\langle #1 \right\rangle}

\newcommand{\rmL}{\mathrm{L}}

\newcommand{\calG}{\mathcal{G}}

\title{Maximal displacement of a branching random walk in time-inhomogeneous environment}
\author{Bastien Mallein\footnote{LPMA, Univ. P. et M. Curie (Paris 6). Research partially supported by the ANR project MEMEMO.}\footnote{DMA, École Normale Supérieure (Paris).}}
\date{\today}

\begin{document}

\maketitle

\begin{abstract}
Consider a branching random walk evolving in a macroscopic time-inhomogeneous environment, that scales with the length $n$ of the process under study. We compute the first two terms of the asymptotic of the maximal displacement at time $n$. The coefficient of the first (ballistic) order is obtained as the solution of an optimization problem, while the second term, of order $n^{1/3}$, comes from time-inhomogeneous random walk estimates, that may be of independent interest. This result partially answers a conjecture of Fang and Zeitouni. Same techniques are used to obtain the asymptotic of other quantities, such as the consistent maximal displacement.
\end{abstract}

\section{Introduction}
\label{sec:introduction}

A \textit{time-inhomogeneous branching random walk on $\R$} is a process evolving as follows: it starts with one individual located at the origin at time 0. At each time $k \in \N$, every individual alive at generation $k$ dies, giving birth to a certain number of children, which are positioned around their parent according to independent versions of a point process, whose law may depend on the time. If the law of the point process governing the reproduction does not depend on the time, the process is simply called a \textit{branching random walk}.

We write $M_n$ for the maximal displacement at time $n$ in the branching random walk. The asymptotic of $M_n$, when the reproduction law does not depend on the time, is well understood. The first results come from early works of Hammersley \cite{Ham74}, Kingman \cite{Kin75} and Biggins \cite{Big76}, who obtained the existence of an explicit constant $v$ such that $\frac{M_n}{n} \conv{n}{+\infty} v$; Hu and Shi \cite{HuS09} and Addario-Berry and Reed \cite{ABR09} exhibited the logarithmic second term in the asymptotic $c_0 \log n$ for some explicit constant $c_0$; finally Aidekon proved in \cite{Aid13} the convergence of $M_n - nv - c_0 \log n$ to a random shift of a Gumble distribution.

A model of time-inhomogeneous branching random walk has been introduced by Fang and Zeitouni in \cite{FaZ11}. In this process, the reproduction law of the individuals evolves at macroscopic scale: every individual has two children, each of which moves independently according to a centred Gaussian random variable, with variance $\sigma_1^2$ (respectively $\sigma_2^2$) if the child is alive before time $n/2$ (resp. between times $n/2$ and $n$). In this branching random walk, the asymptotic of $M_n$ is still given by a ballistic first order, plus a negative logarithmic correction and bounded fluctuations of order one. However, first and second order strongly depend on the sign of $\sigma^2_1-\sigma^2_2$, and the coefficient of the logarithmic correction exhibits a phase transition as $\sigma^2_1$ grows bigger than $\sigma^2_2$. This result has been generalised to more general reproduction laws in \cite{Mal14}.

The study of the maximal displacement in a time-inhomogeneous branching Brownian motion, the continuous time counterpart of the branching random walk, with smoothly varying environment has been started in \cite{FaZ12}. In this process individuals split into 2 children at rate 1, and move using independent Gaussian diffusion with variance $\sigma^2_{t/n}$ at time $t \in [0,n]$. In \cite{FaZ12}, Fang and Zeitouni conjectured that, under mild hypotheses, for an explicit $v^*$, the sequence $(M_n - nv^* - g(n), n \geq 1)$ is tensed for a given function $g$ such that
\[ -\infty < \liminf_{n \to +\infty} \frac{g(n)}{n^{1/3}} \leq \limsup_{n \to +\infty} \frac{g(n)}{n^{1/3}} \leq 0.\]
They proved this result for smoothly decreasing variance. Using PDE techniques, Nolen, Roquejoffre and Ryzhik \cite{NRR13} established, again in the case of decreasing variances, that $g(n) = l^* n^{1/3} + O(\log n)$ for some explicit constant $l^*$. Maillard and Zeitouni \cite{MaZ13} obtained, independently from our result, the more precise result $g(n) = l^* n^{1/3} - c_1  \log n$, for some explicit $c_1$, using proofs similar to the ones presented here, based on first and second moment particles computation and the study of a particular PDE.

We prove in this article that for a large class of time-inhomogeneous branching random walks, $M_n - nv^* \sim_{n \to +\infty} l^* n^{1/3}$ for some explicit constants $v^*$ and $l^*$. Conversely to previous articles in the domain, the displacements we authorize are non necessarily Gaussian. Moreover, the law of the number of children is correlated with the displacement and depends on the time. More importantly, we do not restrict ourselves to --a hypothesis similar to-- decreasing variance. The decreasing variance case remains interesting, as in this case quantities such as $v^*$ and $l^*$ admit a closed expression. We do not prove in this article there exists a function $g$ such that $M_n-n v^* -g(n)$ is tight, therefore we do not exactly answer to the conjecture of Fang and Zeitouni. However, Fang \cite{Fan12} proved that the sequence $M_n$ shifted by its median is tight for a large class of \textit{generalized branching random walks}. This class does not exactly covers the class of time-inhomogeneous branching random walks we consider, but there is a non-trivial intersection. On this subset, the conjecture is then proved applying Theorem \ref{thm:main}.

To address the fact the displacements are non Gaussian, we use the Sakhanenko estimate \cite{Sak84}. This theorem couples sums of independent random variables with a Brownian motion. Consequently we use Brownian estimates to compute our branching random walk asymptotics. The non-monotonicity of the variance leads to additional concerns. We discuss in Section \ref{subsec:heuristic} a formula for $\lim_{n \to +\infty} \frac{M_n}{n}$, given as the solution of an optimization problem under constraints \eqref{eqn:existence_max}. This equation is solved in Section \ref{sec:optimization}, using some analytical tools such as the existence of Lagrange multipliers in Banach spaces described in \cite{Kur76}. When we solve this problem, an increasing function appears naturally, that replaces the inverse of the variance in computations of \cite{NRR13} and \cite{MaZ13}.

\paragraph*{Notation}
In this article, $c,C$ stand for two positive constants, respectively small enough and large enough, which may change from line to line, and depend only on the law of the random variables we consider. We always assume the convention $\max \emptyset = - \infty$ and $\min \emptyset = +\infty$. For $x \in \R$, we write $x_+ = \max(x,0)$, $x_- = \max(-x,0)$ and $\log_+(x) = (\log x)_+$. For all function $f : [0,1] \to \R$, we say that $f$ is \textit{Riemann-integrable} if
\[
  \liminf_{n \to +\infty} \frac{1}{n} \sum_{k=0}^{n-1} \min_{s \in [\frac{k-1}{n},\frac{k+2}{n}]} f_s = \limsup_{n \to +\infty} \frac{1}{n} \sum_{k=0}^{n-1} \max_{s \in [\frac{k-1}{n},\frac{k+2}{n}]} f_s,
\]
and this common value is written $\int_0^1 f_s ds$. In particular, a Riemann-integrable function is bounded. A subset $F$ of $[0,1]$ is said to be \textit{Riemann-integrable} if and only if $\mathbf{1}_F$ is Riemann-integrable. For example, an open subset of $(0,1)$ of Lebesgue measure 1/2 that contains all rational numbers is not Riemann-integrable. Finally, if $A$ is a measurable event, we write $\E(\cdot ; A)$ for $\E(\cdot \mathbf{1}_A)$. An index of notations is present in Appendix \ref{app:notation}

The rest of the introduction is organised as follows. We start with some branching random walk notation in Section \ref{subsec:notation}. We describe in Section \ref{subsec:heuristic} the optimization problem that gives the speed of the time-inhomogeneous branching random walk. In Section \ref{subsec:main}, we state the main result of this article: the asymptotic of the maximal displacement in a time-inhomogeneous branching random walk. We also introduce another quantity of interest for the branching random walk: the consistent maximal displacement with respect to the optimal path in Section \ref{subsec:cmbintro}. Finally, in Section \ref{subsec:introrw}, we introduce some of the random walk estimates that are used to estimate moments of the branching random walk.

\subsection{Branching random walk notation}
\label{subsec:notation}

Let $\T$ be a plane rooted tree --which can be encoded according to the Ulam-Harris notation, for example-- and $V : \T \to \R$. We call $(\T,V)$ a (plane, rooted) \textit{marked tree}. For a given individual $u \in \T$, we write $|u|$ for the generation to which $u$ belongs, i.e. its distance to the root $\emptyset$ according to the graph distance in the tree $\T$. If $u$ is not the root, we denote by $\pi u$ the parent of $u$. For $k \leq |u|$, we write $u_k$ its ancestor alive at generation $k$, with $u_0 = \emptyset$. We call $V(u)$ the position of $u$ and $(V(u_0), V(u_1), \ldots V(u))$ the path followed by $u$.

For any $n \in \N$, we call $\{ u \in \T, |u|=n\}$ the \textit{$n^\mathrm{th}$ generation of the tree $\T$}, which is often abbreviated as $\{|u|=n\}$ if the notation is clear in the context. For $u \in \T$, the set of children of $u$ is written $\Omega(u) = \{ u' \in \T : \pi u' = u\}$, and let $L^u = (V(u')-V(u), u' \in \Omega(u))$ be the point process of the displacement of the children with respect to their parent. In this article, the point processes we consider are a finite or infinite collection of real numbers --repetitions are allowed-- that admit a maximal element and no accumulation point.

We add a piece of notation on random point processes. Let $L$ be a point process with law $\calL$, i.e. a random variable taking values in $\cup_{k \in \Z_+ \cup \{+\infty\}} \R^k$. As the point processes we consider have a maximal element and no accumulation point, we can write $L = (\ell_1,\ldots \ell_N)$, where $N \in \Z_+ \cup \{+\infty\}$ is the random variable of the number of points in the process, and $\ell_1 \geq \ell_2 \geq \cdots \geq \ell_N$ is the set of points in $L$, listed in the decreasing order, with the convention $\ell_{+\infty} = -\infty$.

We now consider $(\calL_t, t \in [0,1])$ a family of laws of point processes. For $t \in [0,1]$, we write $L_t$ for a point process with law $\calL_t$. For $\theta \geq 0$, we denote by $\kappa_t(\theta) = \log \E\left[ \sum_{\ell \in L_t} e^{\theta \ell}\right]$ the log-Laplace transform of $\theta$ and by $\kappa^*_t(a) = \sup_{\theta > 0} [\theta a - \kappa_t(\theta)]$ its Fenchel-Legendre transform.  We recall the following elementary fact: if $\kappa_t^*$ is differentiable at $a$, then putting $\theta = \partial_a \kappa^*_t(a)$, we  have
\begin{equation}
  \label{eqn:legendreestimate}
  \theta a - \kappa_t(\theta) = \kappa^*_t(a).
\end{equation}

The \textit{branching random walk of length $n$ in the time-inhomogeneous environment $(\calL_t, t \in [0,1])$} (BRWtie) is the marked tree $(\T^{(n)},V^{(n)})$ such that $\{L^u, u \in \T^{(n)}\}$ forms a family of independent point processes, where $L^u$ has law $\calL_\frac{|u|+1}{n}$ if $|u|<n$, and is empty otherwise. In particular, $\T^{(n)}$ is the (time-inhomogeneous) Galton-Watson tree of the genealogy of this process. When the value of $n$ is clear in the context, we often omit the superscript, to lighten notations.

It is often easier to consider processes that never get extinct, and have supercritical offspring above a given straight line with slope $p$. We introduce the (strong) supercritical assumption
\begin{equation}
  \label{eqn:breeding}
  \forall t \in [0,1], \P(L_t = \emptyset) = 0 \quad \mathrm{and} \quad \exists p \in \R : \inf_{t \in [0,1]} \P(\# \{ \ell \in L_t : \ell \geq p \} \geq 2)>0.
\end{equation}
Such a strong supercritical assumption is not always necessary, but is technically convenient to obtain concentration inequalities for the maximal displacement. It is also helpful to guarantee the existence of a solution of the optimization problem that defines $v^*$. The second part of the assumption is used only once in this article: to prove Lemma 5.4 (which is used to prove the second equation of Lemma 5.6).

In this article, we assume that the function $t \to \calL_t$ satisfies some strong regularity assumptions. We write 
\begin{equation}
  \label{eqn:defDomain} D = \{ (t,\theta) \in [0,1] \times [0,+\infty) : \kappa_t(\theta)< +\infty\} \quad \mathrm{and} \quad D^* = \{ (t,a) : \kappa^*_t(a)<+\infty \},
\end{equation}
and we assume that $D$ and $D^*$ are non-empty, that $D$ (resp. $D^*$) is open in $[0,1] \times [0,+\infty)$ (resp. $[0,1] \times \R$) and that
\begin{equation}
  \label{eqn:regularity}
  \kappa \in \calC^{1,2}\left(D\right) \text{ and } \kappa^* \in \calC^{1,2}\left(D^*\right).
\end{equation}
These regularity assumptions are used to ensure that the solution of the optimization problem defining $v^*$ is a regular point in some sense. If we know, by some other way, that the solution is regular, then these assumptions are no longer needed. These assumptions imply that the maximum of the point processes we consider has at least exponential tails, and that the probability that this maximum is equal to the essential supremum is zero. They are not optimal, but sufficient to define a quantity that we prove to be the speed of the BRWtie. For example, a random number of i.i.d. random variables with exponential left tails satisfy the above condition. Conversely, heavy tailed random variables, or if the maximum of the point process verifies
\[
  \P(\max\{ \ell \in L \} \geq x ) \sim_{x \to +\infty} x^{-1 - \epsilon} e^{-x}
\]
will not satisfy \eqref{eqn:defDomain}.

\subsection{The optimization problem}
\label{subsec:heuristic}

We write $\calC$ for the set of continuous functions, and $\calD$ for the set of càdlàg --right-continuous with left limits at each point-- functions on $[0,1]$ which are continuous at point $1$. To a function $b \in \calD$, we associate a path --i.e. a sequence-- of length $n$ defined for $k \leq n$ by $\bar{b}^{(n)}_k = \sum_{j=1}^k b_{j/n}$. We say that $b$ is the \textit{speed profile} of the path $\bar{b}^{(n)}$, and we introduce
\[
  K^* :
  \begin{array}{rcl}
  \calD & \longrightarrow & \calC\\
  b & \longmapsto & \left(\int_0^t \kappa^*_s(b_s) ds, t \in [0,1]\right).
  \end{array}
\]
By standard computations on branching random walks (see, e.g. \cite{BigBO}), for all $t \in [0,1]$, the mean number of individuals that follow the path $\bar{b}^{(n)}$ --i.e. that stay at all time within distance $\sqrt{n}$ from the path-- until time $tn$ verifies
\[ \frac{1}{n} \log \E\left[ \sum_{|u|=\floor{nt}} \ind{ \left|V(u_k) - \bar{b}^{(n)}_{k/n} \right| < \sqrt{n}, k \leq nt} \right] \approx_{n \to +\infty} -K^*(b)_t.\]
Therefore, $e^{-n K^*(b)_t}$ is a good approximation of the number of individuals that stay close to the path $\bar{b}^{(n)}$ until time $tn$. 

If there exists $t_0 \in (0,1)$ such that $K^*(b)_{t_0}>0$, then with high probability, there is no individual who stayed close to this path until time $n t_0$. Conversely, if for all $t \in [0,1]$, $K^*(b)_t \leq 0$, one would expect to find with positive probability at least one individual at time $n$ to the right of $\bar{b}^{(n)}$. Following this heuristic, we introduce
\begin{equation}
  \label{eqn:defv}
  v^* = \sup \left\{ \int_0^1 b_s ds, b \in \calD : \forall t \in [0,1], K^*(b)_t \leq 0 \right\}.
\end{equation}
We expect $nv^*$ to be the highest terminal point in the set of paths that are followed with positive probability by individuals in the branching random walk. Therefore we expect that $\lim_{n \to +\infty} \frac{M_n}{n} = v^*$ in probability.

We are interested in the path that realises the maximum in \eqref{eqn:defv}. We define the optimization problem under constraints
\begin{equation}
  \label{eqn:existence_max}
  \exists a \in \calD : v^* = \int_0^1 a_s ds \quad \mathrm{and} \quad \forall t \in [0,1], K^*(a)_t \leq 0.
\end{equation}
We say that $a$ is a solution to \eqref{eqn:existence_max} if $\int_0^1 a_s ds = 0$ and $K^*(a)$ is non-positive. Such a solution $a$ is called an \textit{optimal speed profile}, and $\bar{a}^{(n)}$ an \textit{optimal path} for the branching random walk. The path followed by the rightmost individual at time $n$ is an optimal path, thus describing such a path is interesting to obtain the second order correction. In effect, as highlighted in the time-homogeneous branching random walk in \cite{AiS10}, the second order of the asymptotic of $M_n$ is linked to the difficulty for a random walk to follow the optimal path.
\begin{proposition}
\label{prop:regularity}
Let $a \in \calD$. Under the assumptions \eqref{eqn:breeding} and \eqref{eqn:regularity}, $a$ is a solution to \eqref{eqn:existence_max}, i.e.
\[
  v^* = \int_0^1 a_s ds \quad \mathrm{and} \quad \forall t \in [0,1], K^*(a)_t \leq 0,
\]
if and only if, setting $\theta_t = \partial_a \kappa^*_t(a_t)$, we have
\begin{enumerate}
  \item $\theta$ is positive and non-decreasing ;
  \item $K^*(a)_1=0$ ;
  \item $\int_0^1 K^*(a)_s d\theta^{-1}_s=0$.
\end{enumerate}
There exists a unique solution $a$ to \eqref{eqn:existence_max}, and $a$ (and $\theta$) are Lipschitz.
\end{proposition}

Consequence of this proposition, we now call $a$ the optimal speed profile, and $\bar{a}$ the optimal path. This result is proved in Section \ref{sec:optimization}. The optimization problem \eqref{eqn:existence_max} is similar to the one solved for the GREM by Bovier and Kurkova \cite{BoK07}.

\subsection{Asymptotic of the maximal displacement}
\label{subsec:main}

Under the assumptions \eqref{eqn:regularity} and \eqref{eqn:existence_max}, let $a$ be the optimal speed profile characterised by Proposition \ref{prop:regularity}. For $t \in [0,1]$ we denote by
\begin{equation}
  \label{eqn:thetadef}
  \theta_t = \partial_a \kappa^*_t(a_t) \quad \mathrm{and} \quad \sigma_t^2 = \partial^2_\theta \kappa_t(\theta_t).
\end{equation}
To obtain the asymptotic of the maximal displacement, we introduce the following regularity assumptions:
\begin{equation}
  \label{eqn:regularitytheta}
  \theta \text{ is absolutely continuous, with a Riemann-integrable derivative } \dot{\theta},
\end{equation}
\begin{equation}
  \label{eqn:regularityenergy}
  \{ t \in [0,1] : K^*_t(a)=0 \} \text{ is Riemann-integrable}.
\end{equation}
Finally, we make the following second order integrability assumption:
\begin{equation}
  \label{eqn:integrability2}
  \sup_{t \in [0,1]} \E\left[ \left( \sum_{\ell \in L_t} e^{\theta_t \ell_t} \right)^2 \right] < +\infty.
\end{equation}

\begin{remark}
This last integrability condition is not optimal. Using the spinal decomposition as well as estimates on random walks enriched by random variables depending only on the last step, as in \cite{Mal14} would lead to an integrability condition of the form $\E(X (\log X)^2) < +\infty$ instead of \eqref{eqn:integrability2}. However, this assumption considerably simplifies the proofs.
\end{remark}

The main result of this article is the following.
\begin{theorem}[Maximal displacement in the BRWtie]
\label{thm:main}
We assume \eqref{eqn:breeding}, \eqref{eqn:regularity}, \eqref{eqn:regularitytheta}, \eqref{eqn:regularityenergy} and \eqref{eqn:integrability2} are verified. We write $\alpha_1$ for the largest zero of the Airy function of first kind --recall that $\alpha_1  \approx -2.3381...$-- and we set
\begin{equation}
  \label{eqn:definel}
  l^* = \frac{\alpha_1}{2^{1/3}} \int_0^1 \frac{(\dot{\theta}_s \sigma_s)^{2/3}}{\theta_s} ds \leq 0.
\end{equation}
Then we have for all $l>0$,
\[ \lim_{n \to +\infty} \frac{1}{n^{1/3}} \log \P\left(M_n \geq nv^* + (l^*+l)n^{1/3}\right) = -\theta_0 l,\]
and for all $\epsilon>0$,
\[ \limsup_{n \to +\infty}  \frac{1}{n^{1/3}} \log \P\left( \left|M_n - nv^* -l^* n^{1/3}\right| \geq \epsilon n^{1/3}\right) < 0.\]
\end{theorem}

This theorem is proved in Section \ref{sec:maxdis}. The presence of the largest zero of the Airy function of first kind is closely related to the asymptotic of the Laplace transform of the area under a Brownian motion staying positive,
\[\E\left( e^{-\int_0^t B_s ds} ; B_s \geq 0, s \leq t\right) \approx_{t \to +\infty} e^{\frac{\alpha_1}{2^{1/3}} t + o(t)}.\]
The fact that the second order of $M_n$ is $n^{1/3}$ can be explained as follows: when $\theta$ is strictly increasing at time $t$, the optimal path has to stay very close to the boundary of the branching random walk at time $nt$. In particular, if $\theta$ is strictly increasing on $[0,1]$, the optimal path has to stay close to the boundary. The $n^{1/3}$ second order is then similar to the asymptotic of the consistent minimal displacement for the time-homogeneous branching random walk, which is of order $n^{1/3}$, as proved in \cite{FaZ10,FHS12}.

\subsection{Consistent maximal displacement}
\label{subsec:cmbintro}
The arguments we develop for the proof of Theorem \ref{thm:main} can easily be extended to obtain the asymptotic of the consistent maximal displacement with respect to the optimal path in the branching random walk, which we define now. For $n \in \N$ and $u \in \T^{(n)}$, we denote by
\[
  \Lambda(u) = \max_{k \leq |u|} \left[ \bar{a}^{(n)}_k - V(u_k) \right],
\]
the maximal delay an ancestor of the individual $u$ has with respect to the optimal path. The consistent maximal displacement with respect to the optimal path is defined by
\begin{equation}
  \label{eqn:defineCMD}
  \Lambda_n = \min_{u \in \T^{(n)}, |u|=n} \Lambda(u).
\end{equation}
This quantity corresponds to the smallest distance from the optimal path $\bar{a}$ at which one can put a barrier below which individuals get killed such that the global system still survives. The consistent maximal displacement has been studied for time-homogeneous branching random walks in \cite{FaZ10}. We obtain the following asymptotic for the consistent maximal displacement in the BRWtie.
\begin{theorem}
\label{thm:cmd}
Under the assumptions \eqref{eqn:breeding}, \eqref{eqn:regularity}, \eqref{eqn:existence_max}, \eqref{eqn:regularitytheta}, \eqref{eqn:regularityenergy} and \eqref{eqn:integrability2}, there exists $\lambda^* \leq -l^*$, defined in Section~\ref{subsec:cmd} such that for any $\lambda \in (0,\lambda^*)$,
\[
  \lim_{n \to +\infty} \frac{1}{n^{1/3}} \log \P\left( \Lambda_n \leq (\lambda^*-\lambda) n^{1/3}\right) = - \theta_0 \lambda,
\]
and for any $\epsilon>0$,
\[  \limsup_{n \to +\infty} \frac{1}{n^{1/3}} \log \P(|\Lambda_n - \lambda^* n^{1/3}| \geq \epsilon n^{1/3})  < 0.\]
\end{theorem}

\begin{remark}
We observe that if $u \in \T^{(n)}$ verifies $V(u) = M_n$, then $\Lambda^*(u) \leq nv^* - M_n$. As a consequence, the inequality $L_n \leq nv^* - M_n$ holds almost surely, which proves that $\lambda^* \leq -l^*$, as soon as these quantities exist.
\end{remark}

In Theorem \ref{thm:cmd}, we give the asymptotic of the consistent maximal displacement with respect to the optimal path. However, this is not the only path one may choose to consider. For example, one can choose the ``natural speed path'', in which the speed profile is a function $v \in \calC$ defined by $v_t = \inf_{\theta > 0} \frac{\kappa_t(\theta)}{\theta}$. Note that $v_t$ is the speed of a time-homogeneous branching random walk with reproduction law $\calL_t$. As for all $t \in [0,1]$, $K^*(v)_t=0$, for any $\lambda > 0$, the number of individuals that stayed above $\bar{v}^{(n)} - \lambda n^{1/3}$ at all time $k \leq n$ is of order $e^{O(n^{1/3})}$.

In Section \ref{sec:path}, we provide a time-inhomogeneous version of the Many-to-one lemma, that links additive moments of the branching random walk with time-inhomogeneous random walk estimates. To prove Theorems \ref{thm:main} and \ref{thm:cmd}, we use random walk estimates that are proved in Section \ref{sec:rw}.

\subsection{Airy functions and random walk estimates}
\label{subsec:introrw}

We introduce a few basic property on Airy functions, that can be found in \cite{AbS64}. The \textit{Airy function of first kind} $\mathrm{Ai}$ can be defined, for $x \in \R$, by the improper integral
\begin{equation}
  \label{eqn:airy}
 \mathrm{Ai}(x) = \frac{1}{\pi}\lim_{t \to + \infty} \int_0^t \cos\left( \tfrac{s^3}{3} + x s \right)ds,
\end{equation}
and the \textit{Airy function of second kind} $\mathrm{Bi}$ by
\begin{equation}
  \label{eqn:airy2}
 \mathrm{Bi}(x) = \frac{1}{\pi} \lim_{t \to + \infty} \int_0^t \exp\left( -\tfrac{s^3}{3} + x s \right) + \sin\left(\tfrac{s^3}{3} + x s \right)ds.
\end{equation}
These two functions form a basis of the space of functions solutions to
\[
  \forall x \in \R, y''(x) - x y(x) = 0,
\]
and verify $\lim_{x \to +\infty} \mathrm{Ai}(x) = 0$ and $\lim_{x \to +\infty} \mathrm{Bi}(x) = +\infty$. The equation $\mathrm{Ai}(x) = 0$ has an infinitely countable number of solutions, all negative with no accumulation points, which are listed in the decreasing order in the following manner: $0 > \alpha_1 > \alpha_2 > \cdots$.

The Laplace transform of the area below a random walk, or a Brownian motion, conditioned to stay positive admits an asymptotic behaviour linked to the largest zero of $\mathrm{Ai}$, as proved by Darling \cite{Dar83}, Louchard \cite{Lou84} and Tak\'acs \cite{Tak92}. This result still holds in time-inhomogeneous settings. Let $(X_{n,k}, n \geq 1, k \leq n)$ be a triangular array of independent centred random variables. We assume that
\begin{equation}
  \label{eqn:variance_rw}
  \exists \sigma \in \calC([0,1],(0,+\infty)) : \forall n \in \N, k \leq n, \E(X_{n,k}^2) = \sigma_{k/n}^2,
\end{equation}
\begin{equation}
  \label{eqn:integrability_rw}
  \exists \mu > 0 : \E\left[ e^{\mu |X_{n,k}|} \right] < +\infty.
\end{equation}
We write $S^{(n)}_k = \sum_{j=1}^k X_{n,j}$ for the time-inhomogeneous random walk.
\begin{theorem}[Time-inhomogeneous Tak\'acs estimate]
\label{thm:taktie}
Under \eqref{eqn:variance_rw} and \eqref{eqn:integrability_rw}, for any continuous function $g$ such that $g(0)>0$ and any absolutely continuous increasing function $h$ with a Riemann-integrable derivative $\dot{h}$, we have
\begin{multline*}
  \lim_{n \to +\infty} \frac{1}{n^{1/3}} \log \E\left[ \exp\left(- \sum_{j=1}^{n} (h_{j/n}-h_{(j-1)/n}) S^{(n)}_{j}\right) ; S_j \leq g_{j/n} n^{1/3}, j \leq n \right] \\
  = \int_0^1 \left(\dot{h}_s g_s + \frac{a_1}{2^{1/3}} (\dot{h}_s \sigma_s)^{2/3}\right) ds.
\end{multline*}
\end{theorem}

This result is, in some sense, similar to the Mogul'ski\u\i{} estimate \cite{Mog74}, which gives the asymptotic of the probability for a random walk to stay in an interval of length $n^{1/3}$. A time-inhomogeneous version of this result, with an additional exponential weight, holds again. To state this result, we introduce a function $\Psi$, defined in the following lemma.
\begin{lemma}
\label{lem:existencePsi}
Let $B$ be a Brownian motion. There exists a unique convex function $\Psi : \R \to \R$ such that for all $h \in \R$
\begin{equation}
 \label{eqn:definepsi}
 \lim_{t \to +\infty} \frac{1}{t} \log \sup_{x \in [0,1]} \E_x\left[ e^{ -h \int_0^t B_s ds} ; B_s \in [0,1], s \in [0,t] \right] = \Psi(h).
\end{equation}
\end{lemma}

\begin{remark}
We show in Appendix \ref{subsec:twosided} that $\Psi$ admits the following alternative definition:
\[
  \forall h > 0, \Psi(h) =  \frac{h^{2/3}}{2^{1/3}} \sup\left\{\lambda \leq 0 : \mathrm{Ai}\left( \lambda \right) \mathrm{Bi} \left( \lambda + (2h)^{1/3} \right) - \mathrm{Bi}\left( \lambda \right) \mathrm{Ai}\left( \lambda + (2h)^{1/3} \right) = 0 \right\},
\]
and prove that $\Psi$ verifies $\Psi(0) = -\frac{\pi^2}{2}$, $\Psi(h) \sim_{h \to +\infty} \alpha_1 \frac{h^{2/3}}{2^{1/3}}$ and $\Psi(h)-\Psi(-h)=-h$ for all $h \in \R$.
\end{remark}

\begin{proof}[Proof of Lemma \ref{lem:existencePsi}]
For $h \in \R$ and $t \geq 0$, we write
\[
  \Psi_t(h) = \frac{1}{t} \log \sup_{x \in [0,1]} \E_x \left[ e^{ h \int_0^t B_s ds} ; B_s \in [0,1], s \in [0,t] \right].
\]
As $B_s \in [0,1]$, we have trivially $|\Psi_t(h)| \leq |h| < +\infty$. Let $0 \leq t_1 \leq t_2$ and $x \in [0,1]$, by the Markov property
\begin{align*}
  &\E_x \left[ e^{ h \int_0^{t_1 + t_2} B_s ds} ; B_s \in [0,1], s \in [0,t_1 + t_2] \right]\\
  &\qquad = \E_x\left[ e^{h \int_0^{t_1} B_s ds} \E_{B_{t_1}} \left[ e^{h\int_0^{t_2} B_s ds} ; B_s \in [0,1], s \in [0,t_2]\right] ; B_s \in [0,1], s \in [0,t_1] \right]\\
  &\qquad \leq e^{t_2 \Psi_{t_2}(h)} \E_x\left[ e^{h \int_0^{t_1} B_s ds}; B_s \in [0,1], s \in [0,t_1] \right]
  \leq e^{t_1 \Psi_{t_1}(h)} e^{t_2 \Psi_{t_2}(h)}.
\end{align*}
As a consequence, for all $h \in \R$, $(\Psi_t(h), t \geq 0)$ is a sub-additive function and
\[
  \lim_{t \to +\infty} \Psi_t(h) = \inf_{t \geq 0} \Psi_t(h) =: \Psi(h).
\]
In particular, for all $h \in \R$, we have $|\Psi(h)| \leq |h| < +\infty$.

We now prove that $\Psi$ is a convex function on $\R$, thus continuous. By the Hölder inequality, for all $\lambda \in [0,1]$, $(h_1,h_2) \in \R^2$, $x \in [0,1]$ and $t \geq 0$, we have
\begin{align*}
  &\E_x \left[ e^{ (\lambda h_1 + (1-\lambda)h_2) \int_0^t B_s ds} ; B_s \in [0,1], s \in [0,t] \right]\\
  \leq &\E_x \left[ \left(e^{\lambda h_1 \int_0^t B_s ds} \ind{B_s \in [0,1], s \in [0,t]} \right)^{\frac{1}{\lambda}}\right]^\lambda \E_x \left[ \left(e^{(1-\lambda) h_2 \int_0^t B_s ds} \ind{B_s \in [0,1], s \in [0,t]} \right)^{\frac{1}{1-\lambda}}\right]^{1-\lambda}\\
  \leq &e^{t \lambda \Psi_t(h_1)} e^{t(1-\lambda) \Psi_t(h_2)}.
\end{align*}
Consequently
\[
  \limsup_{t \to +\infty} \Psi_t(\lambda h_1 + (1-\lambda)h_2) \leq \lambda \limsup_{t \to + \infty} \Psi_t(h_1) + (1-\lambda) \limsup_{t \to +\infty} \Psi_t(h_2),
\]
which proves that $\Psi$ is convex, thus continuous.
\end{proof}

\begin{theorem}[A time-inhomogeneous Mogul'ski\u\i{} estimate]
\label{thm:mogtie}
Under \eqref{eqn:variance_rw} and \eqref{eqn:integrability_rw}, for any pair of continuous functions $f<g$ such that $f(0)<0<g(0)$ and any absolutely continuous function $h$ with a Riemann-integrable derivative $\dot{h}$, we have
\begin{multline*}
  \lim_{n \to +\infty} \frac{1}{n^{1/3}} \log \E\left[ \exp\left( \sum_{j=1}^{n} (h_{j/n}-h_{(j-1)/n}) S^{(n)}_{j} \right) ; \frac{S_j}{n^{1/3}} \in [f_{j/n}, g_{j/n}], j \leq n \right] \\
  = \int_0^1 \left( \dot{h}_s g_s + \frac{\sigma_s^2}{(g_s-f_s)^2}\Psi\left( \tfrac{(g_s-f_s)^3}{\sigma_s^2}\dot{h}_s \right) \right) ds.
\end{multline*}
\end{theorem}

The rest of the paper is organised as follows. Theorems \ref{thm:taktie} and \ref{thm:mogtie} are unified and proved in Section \ref{sec:rw}. These results are used in Section \ref{sec:path} to compute some branching random walk estimates, useful to bound the probability that there exists an individual that stays in a given path until time $n$. We study \eqref{eqn:existence_max} in Section \ref{sec:optimization}, proving in particular Proposition \ref{prop:regularity}. Using the particular structure of the optimal path, we prove Theorems \ref{thm:main} and \ref{thm:cmd} in Section \ref{sec:maxdis}.

\textbf{Acknowledgments.} I would like to thank Pascal Maillard, for introducing me to the time-inhomogeneous branching random walk topic, Ofer Zeitouni for his explanations on \cite{FaZ12} and Zhan Shi for help in all the stages of the research. I also thank the referees for their careful proofreading of this article and pointing out a mistake in one of the proofs. Finally, I wish to thank David Gontier and Cécile Huneau for their help with the PDE analysis in Appendix \ref{app:bm}.

\section{Random walk estimates}
\label{sec:rw}

We consider an array $(X_{n,k}, n \geq 1, k \leq n)$ of independent centred random variables, such that there exist $\sigma \in \calC([0,1], (0,+\infty))$ and $\mu\in (0,+\infty)$ verifying \eqref{eqn:variance_rw} and \eqref{eqn:integrability_rw}. We write $S^{(n)}_k = S^{(n)}_0 + \sum_{j=1}^k X_{n,j}$ for the time-inhomogeneous random walk of length $n$, with $\P_x(S^{(n)}_0=x)=1$. Let $\E_x$ be the expectation corresponding to the probability $\P_x$. Let $h$ be a continuous function on $[0,1]$ such that
\begin{equation}
\label{eqn:regularityh}
  h \text{ is absolutely continuous, with Riemann-integrable derivative } \dot{h}.
\end{equation}
The main result of this section is the computation of the asymptotic, as $n \to +\infty$, of the Laplace transform of the integral of $S^{(n)}$ with respect to $h$, on the event that $S^{(n)}$ stays in a given path, that is defined by \eqref{eqn:defIn}.

Let $f$ and $g$ be two continuous functions on $[0,1]$ such that $f<g$ and $f(0)<0<g(0)$, and $F$ and $G$ be two Riemann-integrable subsets of $[0,1]$ --i.e. verifying $\mathbf{1}_F$ and $\mathbf{1}_G$ are Riemann-integrable-- such that
\begin{equation}
  \label{eqn:paslineaire}
  \{t \in [0,1] : \dot{h}_t < 0 \} \subset F \quad \mathrm{and} \quad \{ t \in [0,1] : \dot{h}_t > 0 \} \subset G.
\end{equation}
Interval $F$ (respectively $G$) represent the set of times at which the barrier $f$ (resp. $g$) is put below (resp. above) the path of the time-inhomogeneous random walk. Consequently, \eqref{eqn:paslineaire} implies that when there is no barrier below, the Laplace exponent is non-negative, so that the random walk does not ``escape'' to $-\infty$ with high probability.

For $n \geq 1$, we introduce the $\frac{1}{n}^\mathrm{th}$ approximation of $F$ and $G$, defined by
\begin{equation}
  \label{eqn:deffnandgn}
  F_n = \left\{ 1\leq k \leq n : \left[\tfrac{k}{n}, \tfrac{k+1}{n}\right] \cap F \neq \emptyset\right\} \text{, } G_n = \left\{ 0 \leq k \leq n : \left[\tfrac{k}{n}, \tfrac{k+1}{n}\right] \cap G \neq \emptyset\right\}.
\end{equation}
The path followed by the random walk of length $n$ is defined, for $0 \leq j \leq n$, by
\begin{equation}
  \label{eqn:defIn}
  I_n(j) =
  \begin{cases}
    \left[f_{j/n}n^{1/3}, g_{j/n}n^{1/3}\right] & \text{if } j \in F_n \cap G_n,\\
    \left[f_{j/n},+\infty\right[ & \text{if } j \in F_n \cap G_n^c,\\
    \left]-\infty, g_{j/n}n^{1/3}\right] & \text{if } j \in F_n^c\cap G_n,\\
    \R & \text{otherwise.}
  \end{cases}
\end{equation}
The random walk $S^{(n)}$ follows the path $I^{(n)}$ if $ \geq f_{k/n} n^{1/3}$ at any time $k \in F_n$, and $S^{(n)}_k \leq g_{k/n} n^{1/3}$ at any time $k \in G_n$. Choosing $F$ and $G$ in an appropriate way, we obtain Theorem \ref{thm:taktie} --where $F=\emptyset$ and $G=[0,1]$-- and Theorem \ref{thm:mogtie} --where $F=G=[0,1]$.

We introduce the quantity
\begin{multline}
  \label{eqn:defineH}
  H^{F,G}_{f,g} = \int_0^1 \dot{h}_s g_s ds + \int_{F\cap G} \frac{\sigma_s^2}{(g_s - f_s)^2} \Psi\left( \tfrac{(g_s-f_s)^3}{\sigma_s^2} \dot{h}_s \right) ds\\
   + \int_{F^c \cap G} \frac{\alpha_1}{2^{1/3}} (\dot{h}_s \sigma_s)^{2/3} ds + \int_{G \cap F^c} \left( \dot{h}_s(f_s - g_s) + \frac{\alpha_1}{2^{1/3}} (-\dot{h}_s \sigma_s)^{2/3} \right) ds,
\end{multline}
where $\Psi$ is the function defined by \eqref{eqn:definepsi}. The first integral in this definition enables to ``center'' the path interval in a way that $g$ is replaced by $0$. The integral term over $F\cap G$ comes from the set of times in which the random walk is blocked in an interval of finite length as in Theorem \ref{thm:mogtie}, and the last two integral terms correspond to paths with only one bound, above or below the random walk respectively.

The rest of the section is devoted to the proof of the following result.
\begin{theorem}
\label{thm:general_rw}
Under the assumptions \eqref{eqn:variance_rw} and \eqref{eqn:integrability_rw}, for any continuous function $h$ satisfying \eqref{eqn:regularityh}, for any pair of continuous functions $f<g$ such that $f(0)<0<g(0)$, for any Riemann-integrable $F,G \subset [0,1]$ such that \eqref{eqn:paslineaire} holds, we have
\begin{equation}
  \label{eqn:limsup}
  \limsup_{n \to +\infty} \frac{1}{n^{1/3}} \sup_{x \in \R} \log \E_x\left[ e^{\sum_{j=1}^n (h_{(j+1)/n}-h_{j/n}) S^{(n)}_j}; S^{(n)}_j \in I^{(n)}_j, j \leq n \right] = H^{F,G}_{f,g}(1).
\end{equation}
Moreover, setting $\tilde{I}^{(n)}_j = I^{(n)}_j \cap [-n^{2/3},n^{2/3}]$, for all $f_1 < a < b < g_1$ we have
\begin{equation}
  \label{eqn:liminf}
  \liminf_{n \to +\infty} \frac{1}{n^{1/3}} \log \E_0\left[ e^{\sum_{j=1}^n (h_{(j+1)/n} - h_{j/n}) S^{(n)}_j} \ind{S^{(n)}_n \in [an^{1/3},bn^{1/3}]} ; S^{(n)}_j \in \tilde{I}^{(n)}_j, j \leq n \right] = H^{F,G}_{f,g}.
\end{equation}
\end{theorem}

\begin{remark}
Observe that when \eqref{eqn:paslineaire} does not hold, the correct rate of growth of the expectations in \eqref{eqn:limsup} and \eqref{eqn:liminf} is exponential, instead of the order $e^{O(n^{1/3})}$.
\end{remark}

To prove this theorem, we decompose the time interval $[0,n]$ into $A$ intervals, each of length $\frac{n}{A}$. On these smaller intervals, the functions $f$, $g$ and $\dot{h}$ can be approached by constants. These intervals are divide into $\frac{n^{1/3}}{tA}$ subintervals of length $tn^{2/3}$. On these subintervals, the time-inhomogeneous random walk can be approached by a Brownian motion, on which the quantities can be explicitly computed, using the Feynman-Kac formula. Letting $n$, $t$ then $A$ grow to $+\infty$, we conclude the proof of Theorem \ref{thm:general_rw}. We give in Section \ref{subsec:bm} the asymptotic of the area under a Brownian motion constrained to stay non-negative or in an interval, and use the Sakhanenko exponential inequality in Section \ref{subsec:rw} to quantify the approximation of a random walk by a Brownian motion, before proving Theorem \ref{thm:general_rw} in Section \ref{subsec:conclusion}.

\subsection{Brownian estimates through the Feynman-Kac formula}
\label{subsec:bm}

The asymptotic of the Laplace transform of the area under a Brownian motion, constrained to stay non-negative or in an interval, is proved in Appendix~\ref{app:bm}. In this section, $(B_t, t \geq 0)$ is a standard Brownian motion, which starts at position $x \in \R$ at time $0$ under the law $\P_x$. We give the main results that are used in the next section to compute similar quantities for time-inhomogeneous random walks. First, for a Brownian motion that stay non-negative:
\begin{lemma}
\label{lem:bmOnesided}
For all $h>0$, $0<a<b$ and $0<a'<b'$, we have
\begin{multline}
  \label{eqn:bmOnesided}
  \lim_{t \to +\infty} \frac{1}{t} \log \sup_{x \in \R} \E_x\left[ e^{-h \int_0^t B_s ds} ; B_s \geq 0, s \leq t \right]\\
  = \lim_{t \to +\infty} \frac{1}{t} \log \inf_{x \in [a,b]} \E_x\left[ e^{-h \int_0^t B_s ds} \ind{B_t \in [a',b']} ; B_s \geq 0, s \leq t  \right] = \frac{\alpha_1}{2^{1/3}}h^{2/3}.
\end{multline}
\end{lemma}

A similar estimate holds for a Brownian motion constrained to stay in the interval $[0,1]$:
\begin{lemma}
\label{lem:bmTwosided}
Let $B$ be a Brownian motion. For all $h \in \R$, $0<a<b<1$ and $0<a'<b'<1$, we have
\begin{multline}
  \label{eqn:bmTwosided}
  \lim_{t \to +\infty} \frac{1}{t} \sup_{x \in [0,1]} \log \E_x\left[ e^{-h \int_0^t B_s ds} ; B_s \in [0,1], s \leq t \right]\\ = 
  \lim_{t \to +\infty} \frac{1}{t} \inf_{x \in [a,b]} \log \E_x\left[ e^{-h \int_0^t B_s ds} \ind{B_t \in [a',b']} ; B_s \in [0,1], s \leq t \right] = \Psi(h).
\end{multline}
Moreover, for all $h>0$, we have
\begin{equation}
  \label{eqn:alternativeDefinition}
  \Psi(h) =  \frac{h^{2/3}}{2^{1/3}} \sup\left\{\lambda \leq 0 : \mathrm{Ai}\left( \lambda \right) \mathrm{Bi} \left( \lambda + (2h)^{1/3} \right) - \mathrm{Bi}\left( \lambda \right) \mathrm{Ai}\left( \lambda + (2h)^{1/3} \right) = 0 \right\}.
\end{equation}
We also have $\Psi(0)=-\frac{\pi^2}{2}$, $\lim_{h \to +\infty} \frac{\Psi(h)}{h^{2/3}} = \frac{\alpha_1}{2^{1/3}}$ and, for $h \in \R$, $\Psi(h) - \Psi(-h) = h$.
\end{lemma}

\subsection{From a Brownian motion to a random walk}
\label{subsec:rw}

We use the Sakhanenko exponential inequality to extend the Brownian estimates to similar quantities for time-inhomogeneous random walks. We obtain here the correct $n^{1/3}$ order, but non-optimal upper and lower bounds. These results are used in the next section to prove Theorem \ref{thm:general_rw}. The Sakhanenko exponential inequality links a time-inhomogeneous random walk with a Brownian motion, in a similar way that the KMT coupling links a classical random walk with a Brownian motion.
\begin{theorem}[Sakhanenko exponential inequality \cite{Sak84}]
\label{thm:sakhanenko}
Let $X=(X_1,\ldots X_n)$ be a sequence of independent centred random variables. We suppose there exists $\lambda>0$ such that for all $j \leq n$
\begin{equation}
  \label{eqn:integrabilityRwBis}
  \lambda \E\left( |X_j|^3 e^{\lambda |X_j|}\right) \leq \E\left(X_j^2\right).
\end{equation}
We can construct a sequence $\tilde{X}=(\tilde{X}_1,\ldots \tilde{X}_n)$ with the same law as $X$; and $Y$ a sequence of centred Gaussian random variables with the same covariance as $\tilde{X}$ such that for some universal constant $C_0$ and all $n \geq 1$
\[ \E\left[ \exp(C_0 \lambda \Delta_n)\right] \leq 1 + \lambda \sqrt{\sum_{j=1}^n \Var(X_j)}, \]
where $\Delta_n = \max_{j \leq n} \left| \sum_{k=1}^j \tilde{X}_k-Y_k \right|$.
\end{theorem}

Using this theorem, we couple a time-inhomogeneous random walk with a Brownian motion in such a way that they stay at distance $O(\log n)$ with high probability. Technically, to prove Theorem \ref{thm:general_rw}, we simply need a uniform control on $\P(\Delta_{n} \geq \epsilon n^{1/3})$. To obtain it, the polynomial Sakhanenko inequality is enough, that only impose a uniform bound on the third moment of the array of random variables instead of \eqref{eqn:integrability_rw}. However in the context of branching random walks, exponential integrability conditions are needed to guarantee the regularity of the optimal path (see Section \ref{subsec:heuristic}).

Let $(X_{n,k}, n \in \N, k \leq n)$ be a triangular array of independent centred random variables, such that there exists a continuous positive function $\sigma^2$ verifying
\begin{equation}
  \label{eqn:variance1}
  \forall n \in \N, k \leq n, \E\left[ X_{n,k}^2 \right] = \sigma^2_{k/n}.
\end{equation}
We set $\underline{\sigma}= \min_{t \in [0,1]} \sigma_t > 0$ and $\bar{\sigma} = \max_{t \in [0,1]} \sigma_t < +\infty$. We also assume that
\begin{equation}
  \label{eqn:integrabilityRw1}
  \exists \lambda > 0 : \sup_{n \in \N, k \leq n} \E\left( e^{\lambda |X_{n,k}|} \right) < + \infty.
\end{equation}
Observe that for all $\mu<\lambda$, there exists $C>0$ such that for all $x \geq 0$, $x^3 e^{\mu x} \leq C e^{\lambda x}$. Thus \eqref{eqn:integrabilityRw1} implies
\begin{equation}
  \label{eqn:integrabilityRwBis1}
  \exists \mu > 0 : \sup_{n \geq 1, k \leq n} \mu \E\left( |X_{n,j}|^3 e^{\mu |X_{n,j}|} \right) \leq \underline{\sigma}^2.
\end{equation}

In the first instance, we bound from above the asymptotic of the Laplace transform of the area under a time-inhomogeneous random walk.
\begin{lemma}
\label{lem:upperboundRw}
We assume \eqref{eqn:variance1} and \eqref{eqn:integrabilityRw1} are verified. For all $h>0$, we have
\begin{equation}
  \label{eqn:upperboundRwOnesided}
  \limsup_{n \to +\infty} \frac{1}{n^{1/3}} \log \sup_{x \in \R} \E_x\left[ e^{-\frac{h}{n} \sum_{j=0}^{n-1} S^{(n)}_j} ; S^{(n)}_j \geq 0, j \leq n \right] \leq \frac{\alpha_1}{2^{1/3}} (h \underline{\sigma})^{2/3}.
\end{equation}
For all $h \in \R$ and $r>0$, we have
\begin{equation}
  \label{eqn:upperboundRwTwosided}
  \limsup_{n \to +\infty} \frac{1}{n^{1/3}} \log \sup_{x \in \R} \E_x\left[ e^{-\frac{h}{n} \sum_{j=0}^{n-1} S^{(n)}_j} ; S^{(n)}_j \in [0, rn^{1/3}] \right]\\
  \leq \frac{\underline{\sigma}^2}{r^2} \Psi\left( \tfrac{r^3}{\underline{\sigma}^2} h \right).
\end{equation}
\end{lemma}

\begin{proof}
In this proof, we assume $h \geq 0$ --and $h>0$ if $r=+\infty$. The result for $h<0$ in \eqref{eqn:upperboundRwTwosided} can be deduced by symmetry and the formula $\Psi(h) - \Psi(-h)= -h$.

For all $r \in [0, +\infty)$, we write $f(r) = \frac{\underline{\sigma}^2}{r^2} \Psi\left( \tfrac{r^3}{\underline{\sigma}^2} h \right)$ and $f(+\infty) = \frac{\alpha_1}{2^{1/3}} (h \underline{\sigma})^{2/3}$. For all $x \in \R$, we use the convention $+\infty + x = x + \infty = +\infty$. By Lemmas \ref{lem:bmOnesided} and \ref{lem:bmTwosided}, for all $r \in [0,+\infty]$, we have
\begin{equation}
  \label{eqn:generalUpperbound}
  \limsup_{t \to +\infty} \frac{1}{t} \log \sup_{x \in \R} \E_x\left[ e^{-h \int_0^t B_{\underline{\sigma}^2s}ds} ; B_{\underline{\sigma}^2s} \in [0,r], s \leq t \right]
  \leq f(r),
\end{equation}
using the scaling property of the Brownian motion.

Let $A \in \N$ and $n \in \N$, we write $T = \ceil{A n^{2/3}}$ and $K = \floor{n/T}$. For all $k \leq K$, we write $m_k = k T$; applying the Markov property at time $m_K,m_{K-1},\ldots m_1$, we have
\begin{multline}
  \label{eqn:decompositionUpper}
  \sup_{x \geq 0} \E_x\left[ e^{-\frac{h}{n} \sum_{j=0}^{n-1} S^{(n)}_j} ; S^{(n)}_j \in  [0,rn^{1/3}], j \leq n \right]\\
  \leq \prod_{k=0}^{K-1} \sup_{x \in \R} \E_x\left[ e^{-\frac{h}{n} \sum_{j=0}^{n-1} S^{(n,k)}_j} ; S^{(n,k)}_j \in  [0,rn^{1/3}], j \leq T \right],
\end{multline}
where we write $S^{(n,k)}_j = S^{(n)}_0 + S^{(n)}_{m_k+j}- S^{(n)}_{m_k}$ for the time-inhomogeneous random walk starting at time $m_k$ and at position $x$ under $\P_x$. We now bound, uniformly in $k<K$, the quantity
\[
  E^{(n)}_k(r) = \sup_{x \in \R} \E_x\left[ e^{-\frac{h}{n} \sum_{j=0}^{T-1} S^{(n,k)}_j} ; S^{(n,k)}_j \in  [0,rn^{1/3}], j \leq T \right].
\]

Let $k < K$, we write $t^k_j = \sum_{i=kT+1}^{kT + j} \sigma^2_{j/n}$. We apply Theorem \ref{thm:sakhanenko}, by \eqref{eqn:variance1} and \eqref{eqn:integrabilityRw1}, there exist Brownian motions $B^{(k)}$ such that, denoting by $\tilde{S}^{(n,k)}$ a random walk with same law as $S^{(n,k)}$ and $\Delta_n^k = \max_{j \leq T} \left| B^{(k)}_{t^k_j} - \tilde{S}^{(n,k)}_j \right|$, there exists $\mu>0$ such that for all $\epsilon>0$, $n \geq 1$ and $k \leq K$,
\[
  \P\left(\Delta_n^k \geq \epsilon n^{1/3}\right) \leq e^{-C_0 \mu \epsilon n^{1/3}} \E\left( e^{C_0 \mu \Delta_n^k} \right) \leq e^{-C_0 \mu \epsilon n^{1/3}} \left(1+\mu \bar{\sigma} A^{1/2} n^{1/3}\right),
\]
where we used \eqref{eqn:integrabilityRw1} (thus \eqref{eqn:integrabilityRwBis1}) and the exponential Markov inequality. Note in particular that for all $\epsilon>0$, $\P(\Delta_n^k \geq \epsilon n^{1/3})$ converges to 0 as $n \to +\infty$, uniformly in $k \leq K$. As a consequence, for all $\epsilon>0$
\begin{align*}
  E^{(n)}_k(r)
  &= \sup_{x \in \R} \E_x\left[ e^{-\frac{h}{n} \sum_{j=0}^{T-1} \tilde{S}^{(n,k)}_j}  ; \tilde{S}^{(n,k)}_j \in  [0,rn^{1/3}], j \leq T \right]\\
  &\leq \sup_{x \in \R} \E_x\left[ e^{-\frac{h}{n} \sum_{j=0}^{T-1} \tilde{S}^{(n,k)}_j} \ind{\Delta_n^k \leq \epsilon n^{1/3}} ; \tilde{S}^{(n,k)}_j \in  [0,rn^{1/3}], j \leq T \right] + \P\left(\Delta_n^k \geq \epsilon n^{1/3}\right).
\end{align*}
Moreover,
\begin{align*}
  &\sup_{x \in \R} \E_x\left[ e^{-\frac{h}{n} \sum_{j=0}^{T-1} \tilde{S}^{(n,k)}_j} \ind{\Delta_n^k \leq \epsilon n^{1/3}} ; \tilde{S}^{(n,k)}_j \in  [0,rn^{1/3}], j \leq T \right]\\
  & \qquad \leq \sup_{x \in \R} \E_x\left[ e^{h\frac{T}{n}\Delta_n^k-\frac{h}{n} \sum_{j=0}^{T-1} B_{t^k_j}} \ind{\Delta_n^k \leq \epsilon n^{1/3}} ; B_{t^k_j} \in  [-\Delta_n^k,rn^{1/3}+\Delta_n^k], j \leq T \right]\\
  & \qquad \leq \sup_{x \in \R} \E_x\left[ e^{h \frac{T}{n} \epsilon n^{1/3}-\frac{h}{n} \sum_{j=0}^{T-1} B_{t^k_j}}; B_{t^k_j} \in [-\epsilon n^{1/3}, (r+\epsilon)n^{1/3}], j \leq T \right] \leq e^{3 h A \epsilon} \tilde{E}^{(n)}_k(r+2\epsilon),
\end{align*}
setting $\tilde{E}^{(n)}_k(r) = \sup_{x \in \R} \E_x\left[ e^{-\frac{h}{n} \sum_{j=0}^{T-1} B_{t^k_j}} ; B_{t^k_j} \in [0, r n^{1/3}], j \leq T \right]$ for all $r \in [0,+\infty]$. We set $\tau^k_j = n^{-2/3}t^k_j$; by the scaling property of the Brownian motion, we have
\[
  \tilde{E}^{(n)}_k(r) = \sup_{x \in \R} \E_x\left[ e^{-\frac{h}{n^{2/3}} \sum_{j=0}^{T-1} B_{\tau^k_j}} ; B_{\tau^k_j} \in [0,r], j \leq T \right].
\]

We now replace the sum in $\tilde{E}$ by an integral: we set
\[
  \omega_{n,A} = \sup_{|t-s| \leq 2 A n^{-1/3}} \left| \sigma^2_t - \sigma^2_s \right| \quad \mathrm{and} \quad \Omega_{n,A} = \sup_{\substack{s,t \leq 2\bar{\sigma}^2 A + \omega_{n,A}\\ |t-s| \leq 2 A \omega_{n,A} + \bar{\sigma}^2 n^{-1/3}}} \left| B_t - B_s\right|.
\]
For all $k < K$ and $j \leq T$, we have
\[
  \left| \tau^k_j - j \sigma^2_{kT/n} n^{-2/3} \right| \leq n^{-2/3} \sum_{i=m_k+1}^{m_k+j} \left| \sigma^2_{i/n} - \sigma^2_{kT/n} \right| \leq 2 A \omega_{n,A},
\]
and $\sup_{s \in [j\underline{\sigma}^2/n,(j+1)\underline{\sigma}^2/n}\left| B_s - B_{t^k_j} \right| \leq \Omega_{n,A}$. As a consequence, for all $\epsilon>0$, we obtain
\begin{align*}
  \tilde{E}^{(n)}_k(r)
  &\leq \sup_{x \in \R} \E_x\left[ e^{-\frac{h}{n^{2/3}} \sum_{j=0}^{T-1} B_{\tau^k_j}} \ind{\Omega_{n,A} \leq \epsilon} ; B_{\tau^k_j} \in [0,r], j \leq T \right] + \P(\Omega_{n,A} \geq \epsilon)\\
  &\leq e^{3hA \epsilon} \sup_{x \in \R} \E_x\left[ e^{-h \int_0^A B_{\underline{\sigma}^2 s} ds} ; B_{\underline{\sigma}^2 s} \in [0,(r+2\epsilon)], s \leq A \right] + \P(\Omega_{n,A} \geq \epsilon).
\end{align*}
We set $\bar{E}^{A}(r) = \sup_{x \in \R} \E_x\left[ e^{-h \int_0^A B_{\underline{\sigma}^2 s} ds} ; B_{\underline{\sigma}^2 s} \in [0,r], s \leq A \right]$. As $B$ is continuous, we have $\lim_{n \to +\infty} \P(\Omega_{n,A} \geq \epsilon) = 0$ uniformly in $k<K$. Therefore \eqref{eqn:decompositionUpper} leads to
\begin{align*}
  &\limsup_{n \to +\infty} \frac{1}{n^{1/3}} \log \sup_{x \geq 0} \E_x\left[ e^{-\frac{h}{n} \sum_{j=0}^{n-1} S^{(n)}_j} ; S^{(n)}_j \in  [0,rn^{1/3}], j \leq n \right]\\
  &\qquad \qquad \leq \limsup_{n \to +\infty} \frac{K}{n^{1/3}} \max_{k \leq K} \log E^{(n)}_k(r)\\
  &\qquad \qquad \leq \frac{1}{A} \limsup_{n \to +\infty} \left[ 3 h A \epsilon + \max_{k \leq K} \log \left( \tilde{E}^{(n)}_k(r+2\epsilon) + \P\left(\Delta_n^k \geq \epsilon n^{1/3}\right) \right) \right]\\
  &\qquad \qquad \leq 6 h \epsilon + \limsup_{n \to +\infty} \log \left[\bar{E}^A(r+4\epsilon) + \P(\Omega_{n,A} \geq \epsilon) + \max_{k \leq K} \P\left(\Delta_n^k \geq \epsilon n^{1/3}\right) \right]\\
  &\leq 6h \epsilon + \frac{1}{A} \log \bar{E}^A(r+4\epsilon).
\end{align*}
We now use \eqref{eqn:generalUpperbound}, letting $A \to +\infty$, and thereby letting $\epsilon \to 0$, this yields
\[
  \limsup_{n \to +\infty} \frac{1}{n^{1/3}} \log \sup_{x \geq 0} \E_x\left[ e^{-\frac{h}{n} \sum_{j=0}^{n-1} S^{(n)}_j} ; S^{(n)}_j \in  [0,rn^{1/3}], j \leq n \right] \leq f(r),
\]
which ends the proof.
\end{proof}

Next, we derive lower bounds with similar computations. We set $I^{(n)}_{a,b} = [an^{1/3},bn^{1/3}]$.
\begin{lemma}
\label{lem:lowerboundRw}
We assume \eqref{eqn:variance1} and \eqref{eqn:integrabilityRw1}. For all $h>0$, $0<a<b$ and $0<a'<b'$, we have
\begin{equation}
  \label{eqn:lowerboundRwOnesided}
  \liminf_{n \to +\infty} \frac{1}{n^{1/3}} \log \inf_{x \in I^{(n)}_{a,b}} \E_x\left[ e^{-\frac{h}{n} \sum_{j=0}^{n-1} S^{(n)}_j} \ind{S_n \in I^{(n)}_{a',b'}} ; S^{(n)}_j \geq 0, j \leq n \right] \geq \frac{\alpha_1}{2^{1/3}} (h \bar{\sigma})^{2/3},
\end{equation}
and for all $h \in \R$, $r>0$, $0<a<b<r$ and $0<a'<b'<r$, we have
\begin{equation}
  \label{eqn:lowerboundRwTwosided}
  \liminf_{n \to +\infty} \frac{1}{n^{1/3}} \log \inf_{x \in I^{(n)}_{a,b}} \E_x\left[ e^{-\frac{h}{n} \sum_{j=0}^{n-1} S^{(n)}_j} \ind{S_n \in I^{(n)}_{a',b'}} ; S^{(n)}_j \in I^{(n)}_{0,r}, j \leq n \right]
  \geq \frac{\bar{\sigma}^2}{r^2} \Psi\left( \tfrac{r^3}{\bar{\sigma}^2} h \right).
\end{equation}
\end{lemma}

\begin{proof}
We once again assume $h \geq 0$; as if $h<0$ we can deduce \eqref{eqn:lowerboundRwTwosided} by symmetry and the formula $\Psi(h) - \Psi(-h)= h$. We write, for all $r \in [0, +\infty)$, $f(r) = \frac{\bar{\sigma}^2}{r^2} \Psi\left( \tfrac{r^3}{\bar{\sigma}^2} h \right)$ and $f(+\infty) = \frac{\alpha_1}{2^{1/3}} (h \bar{\sigma})^{2/3}$. By Lemmas \ref{lem:asymptoticOnesided} and \ref{lem:asymptoticTwosided}, for all $r \in [0,+\infty]$, $0<a<b<r$ and $0<a'<b'<r$, we have
\begin{equation}
  \label{eqn:generalLowerbound}
  \liminf_{t \to +\infty} \frac{1}{t} \log \inf_{x \in [a,b]} \E_x\left[ e^{-h \int_0^t B_{\bar{\sigma}^2s}ds} \ind{B_t \in [a',b']} ; B_s \in [0,r], s \leq t \right]
  \geq f(r).
\end{equation}

We choose $u \in (a',b')$ and $\delta > 0$ such that $(u-3\delta,u+3\delta) \subset (a',b')$, and we introduce $J^{(n)}_{\delta} = I^{(n)}(u-\delta, u+\delta)$. We decompose again $[0,n]$ into subintervals of length of order $n^{2/3}$. Let $A \in \N$ and $n \in \N$, we write $T = \floor{A n^{2/3}}$ and $K = \floor{n/T}$. For all $k \leq K$, we set again $m_k = kT$, for all $r \in [0,+\infty]$, applying the Markov property at times $m_K,m_{K-1},\ldots m_1$ leads to
\begin{multline}
  \label{eqn:decompositionLower}
   \inf_{x \in I^{(n)}_{a,b}} \E_x\left[ e^{-\frac{h}{n} \sum_{j=0}^{n-1} S^{(n)}_j} \ind{S_n \in I^{(n)}_{a',b'}} ; S^{(n)}_j \in I^{(n)}_{0,r}, j \leq n \right]\\
   \geq  \inf_{x \in I^{(n)}_{a,b}} \E_x\left[ e^{-\frac{h}{n} \sum_{j=0}^{T-1} S^{(n)}_j} \ind{S^{(n)}_T \in J^{(n)}_\delta} ; S^{(n)}_j \in I^{(n)}_{0,r}, j \leq T \right] \qquad \qquad \qquad\\
   \qquad \quad\times \prod_{k=1}^{K-1} \inf_{x \in J^{(n)}_\delta} \E_x\left[ e^{-\frac{h}{n} \sum_{j=0}^{T-1} S^{(n,k)}_j} \ind{S^{(n,k)}_T \in J^{(n)}_\delta} ; S^{(n,k)}_j \in I^{(n)}_{0,r}, j \leq T \right]\\
   \times\inf_{x \in J^{(n)}_\delta} \E_x\left[ e^{-\frac{h}{n} \sum_{j=0}^{n-KT} S^{(n,K)}_j} ;  S^{(n,k)}_j \in I^{(n)}_{a',b'}, j \leq n-KT \right],
\end{multline}
where $S^{(n,k)}_j = S^{(n)}_0 + S^{(n)}_{m_k+j}- S^{(n)}_{m_k}$. Let $0<a<b<r$ and $0<a'<b'<r$, we set $\epsilon>0$ such that $a>8\epsilon$, $r-b>8\epsilon$ and $b'-a'>8\epsilon$. We bound uniformly in $k$ the quantity
\[
  E^{(n)}_k(r) =  \inf_{x \in I^{(n)}_{a,b}} \E_x\left[ e^{-\frac{h}{n} \sum_{j=0}^{T-1} S^{(n,k)}_j} \ind{S^{(n,k)}_T \in I^{(n)}_{a',b'}} ; S^{(n,k)}_j \in I^{(n)}_{0,r}, j \leq T \right].
\]

To do so, we set once again, for $k<K$, $t^k_j = \sum_{i=kT+1}^{kT + j} \sigma^2_{j/n}$. By Theorem \ref{thm:sakhanenko}, we introduce a Brownian motion $B$ such that, denoting by $\tilde{S}^{(n,k)}$ a random walk with the same law as $S^{(n,k)}$ and setting $\Delta_n^k = \max_{j \leq T} \left| B_{t^k_j} - \tilde{S}^{(n,k)}_j \right|$, for all $\epsilon>0$, by \eqref{eqn:integrabilityRwBis1} and the exponential Markov inequality we get
\[
  \sup_{k \leq K} \P\left(\Delta_n^k \geq \epsilon n^{1/3}\right) \leq e^{-C_0 \mu \epsilon n^{1/3}} \left(1+\mu \bar{\sigma} A^{1/2} n^{1/3}\right),
\]
which converges to $0$, uniformly in $k$, as $n \to +\infty$. As a consequence, for all $\epsilon>0$ and $k < K$,
\begin{align*}
  E^{(n)}_k(r)
  &= \inf_{x \in I^{(n)}_{a,b}} \E_x\left[ e^{-\frac{h}{n} \sum_{j=0}^{T-1} \tilde{S}^{(n,k)}_j} \ind{\tilde{S}^{(n,k)}_T \in I^{(n)}_{a',b'}} ; \tilde{S}^{(n,k)}_j \in I^{(n)}_{0,r}, j \leq T \right]\\
  &\geq \inf_{x \in I^{(n)}_{a,b}} \E_x\left[ e^{-\frac{h}{n} \sum_{j=0}^{T-1} \tilde{S}^{(n,k)}_j} \ind{\tilde{S}^{(n,k)}_T \in I^{(n)}_{a',b'}} \ind{\Delta_n^k \leq \epsilon n^{1/3}} ; \tilde{S}^{(n,k)}_j \in I^{(n)}_{0,r}, j \leq T \right]\\
  &\geq  \inf_{x \in  I^{(n)}_{a-\epsilon,b+\epsilon}}  e^{-3hA \epsilon}\E_x\left[ e^{-\frac{h}{n} \sum_{j=0}^{T-1} B_{t^k_j}} \ind{B_{t^k_T} \in  I^{(n)}_{a'+\epsilon,b'-\epsilon}} \ind{\Delta_n^k \leq \epsilon n^{1/3}} ; B_{t^k_j} \in  I^{(n)}_{\epsilon,r-\epsilon}, j \leq T \right]\\
   &\geq e^{-3 h A \epsilon} \left(\tilde{E}^{(n)}_k(r-2\epsilon) - \P(\Delta_n^k \geq \epsilon n^{1/3}) \right),
\end{align*}
where we set
\[
  \tilde{E}^{(n)}_k(r) = \inf_{x \in [a-2\epsilon,b+2\epsilon]} \E_x\left[ e^{-\frac{h}{n^{2/3}} \sum_{j=0}^{T-1} B_{\tau^k_j}} \ind{B_{\tau^k_T} \in [a'+2\epsilon,b'-2\epsilon]} ; B_s \in  [0,r], s \leq \tau^k_T \right],
\]
and $\tau^k_j = t^k_jn^{-2/3}$. We also set
\[
  \omega_{n,A} = \sup_{|t-s| \leq 2 A n^{-1/3}} \left| \sigma^2_t - \sigma^2_s \right| \quad \mathrm{and} \quad \Omega_{n,A} = \sup_{\substack{s,t \leq 2\bar{\sigma}^2 A + \omega_{n,A}\\ |t-s| \leq 2 A \omega_{n,A} + \bar{\sigma}^2 n^{-1/3}}} \left| B_t - B_s\right|,
\]
so that for all $k < K$ and $j \leq T$, we have
\[
  \left| \tau^k_j - j \sigma^2_{kT/n} n^{-2/3} \right| \leq n^{-2/3} \sum_{i=m_k+1}^{m_k+j} \left| \sigma^2_{i/n} - \sigma^2_{kT/n} \right| \leq 2 A \omega_{n,A},
\]
and $\sup_{s \in [\underline{\sigma}^2\frac{j}{n},\underline{\sigma}^2\frac{j+1}{n}]}\left| B_s - B_{t^k_j} \right| \leq \Omega_{n,A}$. As a consequence,
\begin{multline*}
  \tilde{E}^{(n)}_k(r) e^{3hA\epsilon} \geq \\
  \inf_{x \in [a-4\epsilon, b + 4\epsilon]} \E_x\left[ e^{-h \int_0^A B_{\underline{\sigma}^2 s} ds} \ind{B_{\underline{\sigma}^2 A} \in [a'+4\epsilon, b'-4\epsilon]} ; B_{\underline{\sigma}^2 s} \in [0,r-2\epsilon], s \leq A \right]
   - \P(\Omega_{n,A} \geq \epsilon).
\end{multline*}

This last estimate gives a lower bound for $E^{(n)}_k(r)$ which is uniform in $k \leq K$. As a consequence, \eqref{eqn:decompositionLower} yields 
\begin{multline*}
  \liminf_{n \to +\infty} \frac{1}{n^{1/3}} \log \inf_{x \in I^{(n)}_{a,b}} \E_x\left[ e^{-\frac{h}{n} \sum_{j=0}^{n-1} S^{(n)}_j} \ind{S_n \in I^{(n)}_{a',b'}} ; S^{(n)}_j \in I^{(n)}_{0,r}, j \leq n \right] \geq \\
 -6h\epsilon + \frac{1}{A} \log \inf_{x \in [a-4\epsilon, b + 4 \epsilon]} \E_x\left[ e^{-h \int_0^A B_{\underline{\sigma}^2 s} ds} \ind{B_{\underline{\sigma}^2 A} \in [a'+4\epsilon, b'-4\epsilon]} ; B_{\underline{\sigma}^2 s} \in [0,r-4\epsilon], s \leq A \right].
\end{multline*}
Letting $A \to +\infty$, then $\epsilon \to 0$ leads to
\[
  \liminf_{n \to +\infty} \frac{1}{n^{1/3}} \log \inf_{x \in I^{(n)}_{a,b}} \E_x\left[ e^{-\frac{h}{n} \sum_{j=0}^{n-1} S^{(n)}_j} \ind{S_n \in I^{(n)}_{a',b'}} ; S^{(n)}_j \in I^{(n)}_{0,r}, j \leq n \right] \geq  f(r),
\]
which ends the proof.
\end{proof}

\subsection{Proof of Theorem \ref{thm:general_rw}}
\label{subsec:conclusion}

We prove Theorem \ref{thm:general_rw} by decomposing $[0,n]$ into $A$ intervals of length $n/A$, and apply Lemmas~\ref{lem:upperboundRw} and \ref{lem:lowerboundRw}.

\begin{proof}[Proof of Theorem \ref{thm:general_rw}]
We set $n \in \N$ and $A \in \N$. For all $0 \leq a \leq A$, we write $m_a = \floor{na/A}$, and $d_a = m_{a+1}-m_a$.

\paragraph*{Upper bound in \eqref{eqn:limsup}.}
We apply the Markov property at times $m_{A-1}, m_{A-2}, \ldots m_1$, to see that
\begin{multline*}
  \sup_{x \in I^{(n)}_0}  \E_x\left[ e^{\sum_{j=0}^{n-1} (h_{(j+1)/n}-h_{j/n}) S^{(n)}_j}; S^{(n)}_j \in I^{(n)}_j, j \leq n \right]\\
  \leq \prod_{a=0}^{A-1} \underbrace{\sup_{x \in I^{(n)}_{m_a}} \E_x\left[ e^{\sum_{j=0}^{d_a-1} (h_{(m_a+j+1)/n}-h_{(m_a+j)/n}) S^{(n,a)}_j}; S^{(n,a)}_j \in I^{(n)}_{m_a+j}, j \leq d_a \right]}_{R^{(n)}_{a,A}},
\end{multline*}
where $S^{(n,a)}_j = S^{(n)}_0 + S^{(n)}_{m_a+j} - S^{(n)}_{m_a}$ is the time-inhomogeneous random walk starting at time $m_a$ and position $x$. Letting $n \to +\infty$, this yields
\begin{multline}
  \label{eqn:upperDivision}
  \limsup_{n \to +\infty} \sup_{x \in \R} \frac{1}{n^{1/3}} \log \E_x\left[ e^{\sum_{j=1}^n (h_{(j+1)/n}-h_{j/n}) S^{(n)}_j}; S^{(n)}_j \in I^{(n)}_{j}, j \leq n \right]\\
  \leq \sum_{a=0}^{A-1} \limsup_{n \to +\infty} \frac{1}{n^{1/3}} \log R^{(n)}_{a,A}.
\end{multline}

To bound $R^{(n)}_{a,A}$, we replace functions $f,g$ and $\dot{h}$ by constants. We set, for all $A \in \N$ and $a \leq A$,
\begin{multline*}
  \bar{h}_{a,A} = \sup_{t \in [\frac{a-1}{A},\frac{a+2}{A}]} \dot{h}_t, \quad \underline{h}_{a,A} = \inf_{t \in [\frac{a-1}{A}, \frac{a+2}{A}]} \dot{h}_t,\\
  g_{a,A} = \sup_{t \in [\frac{a-1}{A},\frac{a+2}{A}]} g_t, \quad f_{a,A} = \inf_{t \in [\frac{a-1}{A}, \frac{a+2}{A}]} f_t \quad \mathrm{and} \quad \sigma_{a,A} = \inf_{t \in [\frac{a-1}{A},\frac{a+2}{A}]} \sigma_s.
\end{multline*}
For any $n \in \N$ and $k \leq n$, by \eqref{eqn:paslineaire}, if $h_{(k+1)/n}>h_{k/n}$, then $k \in G_n$, and if $h_{(k+1)/n}<h_{k/n}$, then $k \in F_n$. Consequently, for all $x \in I^{(n)}_k$,
\begin{equation}
  \label{eqn:paslineaireConsequence}
  (h_{(k+1)/n} - h_{k/n}) x \leq (h_{(k+1)/n}-h_{k/n})_+ g_{k/n} n^{1/3} - (h_{k/n} - h_{(k+1)/n})_+ f_{k/n} n^{1/3}.
\end{equation}
We bound from above $R^{(n)}_{a,A}$ in four different cases.

First, for all $a<A$, by \eqref{eqn:paslineaireConsequence}, we have
\[
  R^{(n)}_{a,A} \leq \exp\left( \sum_{j=m_a}^{m_{a+1}-1} (h_{(j+1)/n}-h_{j/n})_+ g_{j/n} n^{1/3} - (h_{j/n} - h_{(j+1)/n})_+ f_{k/n} n^{1/3} \right),
\]
and thus,
\begin{equation}
  \label{eqn:free}
  \limsup_{n \to +\infty} \frac{1}{n^{1/3}} \log R^{(n)}_{a,A} \leq \int_{a/A}^{(a+1)/A} (\dot{h}_s)_+ g_s - (\dot{h}_s)_- f_sds = \int_{a/A}^{(a+1)/A} \dot{h}_s g_s - (\dot{h}_s)_- (f_s-g_s) ds.
\end{equation}

This crude estimate can be improved as follows. If $\underline{h}_{a,A} > 0$, then $[\frac{a}{A},\frac{a+1}{A}] \subset G$ and the upper bound $g_{k/n}n^{1/3}$ of the path is present at all times $k \in [m_a, m_{a+1}]$. As a consequence, \eqref{eqn:paslineaireConsequence} becomes
\begin{equation}
  \label{eqn:paslineaireConsequenceUpperfrontier}
  \forall k \in [m_a,m_{a+1}), \sup_{x \in I^{(n)}_k} (h_{(k+1)/n}-h_{k/n})x \leq (h_{(k+1)/n}-h_{k/n})g_{a,A}n^{1/3} + \frac{1}{n} \underline{h}_{a,A}(x - g_{a,A}n^{1/3}).
\end{equation}
We have
\begin{align*}
  R^{(n)}_{a,A} &= \sup_{x \in I^{(n)}_{m_a}} \E_x\left[ e^{\sum_{j=0}^{d_a-1} (h_{(m_a+j+1)/n}-h_{(m_a+j)/n}) S^{(n),a}_j}; S^{(n),a}_j \in I^{(n)}_{m_a+j}, j \leq d_a \right]\\
  &\leq e^{\sum_{j=m_a}^{m_{a+1}-1} (h_{(j+1)/n}-h_{j/n}) g_{a,A} n^{1/3}}\\
  &\qquad\qquad \times \sup_{x \in I^{(n)}_{m_a}} \E_x \left[ e^{\frac{1}{n} \sum_{j=0}^{d_a-1} \underline{h}_{a,A} (S^{(n),a}_j - g_{a,A}n^{1/3})}; S^{(n),a}_j \in I^{(n)}_{m_a+j}, j \leq d_a \right]\\
  &\leq e^{(h_{m_{a+1}/n} - h_{m_a/n})g_{a,A}n^{1/3}} \sup_{x \leq 0} \E_x\left[ e^{\frac{1}{n} \sum_{j=0}^{d_a-1} \underline{h}_{a,A} S^{(n),a}_j} ; S^{(n),a}_{m_a+j}  \leq 0, j \leq d_a\right].
\end{align*}
Letting $n \to +\infty$, we have
\begin{multline*}
  \limsup_{n \to +\infty} \frac{1}{n^{1/3}} \log R^{(n)}_{a,A} \leq (h_{(a+1)/A}-h_{a/A}) g_{a,A} \\
  + \limsup_{n \to +\infty} \frac{1}{n^{1/3}} \log \sup_{x \leq 0} \E_x\left[ e^{\frac{\underline{h}_{a,A}}{n} \sum_{j=0}^{d_a-1} S^{(n),a}_j} ; S^{(n),a}_j \leq 0, j \leq d_a\right].
\end{multline*}
As $d_a \sim_{n \to +\infty} n/A$, by \eqref{eqn:upperboundRwOnesided},
\begin{multline*}
  \limsup_{n \to +\infty} \frac{1}{n^{1/3}} \log \sup_{x \leq 0} \E_x\left[ e^{\frac{\underline{h}_{a,A}}{A (d_a+1)} \sum_{j=0}^{d_a-1} S^{(n),a}_j} ; S^{(n),a}_j \leq 0, j \leq d_a \right]\\
  \leq \frac{1}{A^{1/3}}\frac{\alpha_1}{2^{1/3}} \left(\tfrac{1}{A}\underline{h}_{a,A} \sigma_{a,A} \right)^{2/3} = \frac{\alpha_1}{2^{1/3}A} \left(\underline{h}_{a,A} \sigma_{a,A} \right)^{2/3}.
\end{multline*}
We conclude that
\begin{equation}
  \label{eqn:premierOnesided}
  \limsup_{n \to +\infty} \frac{1}{n^{1/3}} \log R^{(n)}_{a,A} \leq (h_{(a+1)/A}-h_{a/A}) g_{a,A} + \frac{\alpha_1}{2^{1/3}A} \left(\underline{h}_{a,A} \sigma_{a,A} \right)^{2/3}.
\end{equation}
By symmetry, if $\bar{h}_{a,A}<0$, then $[\frac{a}{A},\frac{a+1}{A}] \subset F$, $h_{(k+1)/n}<h_{k/n}$ and the lower bound of the path is present at all time, which leads to
\begin{equation}
  \label{eqn:secondOnesided}
  \limsup_{n \to +\infty} \frac{1}{n^{1/3}} \log R^{(n)}_{a,A} \leq (h_{(a+1)/A}-h_{a/A}) f_{a,A} + \frac{\alpha_1}{2^{1/3}A} \left(-\bar{h}_{a,A} \sigma_{a,A} \right)^{2/3}.
\end{equation}

Fourth and the smallest upper bound; if $[\frac{a}{A},\frac{a+1}{A}] \subset F\cap G$, then both bounds of the path are present at any time in $[m_a,m_{a+1}]$, and, by \eqref{eqn:paslineaireConsequenceUpperfrontier}, setting $r_{a,A} = g_{a,A}-f_{a,A}$,
\begin{multline*}
  R^{(n)}_{a,A} \leq e^{(h_{m_{a+1}/n} - h_{m_a/n})g_{a,A}n^{1/3}} \\
  \times \sup_{x \in [r_{a,A} n^{1/3},0]} \E_x\left[ e^{\frac{1}{n} \sum_{j=0}^{d_a-1} \underline{h}_{a,A} S^{(n),a}_j} ; S^{(n),a}_j  \in [-r_{a,A} n^{1/3},0], j \leq d_a\right].
\end{multline*}
We conclude that
\begin{multline*}
    \limsup_{n \to +\infty} \frac{1}{n^{1/3}} \log R^{(n)}_{a,A} \leq (h_{(a+1)/A}-h_{a/A}) g_{a,A} \\
  + \limsup_{n \to +\infty} \frac{1}{n^{1/3}} \log \sup_{x \in  [-r_{a,A} n^{1/3},0]} \E_x\left[ e^{\frac{\underline{h}_{a,A}}{n} \sum_{j=0}^{d_a-1} S^{(n),a}_j} ; S^{(n),a}_j \in  [-r_{a,A} n^{1/3},0], j \leq d_a\right].
\end{multline*}
Applying then \eqref{eqn:upperboundRwTwosided}, this yields
\begin{multline*}
  \limsup_{n \to +\infty} \frac{1}{n^{1/3}} \log \sup_{x \in  [-r_{a,A} n^{1/3},0]} \E_x\left[ e^{\frac{\underline{h}_{a,A}}{n} \sum_{j=0}^{d_a-1} S^{(n),a}_j} ; S^{(n),a}_j \in  [-r_{a,A} n^{1/3},0], j \leq d_a\right]\\
  \leq \frac{\sigma_{a,A}^2}{A r_{a,A}^2} \Psi\left( \frac{r_{a,A}^3}{\sigma_{a,A}^2} \underline{h}_{a,A} \right),
\end{multline*}
which yields
\begin{equation}
  \label{eqn:twoSided}
  \limsup_{n \to +\infty} \frac{1}{n^{1/3}} \log R^{(n)}_{a,A} \leq (h_{(a+1)/A}-h_{a/A}) g_{a,A} + \frac{\sigma_{a,A}^2}{A (g_{a,A} - f_{a,A})^2} \Psi\left( \frac{(g_{a,A}- f_{a,A})^3}{\sigma_{a,A}^2} \underline{h}_{a,A} \right).
\end{equation}
We now let $A$ grow to $+\infty$ in \eqref{eqn:upperDivision}. By Riemann-integrability of $F,G$ and $\dot{h}$, we have
\begin{multline}
  \label{eqn:twoSidedLimit}
  \limsup_{A \to +\infty} \sum_{\substack{0\leq a < A \\ [\frac{a}{A},\frac{a+1}{A}] \subset F\cap G}} \left[ (h_{(a+1)/A}-h_{a/A}) g_{a,A} + \frac{\sigma_{a,A}^2}{A (g_{a,A} - f_{a,A})^2} \Psi\left( \frac{(g_{a,A}- f_{a,A})^3}{\sigma_{a,A}^2} \underline{h}_{a,A} \right)\right]\\ \leq \int_{F\cap G} \dot{h}_s g_s + \frac{\sigma^2_s}{(g_s-f_s)^2} \Psi\left( \frac{(g_s-f_s)^3}{\sigma_s^2} \dot{h}_s \right) ds.
\end{multline}
Similarly, using the fact that $\dot{h}$ is non-negative on $F^c$, and non-positive on $G^c$, \eqref{eqn:premierOnesided} and \eqref{eqn:secondOnesided} lead respectively to
\begin{multline}
  \label{eqn:premierOnesidedLimit}
  \limsup_{A \to +\infty} \sum_{\substack{0 \leq a < A \\ \underline{h}_{a,A}>0, [\frac{a}{A},\frac{a+1}{A}] \not\subset F\cap G}} \left[ (h_{(a+1)/A}-h_{a/A}) g_{a,A} + \frac{\alpha_1}{A 2^{1/3}} \left( \underline{h}_{a,A} \sigma_{a,A} \right)^{2/3}\right]\\
  \leq \int_{F^c \cap G} \dot{h}_s g_s + \frac{\alpha_1}{2^{1/3}} \left( \dot{h}_s \sigma_s\right)^{2/3} ds , 
\end{multline}
and to
\begin{multline}
  \label{eqn:secondOnesidedLimit}
  \limsup_{A \to +\infty} \sum_{\substack{0 \leq a < A \\ \bar{h}_{a,A}<0, [\frac{a}{A},\frac{a+1}{A}] \not\subset F\cap G}} \left[(h_{(a+1)/A}-h_{a/A}) f_{a,A} + \frac{\alpha_1}{A 2^{1/3}} \left( - \bar{h}_{a,A} \sigma_{a,A} \right)^{2/3}\right]\\
  \leq \int_{F^c \cap G} \dot{h}_s g_s + \dot{h}_s (f_s - g_s) + \frac{\alpha_1}{2^{1/3}} \left( -\dot{h}_s \sigma_s\right)^{2/3} ds.
\end{multline}
Finally, by \eqref{eqn:free}, \eqref{eqn:twoSidedLimit}, \eqref{eqn:premierOnesidedLimit} and \eqref{eqn:secondOnesidedLimit}, letting $A \to +\infty$, \eqref{eqn:upperDivision} yields
\[
  \limsup_{n \to +\infty} \sup_{x \in \R} \frac{1}{n^{1/3}} \log \E_x\left[ e^{\sum_{j=1}^n (h_{(j+1)/n}-h_{j/n}) S^{(n)}_j}; S^{(n)}_j \in I^{(n)}_{j}, j \leq n \right] \leq H^{F,G}_{f,g}.
\]

\paragraph*{Lower bound in \eqref{eqn:liminf}}
We now take care of the lower bound. We start by fixing $H>0$, and we write
\[
  I^{(n,H)}_{j} = I^{(n)}_j \cap [-Hn^{1/3}, Hn^{1/3}],
\]
letting $H$ grow to $+\infty$ at the end of the proof. We only need \eqref{eqn:lowerboundRwTwosided} here.

We choose $k \in \calC([0,1])$ a continuous function such that $k_0=0$ and $k_1 \in (a,b)$ and $\epsilon>0$ such that for all $t \in [0,1]$, $k_t \in [f_t + 4\epsilon, g_t-4\epsilon]$ and $k_1 \in [a+4\epsilon, b-4\epsilon]$. We set
\[
  J^{(n)}_a = \left[ (k_{a/A} - \epsilon)n^{1/3}, (k_{a/A} + \epsilon)n^{1/3}\right].
\]
We apply the Markov property at times $m_{A-1},\ldots m_1$, only considering random walk paths that are in interval $J^{(n)}_a$ at any time $m_a$. For all $n \geq 1$ large enough, we have
\begin{align*}
  \E & \left[ e^{\sum_{j=0}^{n-1} (h_{(j+1)/n}-h_{j/n}) S^{(n)}_j} \ind{\frac{S^{(n)}_n}{n^{1/3}} \in [a',b']} ; S^{(n)}_j \in \tilde{I}^{(n)}_j, j \leq n \right]\\
  &\geq \prod_{a=0}^{A-1} \inf_{x \in I^{(n)}_{m_a}} \E_x\left[ e^{\sum_{j=0}^{d_a-1} (h_{(m_a+j+1)/n}-h_{(m_a+j)/n}) S^{(n,a)}_j} \ind{S^{(n,a)}_{d_a} \in J^{(n)}_{a+1}}; S^{(n,a)}_j \in I^{(n,H)}_{m_a+j}, j \leq d_a \right]\\
  &=: \prod_{a=0}^{A-1} \tilde{R}^{(n)}_{a,A},
\end{align*}
with the same random walk notation as in the previous paragraph. Therefore,
\begin{equation}
  \label{eqn:lowerDivision}
  \liminf_{n \to +\infty} \frac{1}{n^{1/3}} \log \E\left[ e^{\sum_{j=1}^n (h_{(j+1)/n}-h_{j/n}) S^{(n)}_j}; S^{(n)}_j \in I^{(n)}_{j}, j \leq n \right]
  \geq \sum_{a=0}^{A-1} \liminf_{n \to +\infty} \frac{1}{n^{1/3}} \log \tilde{R}^{(n)}_{a,A}.
\end{equation}
We now bound from below $\tilde{R}^{(n)}_{a,A}$, replacing functions $f,g$ and $\dot{h}$ by constants. We write here
\[
  f_{a,A} = \sup_{t \in [\frac{a-1}{A},\frac{a+2}{A}]} f_t, \quad g_{a,A} = \inf_{t \in [\frac{a-1}{A}, \frac{a+2}{A}]} g_t \quad \mathrm{and} \quad \sigma_{a,A} = \inf_{t \in [\frac{a-1}{A}, \frac{a+2}{A}]} \sigma_t,
\]
keeping notations $\bar{h}_{a,A}$ and $\underline{h}_{a,A}$ as above. We assume $A > 0$ is chosen large enough such that
\[
  \sup_{|t-s| \leq \frac{2}{A}} |f_t - f_s| + |g_t-g_s| + |k_t-k_s| \leq \epsilon.
\]

We first observe that $[f_{a,A}n^{1/3}, g_{a,A} n^{1/3}] \subset I^{(n,H)}_j$ for all $j \in [m_a,m_{a+1}]$, therefore, writing $r_{a,A} = g_{a,A} - f_{a,A}$,
\begin{multline*}
  \tilde{R}^{(n)}_{a,A} \geq e^{(h_{m_{a+1}/n} - h_{m_a/n}) g_{a,A}n^{1/3}} \\
  \times \inf_{x \in J^{(n)}_a} \E_x\left[ e^{\frac{\bar{h}_{a,A}}{n} \sum_{j=0}^{d_a-1}  S^{(n,a)}_j} \ind{S^{(n,a)}_{d_a} \in J^{(n)}_{a+1}}; S^{(n,a)}_j \in [-r_{a,A}n^{1/3}, 0], j \leq d_a \right].
\end{multline*}
Thus, by \eqref{eqn:lowerboundRwTwosided}, we have
\begin{equation}
  \label{eqn:twosidedLower}
  \liminf_{n \to +\infty} \frac{1}{n^{1/3}} \log \tilde{R}^{(n)}_{a,A} \geq (h_{(a+1)/A} - h_{a/A}) g_{a,A} + \frac{\sigma_{a,A}^2}{A(g_{a,A}-f_{a,A})^2} \Psi\left( \tfrac{(g_{a,A}-f_{a,A})^3}{\sigma_{a,A}^2} \bar{h}_{a,A} \right).
\end{equation}

This lower bound can be improved, if $[\frac{a}{A},\frac{a+1}{A}] \subset F^c$. We have $[-Hn^{1/3}, g_{a,A}n^{1/3}] \subset I^{(n,H)}_j$ for all $j \in [m_a,m_{a+1}]$, thus
\begin{multline*}
  \tilde{R}^{(n)}_{a,A} \geq e^{(h_{m_{a+1}/n} - h_{m_a/n}) g_{a,A}n^{1/3}} \\ \times \inf_{x \in J^{(n)}_a} \E_x\left[ e^{\frac{\bar{h}_{a,A}}{n} \sum_{j=0}^{d_a-1}  S^{(n,a)}_j} \ind{S^{(n,a)}_{d_a} \in J^{(n)}_{a+1}}; S^{(n,a)}_j \in [-(H-g_{a,A})n^{1/3}, 0], j \leq d_a \right],
\end{multline*}
which leads to
\begin{equation}
  \label{eqn:premierOnesidedLower}
  \liminf_{n \to +\infty} \frac{1}{n^{1/3}} \log \tilde{R}^{(n)}_{a,A} \geq (h_{(a+1)/A} - h_{a/A}) g_{a,A} + \frac{\sigma_{a,A}^2}{A(g_{a,A}+H)^2} \Psi\left( \tfrac{(g_{a,A}+H)^3}{\sigma_{a,A}^2} \bar{h}_{a,A} \right).  
\end{equation}
By symmetry, if $[\frac{a}{A},\frac{a+1}{A}] \subset G^c$, we have
\begin{equation}
  \label{eqn:secondOnesidedLower}
  \liminf_{n \to +\infty} \frac{1}{n^{1/3}} \log \tilde{R}^{(n)}_{a,A} \geq (h_{(a+1)/A} - h_{a/A}) f_{a,A} + \frac{\sigma_{a,A}^2}{A(H-f_{a,A})^2} \Psi\left(- \tfrac{(H-f_{a,A})^3}{\sigma_{a,A}^2} \underline{h}_{a,A} \right).  
\end{equation}
As a consequence, letting $A \to +\infty$, by Riemann-integrability of $F$, $G$ and $\dot{h}$, \eqref{eqn:twosidedLower} leads to
\begin{equation}
  \label{eqn:twosidedLowerLimit}
  \liminf_{A \to +\infty} \sum_{\substack{0 \leq a \leq A\\ [\frac{a}{A},\frac{a+1}{A}] \cap F\cap G \neq \emptyset}} \liminf_{n \to +\infty} \frac{1}{n^{1/3}} \log \tilde{R}^{(n)}_{a,A}\\
  \geq \int_{F\cap G} \dot{h}_t g_t + \frac{\sigma^2_t}{(g_t-f_t)^2} \Psi\left( \tfrac{(g_t-f_t)^3}{\sigma_t^2} \dot{h}_t \right) dt.
\end{equation}
Similarly, \eqref{eqn:premierOnesidedLower} gives
\begin{equation}
  \label{eqn:premierOnesidedLowerLimit}
  \liminf_{A \to +\infty} \sum_{\substack{0 \leq a \leq A\\ [\frac{a}{A},\frac{a+1}{A}] \subset F^c}} \liminf_{n \to +\infty} \frac{1}{n^{1/3}} \log \tilde{R}^{(n)}_{a,A}\\
  \geq \int_{F^c} \dot{h}_t g_t + \frac{\sigma^2_t}{(g_t+H)^2} \Psi\left( \tfrac{(g_t+H)^3}{\sigma_t^2} \dot{h}_t \right) dt,
\end{equation}
and \eqref{eqn:secondOnesidedLower} gives
\begin{multline}
  \label{eqn:secondOnesidedLowerLimit}
  \liminf_{A \to +\infty} \sum_{\substack{0 \leq a \leq A\\ [\frac{a}{A},\frac{a+1}{A}] \subset F \cap G^c}} \liminf_{n \to +\infty} \frac{1}{n^{1/3}} \log \tilde{R}^{(n)}_{a,A}\\
  \geq \int_{F \cap G^c} \dot{h}_t g_t + \dot{h}_t(f_t - g_t) + \frac{\sigma^2_t}{(H-f_t)^2} \Psi\left( -\tfrac{(H-f_t)^3}{\sigma_t^2} \dot{h}_t \right) dt.
\end{multline}

Finally, we recall that
\[
  \lim_{H \to +\infty} \frac{1}{H^{2/3}} \Psi(H) = \frac{\alpha_1}{2^{1/3}}.
\]
As $\dot{h}$ is non-negative on $G^c$ and null on $F^c \cap G^c$, by dominated convergence, we have
\[
  \lim_{H \to +\infty} \int_{F^c} \frac{\sigma^2_s}{(g_s+H)^2} \Psi\left( \tfrac{(g_s+H)^3}{\sigma_s^2} \dot{h}_s \right) ds = \int_{F^c} \frac{\alpha_1}{2^{1/3}} (\dot{h}_s \sigma_s)^{2/3} ds,
\]
and as $\dot{h}$ is non-positive on $F^c$, we have similarly,
\[
  \lim_{H \to +\infty} \int_{F \cap G^c} \frac{\sigma^2_s}{(H-f_s)^2} \Psi\left( - \tfrac{(H-f_s)^3}{\sigma^2_s} \dot{h}_s \right) ds = \int_{F\cap G^c} \frac{\alpha_1}{2^{1/3}} (-\dot{h}_s \sigma_s)^{2/3}ds.
\]
Consequently, letting $n$, then $A$, then $H$ grow to $+\infty$ --observe that $\epsilon$, given it is small enough, does not have any impact on the asymptotic-- we have
\[
  \liminf_{n \to +\infty} \frac{1}{n^{1/3}} \log \E_0\left[ e^{\sum_{j=0}^{n-1} (h_{(j+1)/n} - h_{j/n}) S^{(n)}_j} \ind{S^{(n)}_n \in [an^{1/3},bn^{1/3}]} ; S^{(n)}_j \in \tilde{I}^{(n)}_j, j \leq n \right]
  \geq H_{f,g}^{F,G}.
\]

\paragraph*{Conclusion}
Using the fact that
\begin{multline*}
  \sup_{x \in \R} \E_x\left[ e^{\sum_{j=0}^{n-1} (h_{(j+1)/n} - h_{j/n})S^{(n)}_j} ; S^{(n)}_j \in I^{(n)}_j, j \leq n \right]\\
  \geq \E_0\left[ e^{\sum_{j=1}^n (h_{(j+1)/n} - h_{j/n}) S^{(n)}_j} \ind{S^{(n)}_n \in [an^{1/3},bn^{1/3}]} ; S^{(n)}_j \in \tilde{I}^{(n)}_j, j \leq n \right],
\end{multline*}
the two inequalities we obtained above allow to conclude the proof.
\end{proof}

\section{The many-to-one lemma and branching random walk estimates}
\label{sec:path}
In this section, we introduce a time-inhomogeneous version of the many-to-one lemma, that links some additive moments of the branching random walk with the random walk estimates obtained in the previous section. Using the well-established method in the branching random walk theory (see e.g. \cite{Aid13,AiJ11,AiS10,FaZ11,FaZ12,GHS11,HaR11,HuS09,MaZ13,Mal14} and a lot of others) that consists in proving the existence of a frontier via a first moment method, then bounding the tail distribution of maximal displacement below this frontier by estimation of first and second moments of the number of individuals below this frontier, and the Cauchy-Schwarz inequality. The frontier will be determined by a differential equation, which is solved in Section 5.

\subsection{Branching random walk notations and the many-to-one lemma}
\label{subsec:manytoone}

The many-to-one lemma can be traced back at least to the early works of Peyrière \cite{Pey74} and Kahane and Peyrière \cite{KaP76}. This result has been used under many forms in the past years, extended to branching Markov processes in \cite{BiK04}. This is a very powerful tool that has been used to obtain different branching random walk estimates, see e.g. \cite{Aid13,AiJ11,AiS10,FaZ11,FaZ12,HaR11}. We introduce some additional branching random walk notation in a first time.

Let $(\T,V)$ be a BRWtie of length $n$ with environment $(\calL_t, t \in [0,1])$. We recall that $\T$ is a tree of height $n$ and that for any $u \in \T$, $|u|$ is the generation to which $u$ belongs, $u_k$ the ancestor of $u$ at generation $k$ and $V(u)$ the position of $u$. We introduce, for $k \leq n$, $\calF_k = \sigma \left( (u,V(u)), |u| \leq k \right)$ the $\sigma$-field generated by the branching random walk up to generation $k$.

For $y \in \R$ and $k \leq n$, we denote by $\P_{k,y}$ the law of the time-inhomogeneous branching random walk $(\T^k, V^k)$ such that $\T^k$ is a tree of length $n-k$, and that $\{L^{u'}, u' \in \T^k, |u'| \leq n-k-1\}$ is a family of independent point processes, with $L^{u'}$ of law $\calL_{(|u'|+k+1)/n}$. With this definition, we observe that conditionally on $\calF_k$, for every individual $u \in \T$ alive at generation $k$, the subtree $\T^u$ of $\T$ rooted at $u$, with marks $V_{|\T^u}$ is a time-inhomogeneous branching random walk with law $\P_{|u|,V(u)}$, independent of the rest of the branching random walk $(\T \backslash \T^u, V)$.

We introduce $\phi$ a continuous positive function on $[0,1]$ such that
\begin{equation}
  \label{eqn:phibiendef}
  \forall t \in [0,1], \kappa_t(\phi_t) < +\infty,
\end{equation}
and set, for $t \in [0,1]$
\begin{equation}
  \label{eqn:meanandvariance}
  b_t = \partial_\theta \kappa_t(\phi_t) \quad \mathrm{and} \quad \sigma^2_t = \partial^2_\theta \kappa_t(\phi_t).
\end{equation}
Let $(X_{n,k}, n \geq 1, k \leq n)$ be a triangular array of independent random variables such that for all $n \geq 1$, $k\leq n$ and $x \in \R$, we have
\[
  \P(X_{n,k} \leq x) = \E\left[ \sum_{\ell \in L_{k/n}} \ind{\ell \leq x} e^{\phi_{k/n} \ell - \kappa_{k/n}(\phi_{k/n})} \right],
\]
where $L_{k/n}$ is a point process of law $\calL_{k/n}$. By \eqref{eqn:regularity} and \eqref{eqn:meanandvariance}, we have
\[
  \E(X_{n,k}) = b_{k/n} \quad \mathrm{and} \quad \E\left( (X_{n,k}-b_{k/n})^2 \right) = \sigma^2_{k/n}.
\]
For $k \leq n$, we denote by $S_k = \sum_{j=1}^k X_{n,j}$ the time-inhomogeneous random walk associated to $\phi$, by $\bar{b}^{(n)}_k = \sum_{j=1}^k b_{j/n}$, by $\tilde{S}_k = S_k - \bar{b}^{(n)}_k$ the centred version of this random walk and by
\begin{equation}
  \label{eqn:energyDef}
  E_k := \sum_{j=1}^k \phi_{j/n} b_{j/n} - \kappa_{j/n}(\phi_{j/n}) = \sum_{j=1}^k \kappa^*_{j/n}(b_{j/n}),
\end{equation}
by \eqref{eqn:legendreestimate}. Under the law $\P_{k,y}$, $(S_j, j \leq n-k)$ has the same law as $\left(y + \sum_{i=k+1}^{k+j+1} X_{n,i},j \leq n-k\right)$.

\begin{lemma}[Many-to-one lemma]
\label{lem:manytoone}
Let $n \geq 1$ and $k \leq n$. Under assumption \eqref{eqn:phibiendef}, for any measurable non-negative function $f$, we have
\[
  \E\left( \sum_{|u|=k} f(V(u_j),j \leq k) \right) = e^{-E_k}\E\left[ e^{-\phi_{k/n} \tilde{S}_k + \sum_{j=0}^{k-1} (\phi_{(j+1)/n}-\phi_{j/n}) \tilde{S}_j} f(S_j, j \leq k) \right].
\]
\end{lemma}

\begin{remark}
As an immediate corollary of the many-to-one lemma, we have, for $p \leq n$, $y \in \R$ and $k \leq n-p$,
\begin{multline*}
  \E_{p,y} \left( \sum_{|u|=p} f(V(u_j), j \leq k) \right)\\
  = e^{E_p-E_{k+p}} e^{\phi_{p/n} y} \E_{p,y} \left[ e^{-\phi_{(k+p)/n} \tilde{S}_k + \sum_{j=p}^{p+k-1}(\phi_{(j+1)/n}-\phi_{j/n}) \tilde{S}_{j-p}} f(S_j, j \leq k) \right].
\end{multline*}
\end{remark}

\begin{proof}
Let $n \geq 1$, $k \leq n$ and $f$ non-negative and measurable, we prove by induction on $k \leq n$ that
\[
  \E\left( \sum_{|u|=k} f(V(u_j),j \leq k) \right) = \E\left[ e^{-\sum_{j=1}^k \phi_{j/n}X_{n,j} - \kappa_{j/n}(\phi_{j/n})} f(S_j, j \leq k) \right].
\]
We first observe that if $k=1$, by definition of $X_{n,1}$, we have
\[
  \E\left( \sum_{|u|=1} f(V(u)) \right) = \E\left[ e^{-\phi_{1/n}X_{n,1} + \kappa_{1/n}(\phi_{1/n})} f(X_{n,1}) \right].
\]
Let $k \geq 2$. By conditioning on $\calF_{k-1}$, we have
\begin{align*}
  \E\left( \sum_{|u|=k} f(V(u_j), j \leq k) \right) &= \E\left[ \sum_{|u|=k-1} \sum_{u' \in \Omega(u)} f(V(u'_j), j \leq k) \right]\\
  &= \E\left( \sum_{|u|=k-1} g(V(u_j), j \leq k-1) \right),
\end{align*}
where, for $(x_j, j \leq k-1) \in \R^{k-1}$,
\begin{align*}
  g(x_j, j \leq k-1) &= \E\left[ \sum_{\ell \in L_{k/n}} f(x_1,\ldots x_{k-1}, x_{k-1} + \ell) \right]\\
  &= \E\left[ e^{-\phi_{k/n} X_{n,k} + \kappa_{k/n}(\phi_{k/n})} f(x_1,\ldots x_{k-1}, x_{k-1} + X_{n,k}) \right].
\end{align*}
Using the induction hypothesis, we conclude that
\begin{align*}
  \E\left( \sum_{|u|=k} f(V(u_j), j \leq k) \right) &= \E\left[ e^{-\sum_{j=1}^k \phi_{j/n} X_{n,j} - \kappa_{j/n}(\phi_{j/n})} f(S_j, j \leq k) \right]\\
  &= e^{-E_k} \E\left[ e^{-\sum_{j=1}^k \phi_{j/n} (X_{n,j}-b_{j/n})} f(S_j, j \leq k) \right].
\end{align*}
Finally, we modify the exponential weight by the Abel transform,
\begin{align*}
  \sum_{j=1}^k \phi_{j/n} (X_{n,j} - b_j)
  &= \sum_{j=1}^k \phi_{j/n} (\tilde{S}_j - \tilde{S}_{j-1})
  = \sum_{j=1}^k \phi_{j/n} \tilde{S}_j - \sum_{j=1}^k \phi_{j/n} \tilde{S}_{j-1}\\
  &= \sum_{j=1}^k \phi_{j/n} \tilde{S}_j - \sum_{j=1}^{k-1} \phi_{(j+1)/n} \tilde{S}_j
  = \phi_{k/n} \tilde{S}_k - \sum_{j=1}^{k-1} (\phi_{(j+1)/n}-\phi_{j/n}) \tilde{S}_j,
\end{align*}
which ends the proof.
\end{proof}

\subsection{Number of individuals staying along a path}
\label{subsec:moments}

In this section, we bound some quantities related to the number of individuals that stay along a path. We start with an upper bound of the expected number of individuals that stay in the path until some time $k \leq n$, and then exit the path by the upper frontier. Subsequently, we bound the probability that there exists an individual that stays in the path until time $n$. We then compute the first two moments of the number of such individuals, and apply the Cauchy-Schwarz inequality to conclude. We assume in this section that
\begin{equation}
  \label{eqn:regularphi}
  \phi \text{ is absolutely continuous, with a Riemann-integrable derivative } \dot{\phi},
\end{equation}
as we plan to apply Theorem \ref{thm:general_rw} with function $h = \phi$. Under this assumption, $\phi$ is Lipschitz, thus so is $b$. As a consequence, we have
\begin{equation}
  \label{eqn:energy}
  \sup_{\substack{n \in \N\\ k \leq n}} \sup_{t \in [\frac{k-1}{n},\frac{k+2}{n}]} \left| E_k - n K^*(b)_t \right| < +\infty \quad \mathrm{and} \quad   \sup_{\substack{n \in \N\\ k \leq n}} \sup_{t \in [\frac{k-1}{n},\frac{k+2}{n}]} \left| \bar{b}^{(n)}_k - n \int_0^t b_s ds \right| < +\infty.
\end{equation}

Let $f<g$ be two continuous functions such that $f(0)<0<g(0)$, and $F$ and $G$ two Riemann-integrable subsets of $[0,1]$ such that
\begin{equation}
  \label{eqn:FetG}
  \{ t \in [0,1] : \dot{\phi}_t < 0 \} \subset F \quad \mathrm{and} \quad \{ t \in [0,1] : \dot{\phi}_t > 0 \} \subset G.
\end{equation}
We write, for $t \in [0,1]$
\begin{multline}
  \label{eqn:defineK}
  H_t^{F,G}(f,g,\phi) = \int_0^t \dot{\phi}_s g_s ds + \int_0^t \mathbf{1}_{F \cap G}(s) \frac{\sigma_s^2}{(g_s - f_s)^2} \Psi\left( \tfrac{(g_s-f_s)^3}{\sigma_s^2} \dot{\phi}_s \right) ds\\
   + \int_0^t \mathbf{1}_{F^c \cap G}(s) \frac{a_1}{2^{1/3}} (\dot{\phi}_s \sigma_s)^{2/3}  + \mathbf{1}_{F \cap G^c} \left( \dot{\phi}_s(f_s - g_s) + \frac{a_1}{2^{1/3}} (-\dot{\phi}_s \sigma_s)^{2/3} \right) ds.
\end{multline}
We keep notation of Section \ref{sec:rw}: $F_n$ and $G_n$ are the subsets of $\{0,\ldots n-1\}$ defined in \eqref{eqn:deffnandgn}, and the path $I^{(n)}_k$ as defined in \eqref{eqn:defIn}. We are interested in the individuals $u$ alive at generation $n$ such that for all $k \leq n$, $V(u_k)-\bar{b}^{(n)}_k \in I^{(n)}_k$.

\subsubsection{A frontier estimate}
\label{subsubsec:frontier}

We compute the number of individuals that stayed in $\bar{b}^{(n)} + I^{(n)}$ until some time $k-1$ and then crossed the upper boundary $\bar{b}^{(n)}_k + g_{k/n} n^{1/3}$ of the path at time $k \in G_n$. We denote by
\[
  \calA_n^{F,G}(f,g) = \left\{ u \in \T, |u| \in G_n : V(u) - \bar{b}^{(n)}_{|u|} > g_{|u|/n}n^{1/3} , V(u_j) - \bar{b}^{(n)}_j \in I^{(n)}_j , j < |u|\right\},
\]
the set of such individuals, and by $A_n^{F,G}(f,g) = \# \calA_n^{F,G}(f,g)$.

\begin{lemma}
\label{lem:estimate_upperfrontier}
Under the assumptions \eqref{eqn:regularity}, \eqref{eqn:phibiendef}, \eqref{eqn:regularphi} and \eqref{eqn:FetG}, if $G \subset \{ t \in [0,1] : K^*(b)_t = 0\}$,
\begin{align*}
  \limsup_{n \to +\infty} \frac{1}{n^{1/3}} \log \E(A_n^{F,G}(f,g)) \leq \sup_{t \in [0,1]} \left[K^{F,G}_t(f,g,\phi) - \phi_t g_t \right].
\end{align*}
\end{lemma}

\begin{remark}
Observe that in order to use this lemma, we need to assume that
\[ \{ t \in [0,1] : \dot{\phi}_t > 0 \} \subset G \subset \{ t \in [0,1] : K^*(b)_t = 0\}, \]
we cannot consider paths of speed profile $b$ such that the associated parameter $\phi$ increases at a time when there is an exponentially large number of individuals following the path. For such paths, the mean of $A_n$ grows exponentially fast.
\end{remark}

\begin{proof}
By \eqref{eqn:phibiendef} and Lemma \ref{lem:manytoone}, we have
\begin{align*}
  \E(A_n^{F,G}(f,g)) &= \sum_{k\in G_n} \E\left[ \sum_{|u|=k} \ind{V(u) - \bar{b}^{(n)}_{k} > g_{k/n}n^{1/3}} \ind{V(u_j) - \bar{b}^{(n)}_j \in I^{(n)}_j , j < k} \right]\\
  &= \sum_{k \in G_n} e^{-E_k} \E\left[ e^{-\phi_{k/n} \tilde{S}_k + \sum_{j=0}^{k-1} (\phi_{(j+1)/n} - \phi_{j/n})\tilde{S}_j} \ind{\tilde{S}_k > g_{k/n} n^{1/3}} \ind{\tilde{S}_j \in I^{(n)}_j, j < k} \right].
\end{align*}
For all $k \in G_n$, there exists $t \in [k/n,(k+1)/n]$ such that $t \in G$, thus $K^*(b)_t=0$. By \eqref{eqn:energy}, this implies that $\sup_{n \in \N, k \in G_n} E_k< +\infty$, hence
\[
  \E(A_n^{F,G}(f,g))
  \leq C \sum_{k \in G_n} e^{-\phi_{k/n} g_{k/n} n^{1/3}} \E\left[ e^{\sum_{j=0}^{k-1} (\phi_{(j+1)/n}-\phi_{j/n})\tilde{S}_j} \ind{\tilde{S}_k > g_{k/n} n^{1/3}} \ind{\tilde{S}_j \in I^{(n)}_j, j < k} \right].
\]
As \eqref{eqn:FetG} is verified, similarly to \eqref{eqn:paslineaireConsequence}, for all $k \leq n$ and $x \in I^{(n)}_k$, we have
\begin{equation}
  \label{eqn:FetGconsequence}
  (\phi_{(k+1)/n} - \phi_{k/n}) x \leq (\phi_{(k+1)/n}-\phi_{k/n})_+ g_{k/n} n^{1/3} - (\phi_{k/n} - \phi_{(k+1)/n})_+ f_{k/n} n^{1/3}.
\end{equation}
In particular, $(\phi_{(k+1)/n}-\phi_{k/n}) x \leq \left| \phi_{(k+1)/n} - \phi_{k/n} \right| \left( \norme{f}_\infty + \norme{g}_\infty \right)$. Let $A>0$ be a large integer. For $a < A$, we set $m_a = \floor{an/A}$ and
\begin{multline*}
  \underline{g}_{a,A} = \inf\left\{g_t,t \in \left[\tfrac{a-1}{A},\tfrac{a+2}{A}\right]\right\},  \quad \underline{\phi}_{a,A} = \inf\left\{ \phi_t, t \in \left[\tfrac{a-1}{A},\tfrac{a+2}{A}\right] \right\}\\ \mathrm{and} \quad d_{a,A} = (\norm{f}_\infty + \norm{g}_\infty)\int_{(a-1)/A}^{(a+2)/A} |\dot{\phi}_s| ds.
\end{multline*}
For $k \in (m_a, m_{a+1}]$, applying the Markov property at time $m_a$, we have
\[
  \E\left[ e^{\sum_{j=0}^{k-1} (\phi_{(j+1)/n}-\phi_{j/n}) \tilde{S}_j} \ind{\tilde{S}_k > g_{k/n} n^{1/3}} \ind{\tilde{S}_j \in I^{(n)}_j, j < k} \right] \leq \exp\left(d_{a,A} n^{1/3}\right) \Phi_{a,A}^{(n)},
\]
where $\Phi_{a,A}^{(n)} = \E\left[ e^{\sum_{j=1}^{m_a}(\phi_{(j+1)/n}-\phi_{j/n}) \tilde{S}_j} \ind{\tilde{S}_j \in I^{(n)}_j, j \leq m_a} \right]$. We observe that $\tilde{S}$ is a centred random walk which, by \eqref{eqn:meanandvariance}, verifies \eqref{eqn:variance_rw} with variance function $\sigma^2$. Moreover, as
\begin{multline*}
  \E\left[ e^{\mu|X_{n,k}|} \right] \leq \E\left[ e^{\mu X_{n,k}} + e^{-\mu X_{n,k}} \right]
  \leq \E\left[ \sum_{\ell \in L_{k/n}} e^{(\phi_t+\mu) \ell - \kappa_t(\phi_t)} + e^{(\phi_t - \mu) \ell - \kappa_t(\phi_t)} \right]\\ \leq e^{\kappa_t(\phi_t+\mu) - \kappa_t(\phi_t)} + e^{\kappa_t(\phi_t-\mu) - \kappa_t(\phi_t)},
\end{multline*}
by \eqref{eqn:regularity}, there exists $\mu>0$ such that $\sup_{n \in \N, k \leq n} \E\left[ e^{\mu|X_{n,k}|} \right] < +\infty$ and \eqref{eqn:integrability_rw} is verified. For all $a \leq A$, we apply Theorem \ref{thm:general_rw}, to $h_t = \phi_{t\wedge a/A}$, functions $f$ and $g$ and intervals $F$ and $G$ stopped at time $a/A$. We have $ \displaystyle \limsup_{n \to +\infty} \frac{\log \Phi_{a,A}^{(n)}}{n^{1/3}} = K_{a/A}^{F,G}(f,g,\phi)$.

We observe that
\[
  \E(A^{F,G}_n(f,g)) \leq C \sum_{a=0}^{A-1} \frac{n}{A} \exp\left( \left(d_{a,A} - \underline{\phi}_{a,A} \underline{g}_{a,A}\right) n^{1/3}\right) \Phi_{a,A}^{(n)}.
\]
Letting $n \to +\infty$, we have
\[
  \limsup_{n \to +\infty} \frac{\log E(A_n(f,g))}{n^{1/3}} \leq \max_{a < A} K_{a/A}^{F,G}(f,g,\phi) - \underline{\phi}_{a,A} \underline{g}_{a,A} + d_{a,A}.
\]
By uniform continuity of $K,g,\phi$, and as $\lim_{A \to +\infty} d_{a,A} = 0$, letting $A  \to +\infty$, we have
\[
  \limsup_{n \to +\infty} \frac{\log E(A_n^{F,G}(f,g))}{n^{1/3}} \leq \sup_{t \in [0,1]} \left[ H_t^{F,G}(f,g,\phi) - \phi_t g_t \right].
\]
\end{proof}

Lemma \ref{lem:estimate_upperfrontier} is used to obtain an upper bound for the maximal displacement among individuals that stay above $\bar{b}^{(n)}_k + n^{1/3}f_{k/n}$ at any time $k \in F_n$. If the quantity $\sup_{t \in [0,1]} \left[H_t^{F,G}(f,g,\phi) - \phi_t g_t \right]$ is negative, then with high probability, no individual crosses the frontier $\bar{b}^{(n)}_k + n^{1/3} g_{k/n}$ at time $k \in (G \cup \{1\})_n$. In particular, there is at time $n$ no individual above $\bar{b}^{(n)}_n + g_1 n^{1/3}$. If we choose $g$ and $G$ in a proper manner, the upper bound obtained here is tight.

\subsubsection{Concentration estimate by a second moment method}
\label{subsubsec:secondmoment}

We take interest in the number of individuals which stay at any time $k \leq n$ in $\bar{b}^{(n)}_k + I^{(n)}_k$. For all $0<x < g_1 - f_1$, we set
\[
   \calB_n^{F,G} (f,g,x) = \left\{ |u|=n : V(u_j) - \bar{b}^{(n)}_j \in \tilde{I}^{(n)}_j, j \leq n, V(u)-\bar{b}^{(n)}_n\geq (g_1-x)n^{1/3}\right\},
\]
where $\tilde{I}^{(n)}_j = I^{(n)}_j \cap [-n^{2/3},n^{2/3}]$. We denote by $B_n^{F,G}(f,g,x) = \# \calB_n^{F,G}(f,g,x)$. In order to bound from above the probability that $\calB_n \neq \emptyset$, we compute the mean of $B_n$.

\begin{lemma}
\label{lem:firstorder}
We assume \eqref{eqn:regularity}, \eqref{eqn:phibiendef}, \eqref{eqn:regularphi} and \eqref{eqn:FetG}. If $K^*(b)_1=0$ then
\begin{equation*}
  \lim_{n \to +\infty} \frac{1}{n^{1/3}} \log \E(B^{F,G}_n(f,g,x)) = K_1^{F,G}(f,g,\phi) -  \phi_1 (g_1-x).
\end{equation*}
\end{lemma}

\begin{proof}
Observe that, as $K^*(b)_1 = 0$, by \eqref{eqn:energy} $|E_n|$ is bounded by a constant uniformly in $n \in \N$. Using the many-to-one lemma, we have
\begin{align*}
  \E(B_n^{F,G}(f,g,x)) &= e^{-E_n}\E\left[ e^{-\phi_1 \tilde{S}_n + \sum_{j=1}^n  (\phi_{(j+1)/n}-\phi_{j/n}) \tilde{S}_j} \ind{\tilde{S}_j \in \tilde{I}^{(n)}_j, j \leq n} \ind{\tilde{S}_n \geq (g_1 - x)n^{1/3}} \right]\\
  &\leq C e^{-\phi_1 (g_1 - x)n^{1/3}} \E\left[ e^{\sum_{j=1}^n (\phi_{(j+1)/n}-\phi_{j/n}) \tilde{S}_j} \ind{\tilde{S}_j \in I^{(n)}_j, j \leq n} \right].
\end{align*}
Therefore applying Theorem \ref{thm:general_rw}, we have
\[
  \limsup_{n \to +\infty} \frac{\log \E(B^{F,G}_n(f,g,x))}{n^{1/3}} = K_1^{F,G}(f,g,\phi) - \phi_1 (g_1-x).
\]

We compute a lower bound for $\E(B_n)$. Applying Lemma \ref{lem:manytoone}, for any $\epsilon>0$ we have
\begin{multline*}
  \E(B_n^{F,G}(f,g,x))\\ \geq e^{-E_n}\E\left[ e^{-\phi_1 \tilde{S}_n + \sum_{j=1}^n (\phi_{(j+1)/n}-\phi_{j/n}) \tilde{S}_j} \ind{\tilde{S}_j \in \tilde{I}^{(n)}_j, j \leq n} \ind{\tilde{S}_n - (g_1 - x)n^{1/3} \in [0,\epsilon n^{1/3}]} \right]\\
  \geq c e^{-\phi_1 (g_1 - x +\epsilon)n^{1/3}} \E\left[ e^{\sum_{j=1}^n (\phi_{(j+1)/n}-\phi_{j/n}) \tilde{S}_j} \ind{\tilde{S}_j \in \tilde{I}^{(n)}_j, j \leq n} \ind{\tilde{S}_n - (g_1 - x)n^{1/3} \in [0,\epsilon n^{1/3}]}\right].
\end{multline*}
Applying Theorem \ref{thm:general_rw} again, we have
\[
  \liminf_{n \to +\infty} \frac{\log \E\left(B^{F,G}_n(f,g,x)\right)}{n^{1/3}} \geq K_1^{F,G}(f,g,\phi) - \phi_1 (g_1 - x + \epsilon).
\]
Letting $\epsilon \to 0$ concludes the proof.
\end{proof}

To obtain a lower bound for $\P(\calB_n \neq \emptyset)$, we compute an upper bound for the second moment of $B_n$. We assume
\begin{equation}
  \label{eqn:integrability2phi}
  \sup_{t \in [0,1]} \E\left[ \left(\sum_{\ell \in L_t} e^{\phi_t \ell} \right)^2 \right] < +\infty
\end{equation}
which enables to bound the second moment of $B_n$.

\begin{lemma}
\label{lem:secondmoment_estimate}
Under the assumptions \eqref{eqn:regularity}, \eqref{eqn:phibiendef}, \eqref{eqn:regularphi}, \eqref{eqn:FetG} and \eqref{eqn:integrability2phi}, if $G = [0,1]$, $K^*(b)_1=0$ and for all $t \in [0,1]$, $K^*(b)_t \leq 0$, then
\begin{multline*}
  \limsup_{n \to + \infty} \frac{1}{n^{1/3}} \log \E\left({B^{F,G}_n(f,g,x)}^2\right)\\
  \leq 2 \left[ K^{F,G}_1(f,g,\phi) - \phi_1 (g_1 - x) \right] - \inf_{t \in [0,1]}\left[ H_t^{F,G}(f,g,\phi)- \phi_t g_t\right].
\end{multline*}
\end{lemma}

\begin{proof}
In order to estimate the second moment of $B_n$, we decompose the pairs of individuals $(u,u') \in \T^2$ according to their most recent common ancestor $u \wedge u'$ as follows:
\begin{multline*}
  \E\left[ B^{F,G}_n(f,g,x)^2 \right]
  = \sum_{k = 0}^n \E\left[ \sum_{\substack{|u|=|v|=n\\ |u \wedge u'|=k}} \ind{u \in \calB^{F,G}_n(f,g,x)} \ind{u' \in \calB^{F,G}_n(f,g,x)} \right]\\
  = \E\left[ B^{F,G}_n(f,g,x) \right] + \sum_{k=0}^{n-1} \E\left[ \sum_{|u|=k} \ind{V(u_j) - \bar{b}^{(n)}_j \in \tilde{I}^{(n)}_j, j \leq k} \sum_{u_1 \neq u_2 \in \Omega(u)} \Lambda(u_1)\Lambda(u_2) \right],
\end{multline*}
where, for $u' \in \T$, we denote by $\Lambda(u') = \sum_{|u|=n, u > u'} \ind{u \in \calB^{F,G}_n(f,g,x)}$ the number of descendants of $u'$ which are in $\calB_n$. We observe that for any two distinct individuals $|u_1|=|u_2|=k$, conditionally to $\calF_{k}$, the quantities $\Lambda(u_1)$ and $\Lambda(u_2)$ are independent.

By the Markov property applied at time $k$, for all $u' \in \T$ with $|u'|=k$, we have
\begin{align*}
  \E\left[ \left.\Lambda(u') \right| \calF_k \right]
  &= \E_{k,V(u')}\left[ \left. \sum_{|u|=n-k} \ind{V(u) - \bar{b}^{(n)}_n \geq (g_1 - x)n^{1/3}} \ind{V(u_j) - \bar{b}^{(n)}_{j+k} \in \tilde{I}^{(n)}_j, j \leq n-k} \right|  \calF_k \right]\\
  &= \exp\left( - E_n + E_k  + \phi_{k/n} (V(u') - \bar{b}^{(n)}_k)\right)\\
  &\qquad \times \E_{k,V(u')} \left[ e^{-\phi_1 \tilde{S}_{n-k} + \sum_{j=0}^{n-k-1} \Delta \phi_{n,k+j} \tilde{S}_j} \ind{\tilde{S}_j \in \tilde{I}^{(n)}_{j+k}, j \leq n-k} \ind{\tilde{S}_{n-k} \geq (g_1 - x)n^{1/3}} \right],
\end{align*}
using the many-to-one lemma. Therefore,
\begin{multline*}
  \E\left[ \left. \Lambda(u') \right| \calF_k \right] \leq C \exp\left(E_k + \phi_{k/n} (V(u') - \bar{b}^{(n)}_k) -\phi_1 (g_1 - x)n^{1/3}\right) \\
  \times \E_{k,V(u')}\left[ e^{\sum_{j=0}^{n-k-1}\Delta \phi_{n,j+k} \tilde{S}_j} \ind{\tilde{S}_j \in I^{(n)}_{j+k}, j \leq n-k} \right].
\end{multline*}

Let $A > 0$ be a large integer, and for $a \leq A$, let $m_a = \floor{ an/A}$. We introduce
\begin{multline*}
  \Phi_{a,A}^\mathrm{start} = \E\left[ \exp\left(\sum_{j=0}^{m_a-1} (\phi_{(j+1)/n}-\phi_{j/n}) \tilde{S}_j\right) \ind{\tilde{S}_j \in \tilde{I}^{(n)}_j, j \leq m_a} \right] \quad \mathrm{and}\\
  \Phi^\mathrm{end}_{a,A} = \sup_{y \in \R} \E_{m_a,y} \left[ \exp\left(\sum_{j=0}^{n-m_a-1} \Delta \phi_{n,m_a+j} \tilde{S}_j\right) \ind{\tilde{S}_j \in I^{(n)}_{m_a+j}, j \leq n-m_a} \right].
\end{multline*}
By Theorem \ref{thm:general_rw}, we have
\[
 \limsup_{n \to +\infty} \frac{\log \Phi^\mathrm{start}_{a,A}}{n^{1/3}} = K_{a/A}^{F,G}(f,g,\phi)\quad \mathrm{and} \quad \limsup_{n \to +\infty} \frac{\log \Phi^\mathrm{end}_{a,A}}{n^{1/3}} = K_1^{F,G}(f,g,\phi) - K_{a/A}^{F,G}(f,g,\phi).
\]
Moreover, using the same estimates as in Lemma \ref{lem:estimate_upperfrontier}, and setting
\begin{multline*}
  \bar{g}_{a,A} = \sup\left\{g_t,t \in \left[\tfrac{a-1}{A},\tfrac{a+1}{A}\right]\right\},  \quad \bar{\phi}_{a,A} = \sup\left\{ \phi_t, t \in \left[\tfrac{a-1}{A},\tfrac{a+1}{A}\right] \right\}\\ \mathrm{and} \quad d_{a,A} = \int_{(a-1)/A}^{(a+1)/A} |\dot{\phi}_s| ds (\norm{f}_\infty + \norm{g}_\infty),
\end{multline*}
for all $k \in [m_a,m_{a+1})$, applying the Markov property at time $m_{a+1}$, we have
\begin{equation}
  \label{eqn:endpart}
  \E\left[ \Lambda(u') | \calF_k \right] \leq C e^{E_k + \phi_{k/n} (V(u') - \bar{b}^{(n)}_k)} \exp\left( \left(d_{a,A}-\phi_1 (g_1 - x)\right)n^{1/3}\right) \Phi^\mathrm{end}_{a+1,A}.
\end{equation}

We observe that for all $u \in \T$ with $|u|=k$ and $V(u) \in \tilde{I}^{(n)}_k$ we have
\begin{align}
  \E\left[ \left. \sum_{u_1 \neq u_2 \in \Omega(u)} e^{\phi_{(k+1)/n} (V(u_1)+V(u_2))} \right| \calF_k \right] &\leq e^{2 \phi_{(k+1)/n} V(u)} \E\left[ \left(\sum_{\ell \in L_{(k+1)/n}} e^{\phi_{(k+1)/n} \ell} \right)^2 \right] \nonumber\\
  &\leq C e^{2 \phi_{k/n} V(u)} e^{n^{2/3}|\phi_{(k+1)/n} - \phi_{k/n}|} \leq Ce^{2 \phi_{k/n} V(u)},\label{eqn:middlepart}
\end{align}
using \eqref{eqn:integrability2phi} and the fact that $\phi$ is Lipschitz. We now bound, for $k \in [m_a, m_{a+1})$
\begin{multline}
  \E\left[ \sum_{|u|=k} e^{2 \phi_{k/n}(V(u) - \bar{b}^{(n)}_k)} \ind{V(u_j) - \bar{b}^{(n)}_j \in \tilde{I}^{(n)}_j, j \leq k} \right]\\
  = \E\left[ e^{\phi_{k/n} \tilde{S}_k + \sum_{j=0}^{k-1} (\phi_{(j+1)/n}-\phi_{j/n}) \tilde{S}_j} \ind{\tilde{S}_j \in \tilde{I}^{(n)}_j, j \leq k} \right],
\end{multline}
using Lemma \ref{lem:manytoone}. As $\sup_{t \in [0,1]} K^*(b)_t \leq 0$ and by \eqref{eqn:energy}, $E_k$ is bounded from above uniformly in $n \in \N$ and $k \leq n$. As $G_n = \{0,\ldots, n\}$, for all $n \in \N$ large enough and $k \in [m_a, m_{a+1})$, applying the Markov property at time $m_an$, it yields
\begin{equation}
 \label{eqn:startpart}
  \E\left[ \sum_{|u|=k} e^{2 \phi_{k/n}(V(u) - \bar{b}^{(n)}_k)} \ind{V(u_j) \in \tilde{I}^{(n)}_j, j \leq k} \right]
  \leq \exp\left( \left(\bar{\phi}_{a,A} \bar{g}_{a,A} +d_{a,A}\right)n^{1/3}\right) \Phi^\mathrm{start}_{a,A}.
\end{equation}
Finally, combining \eqref{eqn:endpart} with \eqref{eqn:middlepart} and \eqref{eqn:startpart}, for all $n \geq 1$ large enough and $k \in [m_a, m_{a+1})$,
\begin{multline*}
  \E\left[ \sum_{|u|=k} \ind{V(u_j) - \bar{b}^{(n)}_j \in \tilde{I}^{(n)}_j, j \leq k} \sum_{u_1 \neq u_2 \in \Omega(u)} \Lambda(u_1)\Lambda(u_2) \right]\\
  \leq C \exp\left[ n^{1/3}\left(-2\phi_1 (g_1 - x) + \bar{\phi}_{a,A} \bar{g}_{a,A} + 3 d_{a,A}\right) \right] \Phi^\mathrm{start}_{a,A} \left(\Phi^\mathrm{end}_{a+1,A}\right)^2,
\end{multline*}
thus
\begin{multline*}
  \limsup_{n \to +\infty} \frac{1}{n^{1/3}} \log \sum_{k=0}^{n-1} \E\left[ \sum_{|u|=k} \ind{V(u_j) - \bar{b}^{(n)}_j \in \tilde{I}^{(n)}_j, j \leq k} \sum_{u_1 \neq u_2 \in \Omega(u)} \Lambda(u_1)\Lambda(u_2) \right]\\
  \leq 2 \left( K_1^{F,G}(f,g,\phi) - (g_1-x)\right) - \min_{a < A} 2 K_{\frac{a+1}{A}}^{F,G}(f,g,\phi)-K_{\frac{a}{A}}^{F,G}(f,g,\phi) - \bar{\phi}_{a,A} \bar{g}_{a,A} - 3 d_{a,A}.
\end{multline*}
Letting $A \to +\infty$, and using Lemma \ref{lem:firstorder}, we obtain
\[
  \limsup_{n \to +\infty} \frac{1}{n^{1/3}} \log \E(B_n(f,g)^2) \leq 2 \left( K_1^{F,G}(f,g,\phi) - (g_1-x)\right) - \inf_{t \in [0,1]} \left(H_t^{F,G}(f,g,\phi) - \phi_t g_t\right).
\]
\end{proof}

Using the previous two lemmas, we can bound from below the probability that there exists an individual that follows the path $\bar{b}^{(n)} + I^{(n)}$. 
\begin{lemma}
\label{lem:lowerbound}
Assuming \eqref{eqn:regularity}, \eqref{eqn:phibiendef}, \eqref{eqn:regularphi}, \eqref{eqn:FetG} and \eqref{eqn:integrability2phi}, if $K^*(b)_1 = \sup_{t \in [0,1]} K^*(b)_t = 0$, then for any $x < g_1$
\begin{equation}
  \label{eqn:probaLower}
  \liminf_{n \to +\infty} \frac{1}{n^{1/3}} \log \P(\calB^{F,G}_n(f,g,x) \neq \emptyset) 
  \geq \inf_{t \in [0,1]} \left( K^{F,G}_t(f,g,\phi) - \phi_t g_t\right).
\end{equation}
\end{lemma}

\begin{proof}
We first assume that $G = [0,1]$. Since $B_n \in \Z_+$ a.s, we have
\[
  \P(\calB^{F,G}_n(f,g,x) \neq \emptyset) = \P(B^{F,G}_n(f,g,x)>0) \geq \frac{\E(B^{F,G}_n(f,g,x))^2}{\E(B^{F,G}_n(f,g,x)^2)},
\]
using the Cauchy-Schwarz inequality. As a consequence,
\begin{align*}
  \liminf_{n \to + \infty} \frac{1}{n^{1/3}} &\log \P(\calB^{F,G}_n(f,g,x) \neq \emptyset)\\
  &\geq 2 \liminf_{n \to +\infty} \frac{1}{n^{1/3}} \log \E(B^{F,G}_n(f,g,x)) - \limsup_{n \to +\infty} \frac{1}{n^{1/3}} \log \E\left(B^{F,G}_n(f,g,x)^2\right)\\
  &\geq \inf_{t \in [0,1]} \left( H_t^{F,G}(f,g,\phi) - \phi_t g_t \right).
\end{align*}

We then extend this estimate for $G$ a Riemann-integrable subset of $[0,1]$, that we can, without loss of generality, choose closed --as the Lebesgue measure of the boundary of a Riemann-integrable set is null. According to \eqref{eqn:FetG}, $\{ \dot{\phi}>0\} \subset G$. We set, for $H > 0$
\[
  g^H_t = \max\left\{ g_t, -\norme{g}_\infty + Hd(t,G) \right\}.
\]
Observe that $g^H$ is an increasing sequence of functions, that are equal to $g$ on $G$ and increase to $+\infty$ on $G^c$. For all $n \in \N$, $x \in [f_1,g_1]$ and $H>0$, we have $\calB^{F,[0,1]}_n(f,g^H,x) \subset \calB^{F,G}_n(f,g,x)$. As a consequence,
\begin{align*}
  \liminf_{n \to +\infty} \frac{1}{n^{1/3}} \log \P(\calB^{F,G}_n(f,g,x) \neq \emptyset)
  &\geq \lim_{H \to +\infty} \liminf_{n \to +\infty} \frac{1}{n^{1/3}} \log \P(\calB^{F,[0,1]}_n(f,g^H,x) \neq \emptyset)\\
  &\geq \lim_{H \to +\infty} \inf_{t \in [0,1]} \left( H_t^{F,[0,1]}(f,g^H, \phi) - \phi_t g^H_t\right).
\end{align*}
By Lemma \ref{lem:bmTwosided}, we have $\Psi(h) \sim_{h \to +\infty} \frac{\alpha_1}{2^{1/3}} h^{2/3}$. Thus, using \eqref{eqn:FetG}, this yields
\[
  \liminf_{n \to +\infty} \frac{1}{n^{1/3}} \log \P(\calB^{F,G}_n(f,g,x) \neq \emptyset) \geq \inf_{t \in [0,1]} \left( H_t^{F,G}(f,g, \phi) - \phi_t g_t\right).
\]
\end{proof}

\begin{remark}
Observe that the inequality in Lemma \ref{lem:lowerbound} is sharp when
\[
  \inf_{t \in [0,1]} \left(H_t^{F,G}(f,g,\phi) - \phi_t g_t\right) = K_1^{F,G}(f,g,\phi) - \phi_1 g_1.
\]
\end{remark}

\section{Identification of the optimal path}
\label{sec:optimization}

We denote by $\calR = \left\{ b \in \calD : \forall t \in [0,1], K^*(b)_t \leq 0 \right\}$. In this section, we take interest in functions $a \in \calR$ that verify
\begin{equation}
  \int_0^1 a_s ds = \sup\left\{ \int_0^1 b_s ds, b \in \calR \right\},
\end{equation}
i.e. which are solution of \eqref{eqn:existence_max}. This equation is an optimisation problem under constraints. Information on its solution can be obtained using a theorem of existence of Lagrange multipliers in Banach spaces.

Let $E,F$ be two Banach spaces, a function $f : E \to F$ is said to be differentiable at $u \in E$ if there exists a linear continuous mapping $D_u f : E \to F$ called its \textit{Fréchet derivative at $u$}, verifying
\[
  f(u+h) = f(u) + D_u f(h) + o(\norm{h}), \quad \norm{h} \to 0, \quad h \in E.
\]
A set $R$ is a \textit{closed convex cone} of $F$ if it is a closed subset of $F$ such that
\[
  \forall x, y \in R, \forall \lambda, \mu \in [0,+\infty)^2 , \lambda x + \mu y \in K.
\]
Finally, we set $F^*$ the set of linear continuous mappings from $F$ to $\R$. We now introduce a result on the existence of Lagrange multipliers in Banach spaces obtained in \cite{Kur76}, Theorem 4.5.
\begin{theorem}[Kurcyusz \cite{Kur76}]
\label{thm:lagrange}
Let $E,F$ be two Banach spaces, $J : E \to \R$, $g : E \to F$ and $R$ be a closed convex cone of $F$. If $\hat{u}$ verifies
\[J(\hat{u}) = \max\{ J(u), u \in E : g(u) \in R\} \quad \mathrm{and} \quad g(\hat{u}) \in R,\]
and if $J$ and $g$ are both differentiable at $\hat{u}$, and $D_{\hat{u}} g$ is a bijection, then there exists $\lambda \in F^*$ such that
\begin{align}
  &\forall h \in E, D_{\hat{u}} J (h) = \lambda^*\left[ D_{\hat{u}} g(h)\right] \label{eqn:lagrange1}\\
  &\forall h \in R, \lambda^*(h) \leq 0 \label{eqn:lagrange2} \\
  &\lambda^*(g(\hat{u})) = 0 \label{eqn:lagrange3}.
\end{align}
\end{theorem}

We first introduce the \textit{natural speed path} of the branching random walk, which is the path driven by $(v_t, t \in [0,1])$.
\begin{lemma}
Under the assumptions \eqref{eqn:breeding} and \eqref{eqn:regularity}, there exists a unique $v\in \calR$ such that for all $t \in [0,1]$, $\kappa^*_t(v_t)=0$. Moreover, for all $t \in [0,1]$, $\bar{\theta}_t := \partial_a \kappa^*_t(v_t)>0$, and $v$ and $\bar{\theta}$ are $\calC_1$ function.
\end{lemma}

\begin{proof}
For any $t \in [0,1]$ we have $\inf_{a \in \R} \kappa^*_t(a) = -\kappa_t(0)< 0$, as $\kappa^*_t$ is the Fenchel-Legendre transform of $\kappa_t$. Moreover, $a \mapsto \kappa^*_t(a)$ is convex, continuous on the interior of its definition set and increasing. By \eqref{eqn:regularity}, we have $\kappa^*(a) \to +\infty$ when $a$ increases to $\sup\{ b \in \R : \kappa^*_t(b) < +\infty\}$. As a consequence, by continuity, there exists $x \in \R$ such that $\kappa^*_t(x)=0$. As $\inf_{a \in \R} \kappa^*_t(a)<0$, $\kappa^*_t$ is strictly increasing at point $x$. Therefore the point $v_t=x$ is uniquely determined, and $\bar{\theta}_t = \partial_a\kappa^*_t(x)$ at point $x$ is positive. Finally, $v \in \calC_1$ by the implicit function theorem; thus so is $\bar{\theta}$, by composition with $\partial_a \kappa^*$.
\end{proof}

We now observe that if $a$ is a solution of \eqref{eqn:existence_max}, then $a$ is a regular point of $\calR$ --i.e. we can apply Theorem \ref{thm:lagrange}.
\begin{lemma}
\label{lem:regularity_existence_max}
Under the assumptions \eqref{eqn:breeding} and \eqref{eqn:regularity}, if $a$ is a solution of \eqref{eqn:existence_max}, then for all $t \in [0,1]$, $\partial_a \kappa^*_t(a_t)>0$.
\end{lemma}

\begin{proof}
Let $a \in \calR$ be a solution of \eqref{eqn:existence_max}. For $t \in [0,1]$, we set $\theta_t = \partial_a \kappa^*_t(a_t)$. We observe that $\theta \in \calD$ is non-negative.

We first assume that for all $t \in [0,1]$, $\theta_t=0$, in which case $\kappa^*_t(a_t)$ is the minimal value of $\kappa^*_t$. By \eqref{eqn:breeding}, we have $\inf_{t \in [0,1]} \kappa_t(0)>0$, thus $\sup_{t \in [0,1]} \inf_{a \in \R} \kappa^*_t(a) < 0$. As a consequence, by continuity, there exists $x \in \calD$ such that for all $t \in [0,1]$, $\kappa^*_t(a_t + x_t) \leq 0$. We have $a+x \in \calR$ and $\int_0^1 a_s + x_s ds > \int_0^1 a_s ds$, which contradicts $a$ is a solution of \eqref{eqn:existence_max}.

We now assume that $\theta$ is non-identically null, but there exists $t \in [0,1]$ such that $\theta_t = 0$. We start with the case $\theta_0>\epsilon>0$. As $\theta \in \calD$, there exists $t>0$ and $\delta>0$ such that $\inf_{s \in [0,\delta]} \theta_s > \epsilon$ and $\sup_{s \in [t,t+\delta]} \theta_s < \epsilon/3$. For $x>0$, we set $a^x =  a- x \mathbf{1}_{[0,\delta]} + 2x \mathbf{1}_{[t,t+\delta]}$. We observe that uniformly for $s \in [0,1]$, as $x \to 0$
\[
  K^*(a^x)_s \leq K^*(a)_s - x \epsilon s \wedge \delta + \tfrac{2}{3} x \epsilon (s-t)_+ \wedge \delta + O(x^2).
\]
Thus there exists $x>0$ small enough such that $a^x \in \calR$ and $\int a^x > \int a$, which contradicts again the fact that $a$ is a solution of \eqref{eqn:existence_max}.

Finally, we assume that $\theta_0=0$. In this case, there exists $\delta > 0$ such that $K^*(a)_t < - \delta t$ for all $t \leq \delta$. Therefore, there exists $t>0$ such that for all $0<s \leq t$, $K^*(a)_s < 0$, and $\theta_t>0$. For all $\theta_t>\epsilon>0$, there exists $\delta'>0$ such that for all for all $s< \delta'$, we have $\theta_s < \epsilon/3$ and for all $s \in [t,t+\delta']$, $\theta_s > 2 \epsilon$. Therefore, setting $a^x = a + 2x \mathbf{1}_{[0,\delta)} - x \mathbf{1}_{[t,t+\delta)}$, as $x \to 0$, uniformly in $s \in [0,1]$, we have
\[
  K^*(a^x) \leq K^*(a)_s + \frac{2}{3}x \epsilon (s \wedge \delta') - x \epsilon ((s-t)_+\wedge \delta) + O(x^2),
\]
so for $x>0$ small enough we have $a^x \in \calR$. Moreover $\int_0^1 a^x > \int_0^1 a$ which, once again, contradicts the fact that $a$ is a solution of \eqref{eqn:existence_max}.
\end{proof}

Applying Theorem \ref{thm:lagrange}, and using the previous lemma, we prove Proposition \ref{prop:regularity}.
\begin{proof}[Proof of Proposition \ref{prop:regularity}]
We first consider a function $a \in \calR$ that verifies
\[
  \int_0^1 a_s ds = \sup\left\{ \int_0^1 b_s ds , b \in \calR \right\},
\]
i.e. such that $a$ is a solution of \eqref{eqn:existence_max}. We set $\theta_t = \partial_a \kappa^*_t(a_t)$, and observe that $\theta \in \calD$.

We introduce $J : a \mapsto \int_0^1 a_s ds$ and $g : a \mapsto \left(\kappa^*_s(a_s),s \in [0,1]\right)$. These functions are differentiable at $a$, and for $h \in \calD$, we have $D_a J(h) = \int_0^1 h_s ds$ and $D_a g(h)_t = \theta_t h_t$. We denote by
\[R = \left\{ h \in \calD : \forall t \in [0,1], \int_0^t h_s ds \leq 0 \right\}, \]
which is a closed convex cone of $\calD$. Using Lemma \ref{lem:regularity_existence_max}, we have $\theta_t > 0$ for all $t \in [0,1]$, thus $D_a g$ is a bijection.

By Theorem \ref{thm:lagrange}, there exists $\lambda^* \in \calD^*$ --which is a measure by the Riesz representation theorem-- such that
\begin{align}
  & \forall h \in \calD , \int_0^1 h_s ds = \int_0^1 D_a g(h)_s \lambda^*(ds)\label{eqn:z1}\\
  & \forall h \in R, \int_0^1 h_s \lambda^*(ds) \leq 0 \label{eqn:z2}\\
  & \int_0^1 g(a)_s \lambda^*(ds) = 0 \label{eqn:z3}.
\end{align}

We observe easily that \eqref{eqn:z1} implies that $\lambda^*$ admits a Radon-Nikod\'ym derivative with respect to the Lebesgue measure, and that $\frac{\lambda^*(ds)}{ds} = \frac{1}{\theta_s}$. As a consequence, we can rewrite \eqref{eqn:z2} as
\[
  \forall h \in R, \int_0^1 h_s \frac{ds}{\theta_s} \leq 0.
\]
We set $f_t = \int_0^t \frac{ds}{\theta_s}$, for all $s, t \in [0,1]$, and $\mu \in (0,1)$, by \eqref{eqn:z2}, we have
\[
  \mu f_t + (1 - \mu) f_s - f_{\mu t + (1 - \mu) s} = \int_0^1 \left(\mu \ind{u < t} + (1 - \mu) \ind{u<s} - \ind{u<\mu t + (1 - \mu)s}\right) \frac{du}{\theta_u}
  \leq 0.
\]
As a consequence, $f$ is concave. In particular, its right derivative function $\frac{1}{\theta}$ is non-increasing. Consequently $\theta$ is non-decreasing.

The last equation \eqref{eqn:z3} gives
\[
  0 = \int_0^1 \kappa^*_s(a_s) \lambda^*(ds)= \int_0^1 \kappa^*_s(a_s) \theta_s^{-1} ds = K^*(a)_1 \frac{1}{\theta_1} - \int_0^1 K^*(a)_s d\theta_s^{-1},
\]
by Stieltjès integration by part. But for all $t \in [0,1]$, $K^*(a)_t \leq 0$, and $\frac{1}{\theta}$ is non-increasing. This yields
\[
  K^*(a)_1 = 0 \quad \mathrm{and} \quad \int_0^1 K^*(a)_s d\theta^{-1}_s = 0.
\]
In particular, as $a \in \calR$, $\theta$ increases on $\{t \in [0,1] : K^*(a)_t = 0\}$.

Conversely, we consider a function $a \in \calR$ such that the function $\theta : t \mapsto \partial_a \kappa^*_t(a_t)$ is non-decreasing, $K^*(a)_1=0$ and $\int_0^1 K^*(a)_s d\theta^{-1}_s=0$. Our aim is to prove that $\int_0^1 a_s ds = v^*$, by observing that $\int_0^1 a_s ds \geq \int_0^1 b_s ds$ for all $b \in \calR$. By convexity of $\kappa^*_t$, we have, for all $t \in [0,1]$
\[
  \kappa^*_t(b_t) \geq \kappa^*_t(a_t) + \theta_t(b_t - a_t),
\]
and integrating with respect to $t$, we obtain
\begin{align*}
  \int_0^1 a_t - b_t dt &\geq \int_0^1 \frac{\kappa^*_t(a_t) - \kappa^*_t(b_t)}{\theta_t} dt\\
  &\leq K^*(a)_1 - K^*(b)_1 - \int_0^1 \left(K^*(a)_t - K^*(b)_t\right) d\theta^{-1}_t,
\end{align*}
by Stieltljès integration by parts. Using the specific properties of $a$, we get
\[
  \int_0^1 a_t - b_t dt \leq -K^*(b)_1 + \int_0^1 K^*(b)_t d\theta^{-1}_t.
\]
As $K^*(b)$ is non-positive, and $\theta^{-1}$ is non-increasing, we conclude that the left-hand side is non-positive, which leads to $\int_0^1 a_s ds \geq \int_0^1 b_s ds$. Optimizing this inequality over $b \in \calR$ proves that $a$ is a solution of \eqref{eqn:existence_max}.

We now prove that if $a$ is a solution of \eqref{eqn:existence_max}, then $a$ is continuous, by proving that this function has no jump. In order to do so, we assume that there exists $t \in (0,1)$ such that $a_t \neq a_{t-}$, i.e. such that $a$ jumps at time $t$. Then, $\theta_t \neq \theta_{t-}$ by continuity of $\partial_a \kappa^*$ on $D^*$. As $\int_0^1 K^*(a)_s d\theta^{-1}_s = 0$ and $d\theta^{-1}$ has an atom at point $t$, thus $K^*(a)_t=0$.

Therefore, if $a$ jumps at time $t$, then the continuous function $s \mapsto K^*(a)_s$ with right and left derivatives at each point, bounded from above by $0$, hits a local maximum at time $t$. Its left derivative $\kappa^*_t (a_{t-})$ is then non-negative and its right derivative $\kappa^*_t(a_t)$ non-positive. As $\kappa^*t$ is a non-decreasing function, we obtain $a_{t-} \geq a_t$.

Moreover, by convexity of $\kappa^*_t$, $x \mapsto \partial_a \kappa^*_t(x)$ is also non-decreasing, and as a consequence $\theta_{t-} \geq \theta_t$, which is a contradiction with the hypothesis $\theta_{t-} \neq \theta_t$ and $\theta$ non-decreasing. We conclude that $a$ (and $\theta$) is continuous as a càdlàg function with no jump.

We now assume there exists another solution $b \in \calR$ to \eqref{eqn:existence_max}. Using the previous computations, we have
$\int_0^1 K^*(b)_s d\theta^{-1}_s = 0$, and $b$ is continuous. As a consequence, denoting by $T$ the support of $d\theta^{-1}$, for all $t \in T$, $K^*(b)_t=0$. Moreover, $K^*(b)$ is a $\calC^1$ function, with a local maximum at time $t$, thus $\kappa^*_t(b_t)=0$, or in other words, $b_t = v_t$, by Lemma \ref{lem:regularity_existence_max}.

Consequently, if we write $\phi_t = \partial_a \kappa^*_t(b_t)$, we know from previous results that $\phi$ is continuous and increasing. Furthermore, $\phi$ increases only on $T$, and $\phi_t = \partial_a \kappa^*_t(v_t) = \bar{\theta}_t$. For all $t \in [0,1]$, we set $\sigma_t=\sup\{s \leq t : s \in T\}$ and  $\tau_t = \inf\{s \geq t : s \in T\}$. If $\sigma$ and $\tau$ are finite then
\[
  \bar{\theta}_{\sigma_t} = \phi_{\sigma_t} = \phi_t  = \phi_{\tau_t} = \bar{\theta}_{\tau_t}.
\]
As $a$ is also a solution of \eqref{eqn:existence_max}, we have $\bar{\theta}_{\sigma_t} = \theta_t = \bar{\theta}_{\tau_t}$, therefore $\theta = \phi$. As a consequence, we have
\[
  a_t = \partial_\theta \kappa_t(\theta_t) = \partial_\theta \kappa_t(\phi_t) = b_t,
\]
which proves the uniqueness of the solution.

We now prove that $a$ and $\theta$ are Lipschitz functions. For all $t \in [0,1]$, $\int_0^t \kappa^*_s(a_s)ds \leq 0$, and $\int_0^t \kappa^*_s(a_s)d\theta^{-1}_s = 0$. In particular, this means that $\kappa^*_t(a_t)$ vanishes $d\theta^{-1}_t$ almost everywhere, thus $\theta_t = \bar{\theta}_t$ $d\theta^{-1}_t$ almost everywhere. By continuity of $\theta$ and $\bar{\theta}$, these functions are identical on $T$. In addition, for all $s<t$ such that $(s,t) \subset [0,1]\backslash T$, we have $\int_s^t d\theta^{-1}_u = 0$, hence $\theta_t=\theta_s$, which proves that $\theta$ is constant on $[0,1]\backslash T$. As a result, for all $s<t \in [0,1]$, we have $\theta_t=\theta_s$ if $(s,t) \subset [0,1]\backslash T$, otherwise
\[
  \theta_s = \inf_{u \geq s, u \in T} \bar{\theta}_u \quad \mathrm{and} \quad \theta_t = \inf_{u \leq t, u \in T} \bar{\theta}_u,
\]
In consequence $ |\theta_t - \theta_s| \leq \sup_{r,r' \in [s,t]} |\bar{\theta}_r - \bar{\theta}_{r'}|$. As $\bar{\theta}$ is $\calC^1$ on $[0,1]$, $\bar{\theta}$ and $\theta$ are Lipschitz functions. As $a_t = \partial_\theta \kappa_t(\theta_t)$, the function $a$ is also Lipschitz.

Finally, we prove the existence of a solution to \eqref{eqn:existence_max}. To do so, we reformulate this optimization problem in terms of an optimization problem for $\theta$. The aim is to find a positive function $\theta \in \calC$ such that
\begin{equation}
  \int_0^1 \partial_\theta \kappa_t(\theta_t) dt = \max\left\{ \int_0^1 \partial_\theta \kappa_t(\phi_t) dt : \phi \in \calC, \forall t \in [0,1], E(\phi)_t < +\infty \right\},
\end{equation}
where $E(\phi)_t = \int_0^t \phi_s \partial_\theta \kappa_t(\phi_t) - \kappa_t(\phi_t)$. By Theorem \ref{thm:lagrange}, if $\theta$ exists, then it is a non-decreasing function. Moreover $E(\theta)_1=0$ and for any $t \in [0,1]$, $\int_0^t E(\theta)_sd\theta^{-1}_s = 0$.

Using these three properties, we have $\theta=\bar{\theta}$ on the support of the measure $d\theta^{-1}$. Moreover, as $E_0(\theta)=E_1(\theta)=0$ and $E_t$ is non-positive, we observe that $E_t(\theta)$ is locally non-increasing in the neighbourhood of 0 and locally non-decreasing in the neighbourhood of 1, in particular
\[
  \theta_0 \partial_\theta \kappa_0(\theta_0) - \kappa_0(\theta_0) \leq 0 \quad \mathrm{and} \quad \theta_1 \partial_\theta \kappa_1(\theta_1) - \kappa_1(\theta_1) \geq 0.
\]
As for all $t \in [0,1]$, the function $\phi \mapsto \phi \partial_\theta \kappa_t(\phi) - \kappa_t(\phi)$ is increasing, we conclude that $\theta_0 \leq \bar{\theta}_0$ and $\theta_1 \geq \bar{\theta}_1$. As a consequence, $T = \{t \in [0,1] : \theta_t = \bar{\theta}_t\}$ is non-empty, and, setting $\sigma_t = \sup\{s \leq t : s \in T \}$ and $\tau_t = \inf\{s \geq t : s \in T\}$ we have $\theta_t = \bar{\theta}_{\sigma_t}$ if $\sigma_t > -\infty$ and $\theta_t = \bar{\theta}_{\tau_t}$ if $\tau_t < +\infty$.

We write 
\[\Theta = \left\{ \theta \in \calC : \theta \text{ non-decreasing}, \theta_0 \geq 0, \forall t \in [0,1], \int_0^t E_s(\theta)d\theta^{-1}_s = 0 \text{ and } E_t(\theta) \leq 0\right\}.\]
This set is uniformly equicontinuous and bounded, thus by Arzelà-Ascoli theorem, it is compact. It is non-empty as for all $\epsilon>0$ small enough, the function $t \mapsto \epsilon$ belongs to $\Theta$. We write $\theta$ a maximizer of $\int_0^1 \partial_\theta\kappa_s(\theta_s)ds$ on $\Theta$.

By continuity, if $E(\theta)_1<0$, then we can increase a little $\theta$ in the neighbourhood of $1$, thus $\theta$ is non-optimal. As a result, $\theta$ is non-decreasing, verifies $E(\theta)_1=0$ and $\int E(\theta)_s d\theta^{-1}_s=0$, which proves that $a = \partial_\theta \kappa(\theta)$ is a solution of \eqref{eqn:existence_max}.
\end{proof}

The previous proof gives some characteristics of the unique solution $a$ of \eqref{eqn:existence_max}. In particular, if we set $\theta_t = \partial_a \kappa^*_t(a_t)$, we know that $\theta$ is positive, non-decreasing, and that on the support of the measure $d\theta^{-1}$, $\theta$ and $\bar{\theta}$ are identical. Consequently, the optimal speed path of the branching random walk verifies the following property: while in the bulk of the branching random walk, it follows an equipotential line, and when close to the boundary it follows the natural speed path.

In certain specific cases, \eqref{eqn:existence_max} can be explicitly solved. For example, when $\bar{\theta}$ increases --which corresponds to the ``decreasing variance'' case in Gaussian settings-- we have $a=v$.
\begin{lemma}
\label{lem:specialcase}
We assume \eqref{eqn:breeding} and \eqref{eqn:regularity}.
\begin{description}
  \item[Non-decreasing case] If $\bar{\theta}$ is non-decreasing, then $a=v$ (and $\bar{\theta} = \theta$).
  \item[Non-increasing case] If $\bar{\theta}$ is non-increasing, then there exists $\theta \in [0,+\infty)$ such that $a_t = \partial_\theta \kappa_t(\theta)$.
  \item[Mixed case] If $\bar{\theta}$ is non-increasing on $[0,1/2]$ and increasing on $[1/2,1]$, then there exists $t \in [1/2,1]$ such that
  \[
    \forall s \in [0,1], \partial_a \kappa^*_s(a_s) = \bar{\theta}_{s \vee t}.
  \]
\end{description}
\end{lemma}

\begin{proof}
We first assume that $\bar{\theta}$ is a non-decreasing function. As $K^*(v)_t=0$ for all $t \in [0,1]$, we have $K^*(v)_1 = \int_0^t K^*(v)_t d \bar{\theta}^{-1}_t = 0$ which, by Proposition \ref{prop:regularity} implies that $v$ is the solution of \eqref{eqn:existence_max}.

We now denote by $a$ the solution of \eqref{eqn:existence_max}, and by $\theta_t = \partial_a \kappa^*_t(a_t)$. Let $T$ be the support of the measure of $d\theta^{-1}_t$, we know by Proposition \ref{prop:regularity} that $\theta$ is non-decreasing and equal to $\bar{\theta}_t$ when it is increasing. In particular, we have
\[
  \frac{1}{\theta_t} = \frac{1}{\theta_0} + \int_0^t d\theta^{-1}_s = \frac{1}{\theta_0} + \int_0^t \ind{s \in T} d\theta^{-1}_s = \frac{1}{\theta_0} + \int_0^t \ind{s \in T} d\bar{\theta}^{-1}_s.
\]

As a consequence, if $\bar{\theta}$ is non-increasing on $[0,t]$, then $\int_0^u \ind{s \in T} d\bar{\theta}^{-1}_s \geq 0$ for all $u \leq t$. As $\theta^{-1}$ is non-increasing, we conclude that $\int_0^u \ind{s \in T} d\bar{\theta}^{-1}_u = 0$, and $\theta_u = \theta_0$. In particular, in the non-increasing case, we conclude that $\theta$ is a constant.

In the mixed case, we have just shown that $\theta$ is constant up to time $1/2$. We set $u=\inf\{t > 1/2 : \bar{\theta}_t = \theta_0\}$. Since $\theta=\bar{\theta}$ on $T$, we know that $T \cap [1/2,u) = \emptyset$. Hence $\theta$ is constant up to point $u$. For $t>u$, as $\bar{\theta}$ increases, we have
\[
  \frac{1}{\theta_t} = \frac{1}{\theta_0} + \int_0^t \ind{s \in T} d\bar{\theta}^{-1}_s = \frac{1}{\bar{\theta}_u} + \int_{u}^t  \ind{s \in T} d\bar{\theta}^{-1}_s \geq \frac{1}{\bar{\theta}_t},
\]
which yields $\bar{\theta}_t \geq \theta_t$. We now observe that $K^*(a)_1=0$, thus $K^*(a)$ attains a local maximum at time 1, and its left derivative $\kappa^*_1(a_1)$ is non-negative. This implies that $\theta_1 \geq \bar{\theta}_1$. If there exists $s > u$ such that $\theta_s <\bar{\theta}_s$, then $T \cap [s,1] = \emptyset$, and $\theta_1 = \theta_s < \bar{\theta}_s \leq \bar{\theta}_1$, which contradicts the previous statement. In consequence, for $t \geq u$, we have $\theta_t = \bar{\theta}_t$, which ends the proof of the mixed case.
\end{proof}

\section{Maximal and consistent maximal displacements}
\label{sec:maxdis}

We apply the estimates obtained in the previous section to compute the asymptotic of some quantities of interest for the BRWtie. In Section \ref{subsec:maxdisplacement}, we take interest in the maximal displacement in a BRWtie with selection. In Section \ref{subsec:cmd}, we obtain a formula for the consistent maximal displacement with respect to a given path. If we apply these estimates in a particular case, we prove Theorems \ref{thm:main} and \ref{thm:cmd}.

\subsection{Maximal displacement in a branching random walk with selection}
\label{subsec:maxdisplacement}

We first define the \textit{maximal displacement in a branching random walk with selection}, which is the position of the rightmost individual among those alive at generation $n$ that stayed above a prescribed curve. We consider a positive function $\phi$ that satisfies \eqref{eqn:phibiendef} and \eqref{eqn:regularphi}. We introduce functions $b$ and $\sigma$ according to \eqref{eqn:meanandvariance}. Let $f$ be a continuous function on $[0,1]$ with $f(0)<0$, and $F$ be a Riemann-integrable subset of $[0,1]$. The set of individuals we consider is
\[
  \calW_n^\phi(f,F) = \left\{ |u|=n : \forall j \in F_n, V(u_j) \geq \bar{b}^{(n)}_j + f_{j/n} n^{1/3} \right\}.
\]
This set is the tree of the BRWtie with selection $(\calW_n^\phi(f,F), V_{|\calW_n^\phi (f,F)})$, in which every individual $u$ alive at time $k \in F_n$ with position below $\bar{b}^{(n)}_k + f_{k/n} n^{1/3}$ is immediately killed, as well as all its descendants. Its maximal displacement at time $n$ is denoted by
\[
  M^\phi_n(f,F) = \max\left\{ V(u), u \in \calW^\phi_n(f,F) \right\}.
\]

To apply the results of the previous section, we assume here that $b$ satisfies
\begin{equation}
  \label{eqn:brealpath}
  \sup_{t \in [0,1]} K^*(b)_t = 0 = K^*(b)_1;
\end{equation}
in other words, there exists individuals that follow the path with speed profile $b$ with positive probability, and at time 1, there are $e^{o(n)}$ of those individuals. We set $G$ the set of zeros of $K^*(b)$, and we assume that
\begin{equation}
  \label{eqn:pourG}
  G = \{ t \in [0,1] : K^*(b)_t = 0 \} \text{ is Riemann-integrable}.
\end{equation}
For $\lambda \in \R$, we set $g^\lambda \in \calC([0,t_\lambda))$ the function solution of
\begin{equation}
  \label{eqn:pourg}
  \forall t \in [0,t_\lambda), \phi_t g^\lambda_t - H_t^{F,G}(f,g^\lambda,\phi) = \phi_0 \lambda.
\end{equation}
To study $M^\phi_n(f,F)$, we first prove the existence of a unique maximal $t_\lambda \in (0,1]$, and a function $g^\lambda$ solution of \eqref{eqn:pourg}. We recall the following theorem of Carathéodory, that can be found in \cite{Fil88}.
\begin{theorem}[Existence and uniqueness of solutions of Carathéodory's ordinary differential equation]
\label{thm:existenceuniqueness}
Let $0 \leq t_1 < t_2 \leq 1$, $x_1 < x_2$, $M>0$ and $f : [t_1,t_2] \times [x_1,x_2] \to [-M,M]$ a bounded function. Let $t_0 \in [t_1,t_2]$ and $x_0 \in [x_1,x_2]$, we consider the differential equation consisting in finding $t > 0$ and a continuous function $\gamma : [t_0,t_0 + t] \to \R$ such that
\begin{equation}
  \label{eqn:differential}
  \forall s \in [t_0,t_0+t], \gamma(s) = x_0 + \int_{t_0}^s f(u,\gamma(u)) du.
\end{equation}

If for all $x \in [x_1,x_2]$, $t \mapsto f(t,x)$ is measurable and for all $t \in [t_1,t_2]$, $x \mapsto f(t,x)$ is continuous, then for all $(t_0,x_0)$ there exists $t\geq \min(t_2 - t_0, \frac{x_2-x_0}{M}, \frac{x_0-x_1}{M})$ and $\gamma$ that satisfy \eqref{eqn:differential}.

If additionally, there exists $L>0$ such that for all $x,y \in [x_1,x_2]$ and $t \in [t_1,t_2]$, $|f(t,x)-f(t,y)| \leq L |x-y|$, then for every pair of solutions $(t,\gamma)$ and $(\tilde{t},\tilde{\gamma})$ of \eqref{eqn:differential}, we have
\[
  \forall s \leq \min(t,\tilde{t}), \gamma(s) = \tilde{\gamma}(s).
\]
Consequently, there exists a unique solution defined on a maximal interval $[t_0,t_0+t_{\mathrm{max}}]$.
\end{theorem}

We use this theorem to prove there exists a unique solution $g$ to \eqref{eqn:pourg}.
\begin{lemma}
Let $f$ be a continuous function, $\phi$ that verifies \eqref{eqn:regularphi}, and $F,G$ two Riemann-integrable subsets of $[0,1]$. For all $\lambda > f_0$, there exists a unique $t_\lambda \in [0,1]$ and a unique continuous function defined on $[0,t_\lambda]$ such that for all $t < t_\lambda$, we have
\[
  g^\lambda_t > f_t \quad \mathrm{and} \quad \phi_t g^\lambda_t = \phi_0\lambda + K_s^{F,G}(f,g^\lambda,\phi).
\]
Moreover, there exists $\lambda_c$ such that for all $\lambda > \lambda_c$, $t_\lambda = 1$ and $\lambda \mapsto g^\lambda$ is continuous with respect to the uniform norm and strictly increasing.
\end{lemma}

\begin{proof}
Let $f$ be a continuous function, and $F$ be a Riemann-integrable subset of $[0,1]$, we set 
\[D = \{ (t,x) \in [0,1] \times \R : \text{if } t \in F,\text{ then } x > \phi_t f_t\},\]
and, for $(t,x) \in D$,
\begin{multline*}
  \Phi(t,x) = \frac{\dot{\phi}_t}{\phi_t} x + \mathbf{1}_{F \cap G}(t) \frac{\sigma_t^2}{(\frac{x}{\phi_t} - f_t)^2} \Psi\left( \tfrac{(\frac{x}{\phi_t}-f_t)^3}{\sigma_t^2} \dot{\phi}_t \right)\\
  +\mathbf{1}_{F^c \cap G}(t) \frac{a_1}{2^{1/3}} (\dot{\phi}_t \sigma_t)^{2/3}  + \mathbf{1}_{F \cap G^c} \left( \dot{\phi}_t(f_t - \tfrac{x}{\phi_t}) + \frac{a_1}{2^{1/3}} (-\dot{\phi}_t \sigma_t)^{2/3} \right).
\end{multline*}
For all $\lambda > f_0$, we introduce
\[
  \Gamma^\lambda = \left\{ (t,h), t \in [0,1], h \in \calC([0,t]) : \forall s \leq t, h_s = \phi_0 \lambda + \int_0^s \Phi(u,h_u)du \right\},
\]
the set of functions such that $g=\frac{h}{\phi}$ is a solution of \eqref{eqn:pourg}.

We observe that for all $[t_1,t_2] \times [x_1,x_2] \subset D$, $\Phi_{|[t_1,t_2] \times [x_1,x_2]}$ is measurable with respect to $t$, and uniformly Lipschitz with respect to $x$. As a consequence, by Theorem \ref{thm:existenceuniqueness}, for all $(t_0,x_0) \in D$, there exists $t>0$ such that there exists a unique function $h \in \calC([0,t])$ satisfying
\begin{equation}
  \label{eqn:differentialEquation}
  \forall s \leq t, h_s = x_0 + \int_{t_0}^s \Phi(u,h_u)du.
\end{equation}

Using this result, we first prove that $\Gamma^\lambda$ is a set of consistent functions. Indeed, let $(t_1,h^1)$ and $(t_2,h^2)$ be two elements of $\Gamma^\lambda$, and let $\tau = \inf\{ s \leq \min(t_1,t_2) : h^1_s \neq h^2_s \}$. We observe that if $\tau < \min(t_1,t_2)$, then by continuity of $h^1$ and $h^2$, we have $h^1_\tau = h^2_\tau$. Furthermore, $s \mapsto h^1_{\tau +s}$ and $s \mapsto h^2_{\tau+s}$ are two different functions satisfying \eqref{eqn:differentialEquation} with $t_0=\tau$ and $x_0 = h^1_\tau = h^2_\tau$, which contradicts the uniqueness of the solution. We conclude that $\tau \geq \min(t_1,t_2)$, every pair of functions in $\Gamma^\lambda$ are consistent up to the first terminal point.

We now define $t_\lambda = \max\left\{ t \in [0,1] : (t,h) \in \Gamma^\lambda\right\}$. We observe easily that $t_\lambda \in (0,1]$ by noting the existence of a local solution starting at time $0$ and position $\phi_0 \lambda$. For $s < t_\lambda$, we write $h^\lambda_s = h_s$, where $(s,h_s) \in \Gamma^\lambda$. By definition, for any $s < t_\lambda$, $h^\lambda_s = \phi_0 \lambda + \int_0^s \Phi(u,h^\lambda_u)du$. By local uniqueness of the solution, if there exists $t \in (0,1)$ such that $h^\lambda_t = h^{\lambda'}_t$, then for all $s \leq t$, $h^\lambda_s = h^{\lambda'}_s$, and in particular $\lambda = \lambda'$. We deduce that for all $\lambda < \lambda'$, if $s < \min(t_\lambda, t_\lambda')$ then $h^{\lambda}_s < h^{\lambda'}_s$.

Moreover, as there exist $C_1$ and $C_2>0$ such that for all $t \in [0,1]$ and $x > C_1$, $\Phi(t,x)< C_2$, we have $\limsup_{t \to t_\lambda} h^\lambda_t < +\infty$. Hence, if $\lambda < \lambda'$ and $t_\lambda > t_\lambda'$, if $x_0 \in \left[ \liminf_{t \to t_\lambda} h^\lambda_t, \limsup_{t \to t_\lambda} h^\lambda_t\right]$, then as $x_0 > h^\lambda_{t_{\lambda'}}$, we can extend $h^{\lambda'}$ on $[t_{\lambda'},t_{\lambda'}+\delta]$, which contradicts the fact that $t_{\lambda'}$ is maximal. We conclude that $t_{\lambda'} \geq t_\lambda$.

If $\lambda'>\lambda > \lambda_c$, the functions $h^{\lambda}$ and $h^{\lambda'}$ are defined on $[0,1]$. Moreover, the set
\[
  H^{\lambda,\lambda'} = \left\{ (t,x) \in [0,1] \times \R : x \in [h^\lambda_t, h^{\lambda'}_t] \right\},
\]
is a compact subset of $D$, that can be paved by a finite number of rectangles in $D$. As a consequence, there exists $L>0$ such that
\[
  \forall t \in [0,1], \forall x,x' : (t,x) \in H^{\lambda,\lambda'}, (t,x') \in H^{\lambda,\lambda'}, \left| \Phi(t,x) - \Phi(t,x') \right| \leq L |x-x'|.
\]
As for all $\mu \in [\lambda,\lambda']$, $(t,h^\mu_t) \in H^{\lambda, \lambda'}$, we observe that
\begin{align*}
  \left| h^{\mu}_t - h^{\mu'}_t \right|&\leq \left| \mu - \mu' \right| + \int_0^t \left| \Phi(s,h^{\mu}_s) - \Phi(s,h^{\mu'}_s) \right| ds \\
  &\leq \left| \mu - \mu' \right| + L \int_0^t \left| h^{\mu}_s - h^{\mu'}a_s \right|ds.
\end{align*}
Applying the Gronwall inequality, for all $\mu,\mu' \in [\lambda, \lambda']$, we have
\[
  \norme{h^{\mu} - h^{\mu'}}_\infty \leq  \left| \mu - {\mu'} \right| e^{L}
\]
which proves that $\lambda \mapsto h^\lambda$ is continuous with respect to the uniform norm.

Finally, there exist $C_0$ and $C_1>0$ such that for all $t \in [0,1]$ and $x \geq C_0$, we have $\Phi(t,x) \geq -C_1$. Therefore, for all $\lambda \geq C_0 + C_1 + \norme{\phi f}_\infty$, for all $t \in [0,1]$, $h^\lambda_t \geq \norme{\phi f}_\infty$, and $t_\lambda = 1$. We set $\lambda_c = \inf\{ \lambda \in \R : t_\lambda = 1\}$, and we conclude the proof by observing that $g^\lambda = \frac{h^\lambda}{\phi^\lambda}$ is the solution of \eqref{eqn:pourg}.
\end{proof}

\begin{lemma}
\label{lem:asymptoticGeneral}
Assuming \eqref{eqn:regularity}, \eqref{eqn:phibiendef}, \eqref{eqn:regularphi}, \eqref{eqn:FetG}, \eqref{eqn:integrability2phi} and \eqref{eqn:brealpath}, for any $\lambda > \max(0,\lambda_c)$, we have
\[
  \lim_{n \to +\infty} \frac{1}{n^{1/3}} \log \P\left( M_n^\phi(f,F) \geq \bar{b}^{(n)}_n + g^\lambda_1 n^{1/3} \right)= - \phi_0 \lambda.
\]
\end{lemma}

\begin{proof}
To obtain an upper bound, we recall that $1 \in G$, as $K^*(b)_1=0$ by \eqref{eqn:brealpath}. Let $\lambda > \max(0,\lambda_c)$, we set $g = g^\lambda$ the unique solution of \eqref{eqn:pourg}. We observe that
\[
  \P\left(M_n^\phi(f,F) \geq \bar{b}^{(n)}_n + g^\lambda_1 n^{1/3}\right) \leq \P\left(\calA_n^{F,G}(f,g) \neq \emptyset\right) \leq \E\left(\calA_n^{F,G}(f,g)\right).
\]
Therefore, by Lemma \ref{lem:estimate_upperfrontier}, we have
\[
  \limsup_{n \to +\infty} \frac{1}{n^{1/3}} \log \P(M_n^\phi(f,F) \geq \bar{b}^{(n)}_n + g^\lambda_1 n^{1/3}) \leq \sup_{t \in [0,1]} K^{F,G}_t(f,g,\phi) - \phi_t g_t = -\phi_0\lambda.
\]

When $\lambda'>\lambda$, we have $g^{\lambda'}_1 > g_1$, and
\[
  \P\left(M_n^\phi(f,F) \geq \bar{b}^{(n)}_n + g^\lambda_1 n^{1/3}\right) \geq \P\left(\calB_n^{F,G}(f,g^{\lambda'},g^{\lambda'}_1 - g_1) \neq \emptyset \right).
\]
Consequently, using Lemma \ref{lem:lowerbound}, we have
\[
  \liminf_{n \to +\infty} \frac{1}{n^{1/3}} \log \P\left(M_n^\phi(f,F) \geq \bar{b}^{(n)}_n + g^\lambda_1 n^{1/3}\right) \geq \sup_{t \in [0,1]} K^{F,G}_t(f,g,\phi) - \phi_t g_t = -\phi_0\lambda'.
\]
As $\lambda'$ decreases to $\lambda$, we have $\displaystyle \lim_{n \to + \infty} \frac{1}{n^{1/3}} \log \P\left(M_n^\phi(f,F) \geq \bar{b}^{(n)}_n + g^\lambda_1 n^{1/3}\right) = -\lambda$.
\end{proof}

The previous lemma gives an estimate of the right tail of $M_n^\phi(f,F)$ for all $f \in \calC$ and Riemann-integrable set $F \subset [0,1]$. Note that to obtain this estimate, we do not need the assumption \eqref{eqn:breeding} of supercritical reproduction, however \eqref{eqn:brealpath} implies that
\[
  \inf_{t \in [0,1]} \liminf_{n \to +\infty} \frac{1}{n} \log \E\left[ \# \{ u \in \T^{(n)}, |u|= \floor{tn} \} \right] \geq 0,
\]
which is a weaker supercriticality condition. Assuming \eqref{eqn:breeding}, we can strengthen Lemma \ref{lem:asymptoticGeneral} to prove a concentration estimate for $M_n^\phi(f,F)$ around $\bar{b}^{(n)}_n + g^0_1 n^{1/3}$.
\begin{lemma}[Concentration inequality]
\label{lem:concentration}
Under the assumptions \eqref{eqn:breeding}, \eqref{eqn:regularity}, \eqref{eqn:phibiendef}, \eqref{eqn:regularphi}, \eqref{eqn:FetG}, \eqref{eqn:integrability2phi} and \eqref{eqn:brealpath}, if $\lambda_c>0$, then for all $\epsilon>0$, we have
\[
  \limsup_{n \to +\infty} \frac{1}{n^{1/3}} \log \P\left( \left|M_n^\phi(f,F) - \bar{b}^{(n)}_n - g^0_1 n^{1/3}\right| \geq \epsilon n^{1/3}\right) <0.
\]
\end{lemma}

\begin{proof}
We set $g = g^0$ the solution of \eqref{eqn:pourg} for $\lambda = 0$. We observe that for all $\epsilon>0$ and $t \in [0,1]$, we have $H_t^{F,G}(f,g+\epsilon,\phi) - \phi_t (g_t + \epsilon) < 0$. Consequently, for any $\epsilon>0$, we have
\begin{align*}
  \limsup_{n \to +\infty} \frac{1}{n^{1/3}} \log \P\!\left( M_n^\phi(f,F) \geq \bar{b}^{(n)}_n + (g_1+\epsilon)n^{1/3} \right)
  &\leq \limsup_{n \to +\infty} \frac{1}{n^{1/3}} \log \P\!\left( \calA^{F,G\cup\{1\}}_n(f,g+\epsilon) \neq \emptyset  \right)\\
  &\leq \sup_{t \in [0,1]} H_t^{F,G}(f,g+\epsilon,\phi) - \phi_t (g_t+\epsilon)<0,
\end{align*}
by Lemma \ref{lem:estimate_upperfrontier}.

To obtain a lower bound, we need to strengthen the tail estimate of $M_n^\phi(f,F)$. Using \eqref{eqn:breeding}, the size of the population in the branching random walk increases at exponential rate. We set $p \in \R$ and $\rho>0$ such that $\rho = \inf_{t \in [0,1]} \P(\# \{ \ell \in L_t : \ell \geq p \} \geq 2)$. We can assume, without loss of generality, that $p < b_0$. We couple the BRWtie with a Galton-Watson process $(Z_n, n \geq 0)$ with $Z_0=1$, and reproduction law defined by $\P(Z_1 = 2) = 1 - \P(Z_1=1) = \rho$; in a way that
\[
  \forall n \in \N, \forall k \leq n, \# \{ u \in \T^{(n)} : |u|=k, V(u) \geq kp \} \geq Z_k.
\]
There exists $\alpha>0$ such that $\limsup_{n \to +\infty} \frac{1}{n} \log \P\left( Z_n \leq e^{\alpha n} \right) < 0$, by standard Galton-Watson theory (see, e.g. \cite{FlW07}). Consequently, with high probability, there are at least $e^{\alpha k}$ individuals to the right of $pk$ at any time $k \leq n$.

Let $\epsilon>0$ and $\eta > 0$, we set $k =\floor{\eta n^{1/3}}$. Applying the Markov property at time $k$, we have
\[
  \P\left(M_n \leq m_n - \epsilon n^{1/3} \right) \leq \P\left(Z_k \leq e^{\alpha k} \right) + \left[1 - \P_{k, kp}\left(M_{n-k} \geq m_n - \epsilon n^{1/3})\right) \right]^{e^{\alpha k}}.
\]
As a consequence
\begin{multline*}
  \limsup_{n \to +\infty} \frac{1}{n^{1/3}} \log \P\left( M_n \leq m_n - \epsilon n^{1/3} \right) \\
  \leq \max\left\{ \limsup_{n \to +\infty} \frac{1}{n^{1/3}} \log \P\left( Z_k \leq e^{\alpha k} \right), -\liminf_{n \to +\infty} \frac{e^{\alpha k}}{n^{1/3}} \P_{k,kp} \left( M_{n-k} \geq m_n - \epsilon n^{1/3} \right)\right\}.
\end{multline*}
We now prove that
\begin{equation}
  \label{eqn:final}
  \liminf_{n \to +\infty} e^{\alpha \eta n^{1/3}} \P_{k,0} \left( M_{n-k} \geq \bar{b}^{(n)}_n + (g_1-\epsilon) n^{1/3}-kp \right)>0.
\end{equation}

Let $\delta > 0$, we choose $\eta = \tfrac{\epsilon}{b_0-p} + \delta$, we have
\begin{multline*}
  \liminf_{n \to +\infty} \frac{1}{n^{1/3}} \log \P_{k,0} \left( M_{n-k} \geq \bar{b}^{(n)}_n + (g_1-\epsilon) n^{1/3}-kp \right)\\
  = \liminf_{n \to +\infty} \frac{1}{n^{1/3}} \log \P_{k,0} \left( M_{n-k} \geq \bar{b}^{(n)}_n - \bar{b}^{(n)}_k + g_1+\delta n^{1/3} \right) \leq  -\phi_0 \lambda_\delta,
\end{multline*}
by applying Lemma \ref{lem:asymptoticGeneral}, where $\lambda_\delta$ is the solution of the equation $g^{\lambda_\delta}_1=\delta$. Here, we implicitly used the fact that the estimate obtained in Lemma \ref{lem:asymptoticGeneral} is true uniformly in $k \in [0,\eta n^{1/3}]$. This is due to the fact that this is also true for Theorem \ref{thm:general_rw}. Finally, letting $\delta \to 0$, we have $\lambda_\delta \to 0$, hence
\[
  \liminf_{n \to +\infty}  \frac{e^{\alpha k}}{n^{1/3}} \P_{k,kp} \left( M_{n-k} \geq m_n - \epsilon n^{1/3} \right) = +\infty,
\]
which concludes the proof.
\end{proof}

\subsubsection*{Proof of Theorem \ref{thm:main}}

We denote by $a$ the solution of \eqref{eqn:existence_max} and by $\theta$ the function defined by $\theta_t = \partial_a \kappa^*_t(a_t)$. We assume that \eqref{eqn:regularity} is verified, i.e. $\theta$ is absolutely continuous with a Riemann-integrable derivative $\dot{h}$. For all $n \in \N$ and $k \leq n$, we set $\bar{a}^{(n)}_k = \sum_{j=1}^k a_{j/n}$. We recall that $l^* = \frac{\alpha_1}{2^{1/3}} \int_0^1 \frac{(\dot{\theta}_s \sigma_s)^{2/3}}{\theta_s} ds$, where $\alpha_1$ is the largest zero of the Airy function of first kind.

\begin{proof}[Proof of Theorem \ref{thm:main}]
We observe that with the previous notation, we have $M_n = M^\theta_n(-l^*-1,\emptyset)$. By Proposition \ref{prop:regularity}, $a$ satisfies \eqref{eqn:brealpath}, $\theta$ is non-decreasing and increases only on $G$. As a consequence, \eqref{eqn:FetG} is verified, and \eqref{eqn:differential} can be written, for $\lambda \in \R$,
\begin{equation}
  \label{eqn:differentialPourmax}
  \forall t \in [0,1], \theta_t g^\lambda_t = \theta_0 \lambda + \int_0^t \dot{\theta}_s g^\lambda_s + \frac{\alpha_1}{2^{1/3}} (\dot{\theta}_s \sigma_s)^{2/3} ds.
\end{equation}
By integration by parts, $g^\lambda_t = \lambda + \int_0^t \frac{\alpha_1}{2^{1/3}} (\dot{\theta}_s \sigma_s)^{2/3}ds$. In particular, $g^\lambda_1 = \lambda + l^*$. As a consequence, applying Lemma \ref{lem:asymptoticGeneral} to $\lambda = l$, we have
\[
  \lim_{n \to +\infty} \frac{1}{n^{1/3}} \log \P(M_n \geq \bar{a}^{(n)}_n + (l^*+l)n^{1/3}) = -\theta_0 l.
\]
Similarly, using Lemma \ref{lem:concentration}, for any $\epsilon>0$,
\[
  \limsup_{n \to +\infty} \frac{1}{n^{1/3}} \log \P\left( \left|M_n - \bar{a}^{(n)}_n - l^* n^{1/3} \right| \geq \epsilon n^{1/3} \right) < 0.
\]
As $a$ is Lipschitz, we have $\displaystyle \bar{a}^{(n)}_n = \sum_{j=1}^n a_{j/n} = n \int_0^1 a_s ds + O(1)$, concluding the proof.
\end{proof}

Mixing Lemma \ref{lem:specialcase} and Theorem \ref{thm:main}, we obtain explicit asymptotic for the maximal displacement, in some particular cases. If $\bar{\theta}$ is non-decreasing, then $\theta = \bar{\theta}$. As a result, setting
\[
  \bar{l}^* = \frac{\alpha_1}{2^{1/3}} \int_0^1 \frac{\left( \dot{\bar{\theta}}_s \sigma_s \right)^{2/3}}{\bar{\theta}_s}ds,
\]
we have $M_n = n \int_0^1 v_s ds + \bar{l}^* n^{1/3} + o(n^{1/3})$ in probability.

\begin{remark}
Let $\sigma \in \calC^2$ be a positive decreasing function. For $t \in [0,1]$, we define the point process $L_t = (\ell^1_t,\ell^2_t)$ with $\ell^1_t,\ell^2_t$ two i.i.d. centred Gaussian random variables with variance $\sigma_t$. We consider the BRWtie with environment $(\calL_t, t \in [0,1])$. We have $\bar{\theta}_t = \frac{\sqrt{2 \log 2}}{\sigma_t}$, which is increasing. Consequently, by Theorem \ref{thm:main} and Lemma \ref{lem:specialcase}
\[ M_n = n \sqrt{2 \log 2} \int_0^1 \sigma_s ds + n\frac{\alpha_1}{2^{1/3}(2 \log 2)^{1/6}} \int_0^1 (-\sigma'_s)^{2/3} \sigma_s^{1/3} ds,\]
which is consistent with the results obtained in \cite{MaZ13} and \cite{NRR13}.
\end{remark}

If $\bar{\theta}$ is non-increasing, then $\theta$ is constant. Applying Theorem \ref{thm:main}, we have $M_n = n v^* + o(n^{1/3})$.

\subsection{Consistent maximal displacement with respect to a given path}
\label{subsec:cmd}

Let $\phi$ be a continuous positive function, we write $b_t = \partial_\theta \kappa_t(\phi_t)$, and we assume that $b$ satisfies \eqref{eqn:brealpath}. We take interest in the consistent maximal displacement with respect to the path with speed profile $b$, defined by
\begin{equation}
  \label{eqn:defineCmdGeneral}
  \Lambda^\phi_n = \min_{|u|=n} \max_{k \leq n} \left( \bar{b}^{(n)}_k - V(u_k) \right).
\end{equation}
In other words, this is the smallest number such that, killing every individual at generation $k$ and in a position below $\bar{b}^{(n)}_k - \Lambda^\phi_n$, an individual remains alive until time $n$.

We set, for $u \in \T$, $\Lambda^\phi(u)= \max_{k \leq |u|} \left( \bar{b}^{(n)}_k - V(u_k) \right)$ the maximal delay of individual $u$. In particular, with the definition of Section \ref{sec:path}, for any $\mu \geq 0$, we have
\begin{equation}
  \label{eqn:important}
  M^\phi_n(-\mu, [0,1]) = \max\left\{ V(u), |u|=n, \Lambda^\phi(u) \leq \mu n^{1/3} \right\},
\end{equation}
in particular $M^\phi_n(-\mu, [0,1]) > - \infty \iff \Lambda^\phi_n \leq \mu n^{1/3}$.

For $\lambda,\mu > 0$, we denote by $g^{\lambda,\mu}$ the solution of
\begin{equation}
  \label{eqn:differentialCmd}
  \phi_t g_t = \phi_0\lambda + \int_0^t \dot{\phi}_s g_s + \ind{K^*(b)_s = 0} \frac{\sigma_s^2}{(g_s + \mu)^2} \Psi\left( \tfrac{(g_s + \mu)^3}{\sigma_s^2} \dot{\phi}_s \right) ds.
\end{equation}
Using the structure of this differential equation, for any $\lambda, \mu > 0$, we have $g^{\lambda,\mu}= g^{\lambda + \mu,0}- \mu$. Indeed, let $\lambda > 0$ and $\mu>0$, and let $g = g^{\lambda +\mu,0}-\mu$. By definition, the differential equation satisfied by $g+\mu$ is
\begin{align*}
  \phi_t (g_t + \mu) &= \phi_0( \lambda + \mu) + \int_0^t \dot{\phi}_s (g_s + \mu) +  \ind{K^*(b)_s = 0} \frac{\sigma_s^2}{(g_s + \mu)^2} \Psi\left( \tfrac{(g_s + \mu)^3}{\sigma_s^2} \dot{\phi}_s \right)ds\\
  \phi_t g_t &= \phi_0\lambda + \int_0^t \dot{\phi}_s g_s + \ind{K^*(b)_s = 0} \frac{\sigma_s^2}{(g_s + \mu)^2} \Psi\left( \tfrac{(g_s + \mu)^3}{\sigma_s^2} \dot{\phi}_s \right) ds,
\end{align*}
and by uniqueness of the solution of the equation, we have $g = g^{\lambda, \mu}$.

For $\lambda > 0$, we set $g^\lambda = g^{\lambda,0}$. We observe that if $\{\dot{\phi}_t > 0\} \subset \{ K^*(b)_t = 0\}$, then, for any $\lambda \geq 0$, $g^\lambda$ is a decreasing function. As $\lambda \mapsto g^\lambda$ is strictly increasing and continuous, there exists a unique non-negative $\lambda^*$ that verifies
\begin{equation}
  \label{eqn:definelambdastar}
  g^{\lambda^*}_1 = 0.
\end{equation}
Alternatively, $\lambda^*$ can be defined as $\tilde{g}_1/\phi_0$, where $\tilde{g}$ is the unique solution of the differential equation
\[
  \forall t \in [0,1),  \phi_t \tilde{g}_t = - \int_t^1 \dot{\phi}_s g_s + \ind{K^*(b)_s = 0} \frac{\sigma_s^2}{\tilde{g}_s^2} \Psi\left( \tfrac{\tilde{g}_s^3}{\sigma_s^2} \dot{\phi}_s \right) ds.
\]

\begin{lemma}[Asymptotic of the consistent maximal displacement]
\label{lem:asymptoticCmd}
Under the assumptions \eqref{eqn:regularity}, \eqref{eqn:phibiendef}, \eqref{eqn:regularphi}, \eqref{eqn:FetG}, \eqref{eqn:integrability2phi} and \eqref{eqn:brealpath}, for any $\lambda < \lambda^*$, we have
\[
  \lim_{n \to +\infty} \frac{1}{n^{1/3}} \log \P\left( \Lambda^\phi_n \leq (\lambda^*-\lambda) n^{1/3} \right)= - \phi_0\lambda.
\]
Moreover, for any $\epsilon>0$,
\[
  \limsup_{n \to +\infty} \frac{1}{n^{1/3}} \log \P\left( \left| \Lambda^\phi_n - \lambda^* n^{1/3} \right| \geq \epsilon n^{1/3} \right) < 0.
\]
\end{lemma}

\begin{proof}
Let $\lambda \in (0,\lambda^*)$, we set $g_t = g^{\lambda^*}_t$. Note first that $\Lambda^\phi_n \leq \lambda n^{1/3}$ if and only if there exists an individual $u$ alive at generation $n$ such that $\Lambda^\phi(u) \leq \lambda n^{1/3}$. To bound this quantity from above, we observe that such an individual either crosses $\bar{b}^{(n)}_{.} + n^{1/3} (g_{./n}-\lambda+\epsilon)$ at some time before $n$, or stays below this boundary until time $n$. Consequently, for any $\epsilon>0$, we have
\begin{multline*}
  \P\left( \calB_n^{[0,1],G}(-\lambda, g-\lambda + \epsilon, -\lambda \right)\\ \leq \P\left( \Lambda^\phi_n \leq \lambda n^{1/3} \right) \leq \P\left(\calA_n^{[0,1],G}(-\lambda,g-\lambda+\epsilon) \neq \emptyset\right) + \P\left( \calB_n^{[0,1],G}(-\lambda, g-\lambda + \epsilon, -\lambda \right).
\end{multline*}
Using Lemma \ref{lem:estimate_upperfrontier}, Lemma \ref{lem:firstorder} and Lemma \ref{lem:lowerbound}, and letting $\epsilon \to 0$, we conclude
\begin{equation*}
  \lim_{n \to +\infty} \frac{1}{n^{1/3}} \log \P\left( \Lambda^\phi_n \leq \lambda n^{1/3} \right) = -\phi_0 (\lambda^* - \lambda).
\end{equation*}

Finally, to bound $\P(\Lambda^\phi_n \geq (\lambda^*+\epsilon)n^{1/3})$ we apply \eqref{eqn:important}, and we get
\[
  \P(\Lambda^\phi_n \geq (\lambda^*+\epsilon)n^{1/3}) = \P(M_n(-\lambda^*-\epsilon,[0,1]) = -\infty).
\]
By Lemma \ref{lem:concentration}, we conclude that $\displaystyle \limsup_{n \to +\infty} \frac{1}{n^{1/3}} \log \P(\Lambda^\phi_n \geq (\lambda^*+\epsilon)n^{1/3}) < 0$.
\end{proof}

\subsubsection*{Proof of Theorem \ref{thm:cmd}}

We now prove Theorem \ref{thm:cmd}, applying Lemma \ref{lem:asymptoticCmd} to $\Lambda_n = \Lambda_n^\theta$.
\begin{proof}[Proof of Theorem \ref{thm:cmd}]
We denote by $G = \{ t \in [0,1] : K^*(a)_t=0\}$. For $\lambda > 0$, we write $g^\lambda$ for the solution of
\begin{equation}
   \label{eqn:lastdifferentialCmd}
  \theta_t g_t = \theta_0\lambda + \int_0^t \dot{\theta}_s g_s + \mathbf{1}_G(s) \frac{\sigma_s^2}{g_s^2} \Psi\left( \tfrac{g_s^3}{\sigma_s^2} \dot{\theta}_s \right) ds,
\end{equation}
and $\lambda^*$ for the unique non-negative real number that verifies $g^{\lambda^*}_1 = 0$. By Proposition \ref{prop:regularity}, $a$ satisfies \eqref{eqn:brealpath} and $\{\theta_t > 0\} \subset \{ K^*(b)_t = 0 \}$. Applying Lemma \ref{lem:asymptoticCmd}, for any $\lambda \in (0,\lambda^*)$, we have $\lim_{n \to +\infty} \frac{1}{n^{1/3}} \log \P\left[ \Lambda_n^\theta \leq (\lambda^* - \lambda)n^{1/3} \right] = - \theta_0 \lambda$. Moreover, for any $\epsilon>0$, we have $\limsup_{n \to +\infty} \frac{1}{n^{1/3}} \log \P\left[ \left| \Lambda_n^\phi - \lambda^* n^{1/3} \right| > \epsilon n^{1/3} \right] < 0$.
\end{proof}

In a similar way, we can compute the consistent maximal displacement with respect to the path with speed profile $v$, which is $\Lambda^{\bar{\theta}}_n$. We denote by $\bar{g}^\lambda$ the solution of the equation
\[
  \bar{\theta}_t g_t = \bar{\theta}_0\lambda + \int_0^t \dot{\bar{\theta}}_s g_s + \frac{\sigma_s^2}{g_s^2} \Psi\left( \tfrac{g_s^3}{\sigma_s^2} \dot{\bar{\theta}}_s \right) ds,
\]
and by $\bar{\lambda}^*$ the solution of $g^{\bar{\lambda}^*}_1 = 0$. By Lemma \ref{lem:asymptoticCmd}, for all $0\leq l \leq \bar{\lambda}^*$,
\[
  \lim_{n \to +\infty} \frac{1}{n^{1/3}} \log \P\left( \Lambda^\phi_n \leq (\lambda^*-\lambda) n^{1/3} \right)= - \phi_0\lambda,
\]
and for all $\epsilon>0$, we have
\[
  \limsup_{n \to +\infty} \frac{1}{n^{1/3}} \log \P\left[ \left| \Lambda_n^{\bar{\theta}} - \bar{\lambda}^* n^{1/3} \right| > \epsilon n^{1/3} \right]<0.
\]

\subsubsection*{Consistent maximal displacement of the time-homogeneous branching random walk}

We consider $(\T,V)$ a time-homogeneous branching random walk, with reproduction law $\calL$. We denote by $\kappa$ the Laplace transform of $\calL$. The optimal speed profile is a constant $v = \inf_{\theta > 0} \frac{\kappa(\theta)}{\theta}$, and we set $\theta^* = \kappa'(\theta^*)$ and $\sigma^2 = \kappa''(\theta^*)$. The equation \eqref{eqn:lastdifferentialCmd} can be written in the simpler form
\[
  \theta^* g^\lambda_t = \theta^* \lambda + \int_0^t \frac{\sigma^2}{\left(g^\lambda_s\right)^2} \Psi(0)ds.
\]
As $\Psi(0) = - \frac{\pi^2}{2}$, the solution of this differential equation is $g^\lambda_t = \left( \lambda^3 - t\frac{3\pi^2 \sigma^2}{2 \theta^*} \right)^{1/3}$.

For $\Lambda_n = \min_{|u| = n} \max_{k \leq n} \left( kv - V(u_k)\right)$, applying Theorem \ref{thm:cmd}, and the Borel-Cantelli lemma, we obtain
\[
  \lim_{n \to +\infty} \frac{\Lambda_n}{n^{1/3}} = \left( \frac{3\pi^2 \sigma^2}{2 \theta^*} \right)^{1/3} \quad \mathrm{a.s.}
\]
This result is similar to the one obtained in \cite{FaZ10} and \cite{FHS12}.

More generally, if $(\T,V)$ is a BRWtie such that $\bar{\theta}$ is non-increasing, then $\theta$ is a constant, and for all $\epsilon>0$,
\[
  \limsup_{n \to +\infty} \frac{1}{n^{1/3}} \log \P\left[ \left|\Lambda_n - \left( \frac{3\pi^2 \sigma^2}{2 \theta} \int_0^1 \ind{K^*(a)_s = 0} ds\right)^{1/3} n^{1/3} \right| \geq \epsilon n^{1/3} \right].
\]

\appendix

\section{Airy facts and Brownian motion estimates}
\label{app:bm}

In Section \ref{subsec:onesided}, using some Airy functions --introduced in Section \ref{subsec:introrw}-- facts, the Feynman-Kac formula and PDE analysis, we compute the asymptotic of the Laplace transform of the area under a Brownian motion constrained to stay positive, proving Lemma \ref{lem:bmOnesided}. Adding some Sturm-Liouville theory, we obtain by similar arguments Lemma \ref{lem:bmTwosided} in Section \ref{subsec:twosided}. In all this section, $B$ stands for a standard Brownian motion, starting from $x$ under law $\P_x$.

\subsection{Asymptotic of the Laplace transform of the area under a Brownian motion constrained to stay non-negative}
\label{subsec:onesided}

In this section, we write $\rmL^2 = \rmL^2([0,+\infty))$ for the set of square-integrable measurable functions on $[0,+\infty)$. This space $\rmL^2$ can be seen as a Hilbert space, when equipped with the scalar product
\[
  \crochet{f,g} = \int_0^{+\infty} f(x) g(x) dx.
\]
We denote by $\calC^2_0=\calC^2_0([0,+\infty))$ the set of twice differentiable functions $w$ with a continuous second derivative, such that $w(0) = \lim_{x \to +\infty} w(x) = 0$. Finally, for any continuous function $w$, $\norme{w}_\infty = \sup_{x \geq 0} |w(x)|$. The main result of the section is: \textit{for all $h>0$, $0<a<b$ and $0<a'<b'$, we have}
\begin{multline}
  \label{eqn:bmOnesidedRappel}
  \lim_{t \to +\infty} \frac{1}{t} \log \sup_{x \in \R} \E_x\left[ e^{-h \int_0^t B_s ds} ; B_s \geq 0, s \leq t \right]\\
  = \lim_{t \to +\infty} \frac{1}{t} \log \inf_{x \in [a,b]} \E_x\left[ e^{-h \int_0^t B_s ds} \ind{B_t \in [a',b']} ; B_s \geq 0, s \leq t  \right] = \frac{\alpha_1}{2^{1/3}}h^{2/3}.
\end{multline}

We recall that $(\alpha_n, n \in \N)$ is the set of zeros of $\mathrm{Ai}$, listed in the decreasing order. We start with some results on the Airy function $\mathrm{Ai}$, defined in \eqref{eqn:airy}.
\begin{lemma}
\label{lem:definepsin}
For $n \in \N$ and $x \geq 0$, we set
\begin{equation}
  \label{eqn:definepsin}
  \psi_n(x) = \mathrm{Ai}(x + \alpha_n)\left(\int_{\alpha_n}^{+\infty} \mathrm{Ai}(y)dy \right)^{-1/2}.
\end{equation}

The following properties hold:
\begin{itemize}
  \item $(\psi_n, n \in \N)$ forms an orthogonal basis of $\rmL^2$;
  \item $\lim_{n \to +\infty} \alpha_n n^{-2/3} = -\frac{3\pi}{2}$;
  \item for all $\lambda \in \R$ and $\psi \in \calC^2$, if
\begin{equation}
  \label{eqn:sturmliouvilleOnesided}
  \begin{cases}
    \forall x > 0, \psi''(x) - x \psi(x) = \lambda \psi(x)\\
    \psi(0) = \lim_{x \to +\infty} \psi(x) = 0,
  \end{cases}
\end{equation}
then either $\psi=0$, or there exist $n \in \N$ and $c \in \R$ such that $\lambda = \alpha_n$ and $\psi = c \psi_n$.
\end{itemize}
\end{lemma}

\begin{proof}
The fact that $\lim_{n \to +\infty} \alpha_n n^{-2/3}=-\frac{3\pi}{2}$ and that $(\psi_n, n \in \N)$ is an orthogonal basis of $\rmL^2$ can be found in \cite{VaS04}. We now consider $(\lambda,\psi)$ a solution of \eqref{eqn:sturmliouvilleOnesided}. In particular $\psi$ verifies
\[
  \forall x > 0, \psi''(x) - (x+\lambda) \psi(x) = 0.
\]
By definition of $\mathrm{Ai}$ and $\mathrm{Bi}$, there exist $c_1,c_2$ such that $\psi(x) = c_1 \mathrm{Ai}(x+\lambda) + c_2 \mathrm{Bi}(x+\lambda).$ As $\lim_{x\to +\infty} \psi(x) = 0$, we have $c_2=0$, and as $\psi(0)=0$, either $c_1=0$, or $\mathrm{Ai}(\lambda)=0$. We conclude that either $\psi=0$, or $\lambda$ is a zero of $\mathrm{Ai}$, in which case $\psi(x) = c_1 \psi_n(x)$ for some $n \in \N$.
\end{proof}
As $\alpha_1$ is the largest zero of $\mathrm{Ai}$, note that the eigenfunction $\psi_1$ corresponding to the largest eigenvector $\alpha_1$ is non-negative on $[0,+\infty)$, and is positive on $(0,+\infty)$.

For $h>0$ and $n \in \N$, we define $\psi_n^h : x \mapsto (2h)^{1/6} \psi_n((2h)^{1/3}x)$. By Lemma \ref{lem:definepsin}, $(\psi_n^h, n \in \N)$ forms an orthonormal basis of $\rmL^2$. With this lemma, we can prove the following preliminary result.
\begin{lemma}
\label{lem:asymptoticOnesided}
Let $h>0$ and $u_0 \in \calC^2_0 \cap \rmL^2$, such that $u_0', u_0'' \in \rmL^2$ and $\norme{u_0''}_\infty < +\infty$. We define, for $t \geq 0$ and $x \geq 0$
\[
  u(t,x) = \E_x\left[ u_0(B_t) e^{-h \int_0^t B_s ds} ; B_s \geq 0, s \in [0,t] \right].
\]
We have
\begin{equation}
  \label{eqn:asymptoticOnesided}
  \lim_{t \to +\infty} \sup_{x \in \R} \left| e^{-\frac{h^{2/3}}{2^{1/3}} \alpha_1 t}u(t,x) - \crochet{u_0,\psi^h_1} \psi^h_1(x) \right| = 0.
\end{equation}
\end{lemma}

\begin{proof}
Let $h>0$, by the Feynman-Kac formula (see e.g. \cite{KaS91}, Theorem 5.7.6), $u$ is the unique solution of the equation
\begin{equation}
  \label{eqn:fenymankacOnesided}
  \begin{cases}
    \forall t > 0, \forall x > 0, \partial_t u(t,x) = \frac{1}{2} \partial^2_{x} u(t,x) - h x u(t,x)\\
    \forall x \geq 0, u(0,x) = u_0(x)\\
    \forall t \geq 0, u(t,0) = \lim_{x \to +\infty} u(t,x) = 0.
  \end{cases}
\end{equation}

We define the operator
\[
  \calG^h : \begin{array}{rcl}
  \calC^2_0 & \longrightarrow & \calC\\
  w & \longmapsto & \left( x \mapsto \frac{1}{2} w''(x) - h x w(x), x \in [0,+\infty) \right),
  \end{array}
\]
By definition of $\mathrm{Ai}$ and of the $\psi_n^h$, we have $\calG^h \psi_n^h = \frac{h^{2/3}}{2^{1/3}} \alpha_n \psi_n^h$, thus $(\psi_n^h)$ forms an orthogonal basis of eigenfunctions of $\calG^h$.

We recall there exists $C>0$ such that for all $x \geq 0$, $\mathrm{Ai}(x) + \mathrm{Ai}'(x) \leq C z^{1/4} e^{-2x^{2/3}/3}$ (see e.g. \cite{VaS04}). For all $w \in \calC^2_0 \cap \rmL^2$ such that $w'$ and $w''$ are bounded, by integration by parts
\begin{align*}
  \crochet{\calG^h w, \psi_n^h} 
  &= \frac{1}{2} \int_0^{+\infty} w''(x) \psi_n^h(x) dx  - h \int_0^{+\infty} x w(x) \psi_n^h(x) dx\\
  &= \frac{1}{2} \int_0^{+\infty} w(x) (\psi_n^h)''(x) dx - h \int_0^{+\infty} x w(x) \psi_n^h(x) dx\\
  &= \int_0^{+\infty} w(x) (\calG^h \psi_n^h)(x) dx = \frac{h^{2/3}}{2^{1/3}} \alpha_n \crochet{w,\psi_n^h}.
\end{align*}
Therefore, decomposing $w$ with respect to the basis $(\psi_n^h)$, we have
\[
  \crochet{\calG^h w, w} =  \crochet{\calG^h w, \sum_{n = 1}^{+\infty} \crochet{\psi_n^h, w} \psi_n^h} = \sum_{n = 1}^{+\infty}   \crochet{w, \psi_n^h}\crochet{\calG^h w,\psi_n^h}
  = \frac{h^{2/3}}{2^{1/3}} \sum_{n = 1}^{+\infty} \alpha_n \crochet{w, \psi_n^h}^2.
\]
As $(\alpha_n)$ is a decreasing sequence, we have
\begin{equation}
  \label{eqn:prems}
  \crochet{\calG^h w, w} \leq \frac{h^{2/3}}{2^{1/3}}\sum_{n = 1}^{+\infty} \alpha_1 \crochet{w, \psi_n^h}^2 \leq \frac{h^{2/3}}{2^{1/3}} \alpha_1 \crochet{w,w}.
\end{equation}
If $\crochet{w,\psi_n^h}=0$, the inequality can be improved in
\begin{equation}
  \label{eqn:deuz}
  \crochet{\calG^h w, w} \leq \frac{h^{2/3}}{2^{1/3}} \sum_{n=2}^{+\infty}  \alpha_2 \crochet{w, \psi_n^h}^2 \leq \frac{h^{2/3}}{2^{1/3}} \alpha_2 \crochet{w,w}.
\end{equation}

Using these results, we now prove \eqref{eqn:asymptoticOnesided}. For $x \geq 0$ and $t \geq 0$, we define
\[
  v(t,x) = e^{- \frac{h^{2/3}}{2^{1/3}} \alpha_1 t} u(t,x) - \crochet{u_0,\psi^h_1} \psi_1^h.
\]
We observe first that for all $t \geq 0$, $\crochet{v(t,\cdot), \psi_1^h} = 0$. Indeed, we have $\crochet{v(0,\cdot),\psi_1^h} = 0$ by definition, and deriving with respect to $t$, we have
\begin{align*}
  \partial_t \crochet{v(t,\cdot),\psi_1^h}
  &= - \frac{h^{2/3}}{2^{1/3}} \alpha_1 e^{- \frac{h^{2/3}}{2^{1/3}} \alpha_1 t} \crochet{u(t,x), \psi^h_1} + e^{- \frac{h^{2/3}}{2^{1/3}} \alpha_1 t} \crochet{\partial_t u(t,x), \psi_1^h} \\
  &= - \frac{h^{2/3}}{2^{1/3}} \alpha_1 e^{- \frac{h^{2/3}}{2^{1/3}} \alpha_1 t} \crochet{u(t,x), \psi^h_1} + e^{- \frac{h^{2/3}}{2^{1/3}} \alpha_1 t} \crochet{\calG^h u(t,x), \psi_1^h}\\
  &= - \frac{h^{2/3}}{2^{1/3}} \alpha_1 e^{- \frac{h^{2/3}}{2^{1/3}} \alpha_1 t} \crochet{u(t,x), \psi^h_1} + \frac{h^{2/3}}{2^{1/3}} \alpha_1 e^{- \frac{h^{2/3}}{2^{1/3}} \alpha_1 t} \crochet{u(t,x), \psi_1^h} = 0.
\end{align*}

We now prove that the non-negative, finite functions
\[
  J_1(t) = \int_0^{+\infty} |v(t,x)|^2 dx \quad \mathrm{and} \quad J_2(t) = \int_0^{+\infty} |\partial_x v(t,x)|^2 dx,
\]
are decreasing, and converge to 0 as $t \to +\infty$. We observe first that
\begin{align*}
  \frac{\partial_t J_1(t)}{2} &= \int_0^{+\infty} v(t,x) \partial_t v(t,x) dx\\
  &= \int_0^{+\infty} v(t,x) \left[- \frac{h^{2/3}}{2^{1/3}} \alpha_1 e^{- \frac{h^{2/3}}{2^{1/3}} \alpha_1 t} u(t,x) + e^{- \frac{h^{2/3}}{2^{1/3}} \alpha_1 t} \partial_t u(t,x)  \right] dx\\
  &= \int_0^{+\infty} v(t,x) \left[- \frac{h^{2/3}}{2^{1/3}} \alpha_1 v(t,x) + \calG^h v(t,x)  \right] dx\\
  &= - \frac{h^{2/3}}{2^{1/3}} \alpha_1 \crochet{v(t,\cdot),v(t,\cdot)} + \crochet{v(t,\cdot),\calG^h v(t,\cdot)}.
\end{align*}
As $\crochet{v(t,\cdot),\psi_1^h} = 0$, by \eqref{eqn:deuz} we have $\partial_t J_1(t) \leq (2h)^{2/3} (\alpha_2 - \alpha_1) J_1(t) \leq -c J_1(t)$. By Grönwall inequality, $J_1(t)$ decreases to $0$ as $t \to +\infty$ exponentially fast. Similarly $J_2(0)<+\infty$ and
\begin{align*}
  \frac{\partial_t J_2(t)}{2}
  &= \int_0^{+\infty} \partial_x v(t,x) \partial_t \partial_x v(t,x) dx = \int_0^{+\infty} \partial_x v(t,x) \partial_x \partial_t v(t,x) dx\\
  &= \crochet{\partial_x v(t,\cdot), \calG^h \partial_x v(t,\cdot)} - \frac{h^{2/3}}{2^{1/3}}  \alpha_1  \crochet{\partial_x v(t,\cdot), \partial_x v(t,\cdot)}  - h \int_0^{+\infty} v(t,x) \partial_x v(t,x) dx \leq 0,
\end{align*}
by integration by parts and \eqref{eqn:prems}. As $J_2$ is non-negative and decreasing, this function converges, as $t \to +\infty$, to $J_2(+\infty) \geq 0$. Moreover, we can write the derivative of $J_1$ as follows
\begin{align*}
  \frac{\partial_t J_1(t)}{2}
  &= -\frac{h^{2/3}}{2^{1/3}} \alpha_1 \crochet{v(t,\cdot),v(t,\cdot)} + \crochet{v(t,\cdot),\calG^h v(t,\cdot)}\\
  &= - \frac{h^{2/3}}{2^{1/3}} \alpha_1 J_1(t) +  \int_0^{+\infty} v(t,x) \partial^2_{x} v(t,x) dx - \int_0^{+\infty} h x v(t,x)^2 dx\\
  &\leq -\frac{J_2(t)}{2} - \frac{h^{2/3}}{2^{1/3}} \alpha_1 J_1(t).
\end{align*}
As $J_1(t)$ decreases toward 0, its derivative cannot remain bounded away from $0$ for large $t$. Therefore, $\lim_{t \to + \infty} J_2(t) = 0$. We conclude that $\lim_{t \to +\infty} \int_0^{+\infty} |v(t,x)|^2 + |\partial_x v(t,x)|^2 dx = 0$, which means that $v(t,\cdot)$ converges to 0 in $H^1$ norm, as $t \to +\infty$. By Sobolev injection in dimension 1, there exists $C>0$ such that
\[
  \norme{v(t,\cdot)}_{\infty} \leq C \int_0^{+\infty} |v(t,x)|^2 + |\partial_x v(t,x)|^2 dx,
\]
which proves \eqref{eqn:asymptoticOnesided}.
\end{proof}

This lemma can be easily extended to authorize any bounded starting function $u_0$.
\begin{corollary}
\label{cor:asymptoticOnesided}
Let $h>0$ and $u_0$ be a measurable bounded function. Setting, for $x \geq 0$ and $t \geq 0$
\[
  u(t,x) = \E_x\left[ u_0(B_t) e^{-h \int_0^t B_s ds} ; B_s \geq 0, s \in [0,t] \right],
\]
we have
\begin{equation}
  \label{eqn:asymptoticOnesidedprime}
  \lim_{t \to +\infty} \sup_{x \in \R} \left| e^{-\frac{h^{2/3}}{2^{1/3}} \alpha_1  t}u(t,x) - \crochet{u_0,\psi_1^h} \psi^h_1(x) \right| = 0.
\end{equation}
\end{corollary}

\begin{proof}
Let $u_0$ be a measurable bounded function. We introduce, for $x \geq 0$ and $\epsilon>0$
\[
  u_\epsilon(x) = u(\epsilon,x) = \E_x \left[ u_0(B_\epsilon) e^{-h\int_0^{\epsilon} B_s ds} ; B_s \geq 0, s \in [0,1] \right].
\]
Observe that by the Markov property, for all $t \geq \epsilon$, we have
\begin{align*}
  u(t,x)
  &= \E_x \left[ \E_{B_{t-\epsilon}}\left[ u_0(B_\epsilon) e^{-h \int_0^\epsilon B_s ds} ; B_s \geq 0, s \in [0,\epsilon] \right] e^{-h\int_0^{t-\epsilon} B_s ds} ; B_s \geq 0, s \in [0,t-\epsilon] \right]\\
  &= \E_x \left[ u_\epsilon(B_{t-\epsilon}) e^{-h\int_0^{t-\epsilon} B_s ds} ; s \in [0,t-\epsilon] \right].
\end{align*}
Therefore, $u(t,x) = u_\epsilon(t-\epsilon,x)$, wher $u_\epsilon(t,x) = \E_x \left[ u_\epsilon(B_{t}) e^{-h\int_0^{t} B_s ds} ; s \in [0,t] \right]$. As $\int_0^\epsilon B_s ds$ is, under the law $\P_x$, a Gaussian random variable with mean $\epsilon x$ and variance $\epsilon^3/3$, we have
\[
  |u_\epsilon(x)| \leq \norme{u_0}_\infty \E_x \left[ e^{-h\int_0^{\epsilon} B_s ds} \right] \leq \norme{u_0}_\infty e^{-\epsilon h x} e^{\frac{h^2 \epsilon^3}{6}}.
\]
Moreover, as $h>0$, by the Ballot lemma, $|u_\epsilon(x)| \leq \norme{u_0}_\infty \P_x \left[ B_s \geq 0, s \in [0,\epsilon] \right] \leq C\epsilon^{-1/2} x$.

For any $\epsilon>0$, there exists $C>0$ such that for all $x \geq 0$,  $u_\epsilon(x) \leq C x\wedge e^{-hx\epsilon}$. Therefore, we can find sequences $(v^{(n)})$ and $(w^{(n)})$ of functions in $\calC^2_0 \cap \rmL^2$, such that $(v^{(n)})', (v^{(n)})'',(w^{(n)})'$ and $(w^{(n)})''$ are in $\rmL^2$, with bounded second derivatives verifying
\[
  w^{(n)} \leq u_\epsilon \leq w^{(n)} + \frac{1}{n} \quad \mathrm{and} \quad v^{(n)}-\frac{1}{n} \leq u_\epsilon \leq v^{(n)}.
\]

For $n \in \N$, $x \geq 0$ and $t \geq 0$, we set
\begin{multline*}
  v^{(n)}(t,x) = \E_x\left[ v^{(n)}(B_t)e^{-\int_0^t B_s ds} ; B_s \geq 0, s \in [0,t] \right] \quad \mathrm{and}\\
  w^{(n)}(t,x) = \E_x\left[ w^{(n)}(B_t)e^{-\int_0^t B_s ds} ; B_s \geq 0, s \in [0,t] \right].
\end{multline*}
Note that for all $x \geq 0$ and $t \geq 0$ we have $w^{(n)}(t,x) \leq u_\epsilon(t,x) \leq v^{(n)}(t,x)$. Moreover, by Lemma \ref{lem:asymptoticOnesided}, we have
\begin{align*}
  &\lim_{t \to +\infty} \sup_{x \in \R} \left| e^{-\frac{h^{2/3}}{2^{1/3}} \alpha_1 t}v^{(n)}(t,x) - \crochet{v^{(n)},\psi_1^h} \psi^h_1(x) \right| = 0,\\
  \mathrm{and} \quad & \lim_{t \to +\infty} \sup_{x \in \R} \left| e^{-\frac{h^{2/3}}{2^{1/3}} \alpha_1 t}w^{(n)}(t,x) - \crochet{w^{(n)},\psi_1^h} \psi^h_1(x) \right| = 0.
\end{align*}

By the dominated convergence theorem, we have
\[
  \lim_{n \to +\infty} \crochet{w^{(n)},\psi_1^h} = \lim_{n \to +\infty} \crochet{v^{(n)},\psi_1^h} = \crochet{u_\epsilon, \psi_1^h}.
\]
As a result, letting $t$, then $n \to +\infty$, this yields
\[
  \lim_{t \to +\infty} \sup_{x \in \R} \left| e^{-\frac{h^{2/3}}{2^{1/3}} \alpha_1 t}u_\epsilon(t,x) - \crochet{u_\epsilon,\psi_1^h} \psi^h_1(x) \right| = 0.
\]
Finally, for almost every $x \geq 0$, letting $\epsilon \to 0$,  we have $u_\epsilon(x) \to u_0(x)$, and thus by dominated convergence theorem again, $\lim_{\epsilon \to 0} \crochet{u_\epsilon,\psi_1^h} = \crochet{u_0,\psi_1^h}.$ Moreover, as $u(t,x) = u_\epsilon(t-\epsilon,x)$, we conclude that
\[
  \lim_{t \to +\infty} \sup_{x \in \R} \left| e^{-\frac{h^{2/3}}{2^{1/3}} \alpha_1 t}u(t,x) - \crochet{u_0,\psi_1^h} \psi^h_1(x) \right| = 0.
\]
\end{proof}

This last corollary is enough to prove the exponential decay of the Laplace transform of the area under a Brownian motion constrained to stay positive.
\begin{proof}[Proof of Lemma \ref{lem:bmOnesided}]
Let $h>0$, for $t, x \geq 0$ we set $u(t,x) = \E_x\left[ e^{-h\int_0^t B_s ds} ; B_s \geq 0, s \in [0,t] \right]$ and $\mu_h = \int_0^{+\infty} \psi_1^h(x)dx<+\infty$. By Corollary \ref{cor:asymptoticOnesided}, we have
\[
  \lim_{t \to +\infty} \sup_{x \in [0,+\infty)} \left| e^{-\frac{h^{2/3}}{2^{1/3}} \alpha_1 t}u(t,x) - \mu_h \psi^h_1(x) \right| = 0.
\]
As $\psi^h_1$ is bounded, we have
\begin{equation}
  \label{eqn:no1}
  \limsup_{t \to +\infty} \sup_{x \geq 0} \frac{1}{t} \log u(t,x) = \frac{h^{2/3}}{2^{1/3}} \alpha_1.
\end{equation}

Similarly, for $0<a<b$ and $0<a'<b'$, we set
\[
  \tilde{u}(t,x) = \E_x\left[ \ind{B_t \in [a',b']} e^{-h\int_0^t B_s ds} ; B_s \geq 0, s \in [0,t] \right]
\]
and $\tilde{\mu}_h = \int_{a'}^{b'} \psi_1^h(x) dx > 0$. By Corollary \ref{cor:asymptoticOnesided} again, we have
\[
  \lim_{t \to +\infty} \sup_{x \geq 0} \left| e^{-\frac{h^{2/3}}{2^{1/3}} \alpha_1 t}\tilde{u}(t,x) - \tilde{\mu}_h \psi^h_1(x) \right| = 0.
\]
In particular, as $\inf_{x \in [a,b]} \psi^h_1>0$, we have
\begin{equation}
  \label{eqn:no2}
  \liminf_{t \to +\infty} \inf_{x \in [a,b]} \frac{1}{t} \log \tilde{u}(t,x) = \frac{h^{2/3}}{2^{1/3}} \alpha_1.
\end{equation}
As $\tilde{u} \leq u$, mixing \eqref{eqn:no1} and \eqref{eqn:no2}, we conclude that
\[
  \lim_{t \to +\infty} \sup_{x \geq 0} \frac{1}{t} \log u(t,x) = \lim_{t \to +\infty} \inf_{x \in [a,b]} \frac{1}{t} \log \tilde{u}(t,x) = \frac{h^{2/3}}{2^{1/3}} \alpha_1.
\]
\end{proof}

\subsection{The area under a Brownian motion constrained to stay in an interval}
\label{subsec:twosided}

The main result of this section is that for all $h \in \R$, $0<a<b<1$ and $0<a'<b'<1$, we have
\begin{multline}
  \label{eqn:convergence}
  \lim_{t \to +\infty} \frac{1}{t} \sup_{x \in [0,1]} \log \E_x\left[ e^{-h \int_0^t B_s ds} ; B_s \in [0,1], s \leq t \right]\\ = 
  \lim_{t \to +\infty} \frac{1}{t} \inf_{x \in [a,b]} \log \E_x\left[ e^{-h \int_0^t B_s ds} \ind{B_t \in [a',b']} ; B_s \in [0,1], s \leq t \right] = \Psi(h),
\end{multline}
where $\Psi$ is the function defined in \eqref{eqn:definepsi}. The organisation, the results and techniques of the section are very similar to Section \ref{subsec:onesided}, with a few exceptions. First, to exhibit an orthonormal basis of eigenfunctions, we need some additional Sturm-Liouville theory, that can be found in \cite{Zet05}. Secondly, we work on $[0,1]$, which is a compact set. This lightens the analysis of the PDE obtained while proving Lemma \ref{lem:asymptoticTwosided}.

In this section, we write $\rmL^2 = \rmL^2([0,1])$ for the set of square-integrable measurable functions on $[0,1]$, equipped with the scalar product
\[
  \crochet{f,g} = \int_0^{1} f(x) g(x) dx.
\]
Moreover, we write $\calC^2_0 = \calC^2_0([0,1])$ for the set of continuous, twice differentiable functions $w$ on $[0,1]$ such that $w(0)=w(1)=0$. Finally, for any continuous function $w$, $\norme{w}_\infty = \sup_{x \in [0,1]} |w(x)|$ and $\norme{w}_2= \crochet{w,w}^{1/2}$. We introduce in a first time a new specific orthogonal basis of $[0,1]$.
\begin{lemma}
\label{lem:definephin}
Let $h>0$. The set of zeros of $\lambda \mapsto \mathrm{Ai}(\lambda) \mathrm{Bi}(\lambda + (2h)^{1/3}) - \mathrm{Ai}(\lambda + (2h)^{1/3})\mathrm{Bi}(\lambda)$ is countable and bounded from above by $0$, listed in the decreasing order $\lambda^h_1> \lambda^h_2> \cdots$. In particular, we have
\begin{equation}
  \label{eqn:defineLambda}
  \lambda^h_1 = \sup\left\{ \lambda \leq 0 : \mathrm{Ai}(\lambda) \mathrm{Bi}(\lambda + (2h)^{1/3}) = \mathrm{Ai}(\lambda + (2h)^{1/3})\mathrm{Bi}(\lambda) \right\}.
\end{equation}
Additionally, for $n \in \N$ and $x \in [0,1]$, we define
\begin{equation}
  \label{eqn:definephin}
  \phi_n^h(x) =  \frac{\mathrm{Ai}\left(\lambda^h_n\right) \mathrm{Bi}\left(\lambda^h_n + (2h)^{1/3}x\right) - \mathrm{Ai}\left(\lambda^h_n + (2h)^{1/3}x\right)\mathrm{Bi}\left(\lambda^h_n\right)}{\norme{\mathrm{Ai}\left(\lambda^h_n\right) \mathrm{Bi}\left(\lambda^h_n + \cdot\right) - \mathrm{Ai}\left(\lambda^h_n + \cdot\right)\mathrm{Bi}\left(\lambda^h_n\right)}_2}.
\end{equation}
The following properties are verified:
\begin{itemize}
  \item $(\phi^h_n, n \in \N)$ forms an orthogonal basis of $\rmL^2$;
  \item $\lim_{n \to +\infty} \lambda_n^h n^{-2} = - \frac{\pi^2}{2}$;
  \item for all $\mu \in \R$ and $\phi \in \calC^2_0$, if
\begin{equation}
  \label{eqn:sturmliouvilleTwosided}
  \begin{cases}
    \forall x \in (0,1), \frac{1}{2} \phi''(x) - h x \phi(x) = \mu \phi(x)\\
    \phi(0) = \phi(1) = 0,
  \end{cases}
\end{equation}
then either $\phi=0$, or there exist $n \in \N$ and $c \in \R$ such that $\mu = \frac{h^{2/3}}{2^{1/3}} \lambda^h_n$ and $\phi = c \phi^h_n$.
\end{itemize}
\end{lemma}

\begin{proof}
We consider equation \eqref{eqn:sturmliouvilleTwosided}. This is a Sturm-Liouville problem with separated boundary conditions, that satisfies the hypotheses of Theorem 4.6.2 of \cite{Zet05}. Therefore, there is an infinite but countable number of real numbers $(\mu^h_n, n \in \N)$ such that the differential equation
\[
  \begin{cases}
    \forall x \in (0,1), \frac{1}{2} \phi''(x) - h x \phi(x) = \mu^h_n \phi(x)\\
    \phi(0) = \phi(1) = 0,
  \end{cases}
\]
admit non-zero solutions. For all $n \in \N$, we write $\phi_n^h$ for one of such solutions normalized so that $\norme{\phi_n^h}_2=1$. For every solution $(\lambda,\phi)$ of \eqref{eqn:sturmliouvilleTwosided}, there exist $n \in \N$ and $c \in \R$ such that $\lambda = \mu_n^h$ and $\phi = c \phi_n^h$. Moreover, since $(\phi_n^h, n \in \N)$ forms an orthonormal basis of $\rmL^2$. By Theorem 4.3.1. of \cite{Zet05}, we have $\lim_{n \to +\infty} \lambda^h_n n^{-2} = - \frac{\pi^2}{2}$.

We identify $(\mu^h_n)$ and $(\phi^h_n)$. By the definition of Airy functions, given $\mu \in \R$, the solutions of
\[
  \begin{cases}
    \frac{1}{2} \phi'' (x) - h x \phi(x) = \mu \phi(x)\\
    \phi(0)=0,
  \end{cases}
\]
are, up to a multiplicative constant
\[
  x \mapsto \mathrm{Bi}\left(\tfrac{2^{1/3}}{h^{2/3}} \mu\right) \mathrm{Ai}\left(\tfrac{2^{1/3}}{h^{2/3}}\mu + (2h)^{1/3}x\right) - \mathrm{Ai}\left(\tfrac{2^{1/3}}{h^{2/3}} \mu\right) \mathrm{Bi}\left(\tfrac{2^{1/3}}{h^{2/3}}\mu + (2h)^{1/3}x\right).
\]
This function is null at point $x=1$ if and only if
\[
  \mathrm{Bi}\left(\tfrac{2^{1/3}}{h^{2/3}} \mu \right) \mathrm{Ai}\left(\tfrac{2^{1/3}}{h^{2/3}}\mu + (2h)^{1/3}\right) - \mathrm{Ai}\left(\tfrac{2^{1/3}}{h^{2/3}} \mu\right) \mathrm{Bi}\left(\tfrac{2^{1/3}}{h^{2/3}}\mu + (2h)^{1/3}\right) = 0.
\]
Therefore, the zeros of $\lambda \mapsto \mathrm{Ai}(\lambda) \mathrm{Bi}(\lambda + (2h)^{1/3}) - \mathrm{Ai}(\lambda + (2h)^{1/3})\mathrm{Bi}(\lambda)$, can be listed in the decreasing order as follows: $\lambda_1^h> \lambda_2^h > \ldots$, and we have $\lambda_n^h = \tfrac{2^{1/3}}{h^{2/3}}\mu_n^h$. Moreover, we conclude that the eigenfunction $\phi_n^h$ described above is proportional to
\[
  x \mapsto \mathrm{Ai}\left(\lambda^h_n\right) \mathrm{Bi}\left(\lambda^h_n + (2h)^{1/3}x\right) - \mathrm{Ai}\left(\lambda^h_n + (2h)^{1/3}x\right) \mathrm{Bi}\left(\lambda^h_n\right),
\]
and has $\rmL^2$ norm 1, which validates the formula \eqref{eqn:definephin}.

We have left to prove that for all $n \geq 1$, $\lambda^h_n < 0$. To do so, we observe that if $(\mu,\phi)$ is a solution of \eqref{eqn:sturmliouvilleTwosided}, we have
\begin{align*}
  \mu \int_0^1 \phi(x)^2 dx &= \int_0^1 \phi(x) \frac{1}{2} \partial^2_{x} \phi(x) - \int_0^1 x \phi(x)^2 dx\\
  &= - \frac{1}{2} \int_0^1 (\partial_x \phi(x))^2 dx - h \int_0^1 x \phi(x)^2 dx < 0,
\end{align*}
which proves that for all $n \in \N$, $\mu_n^h < 0$, so $\lambda^h_1 < 0$ which justifies \eqref{eqn:defineLambda}.
\end{proof}
We observe that once again, the eigenfunction $\phi^h_1$ corresponding to the largest eigenvector $\frac{h^{2/3}}{2^{1/3}} \lambda^h_n$ is a non-negative function on $[0,1]$, and positive on $(0,1)$.

Using this lemma, we can obtain a precise asymptotic of the Laplace transform of the area under a Brownian motion.
\begin{lemma}
\label{lem:asymptoticTwosided}
Let $h>0$ and $u_0 \in \calC^2([0,1])$ such that $u_0(0) = u_0(1) = 0$. We define, for $t,x \geq 0$
\[
  u(t,x) = \E_x\left[ u_0(B_t) e^{-h \int_0^t B_s ds} ; B_s \in [0,1], s \in [0,t] \right].
\]
We have
\begin{equation}
  \label{eqn:asymptoticTwosided}
  \lim_{t \to +\infty} \sup_{x \in \R} \left| e^{-\frac{h^{2/3}}{2^{1/3}} \lambda^h_1 t}u(t,x) - \crochet{u_0,\phi^h_1} \phi^h_1(x) \right| = 0.
\end{equation}
\end{lemma}

\begin{proof}
This proof is very similar to the proof of Lemma \ref{lem:asymptoticOnesided}. For $h>0$, by the Feynman-Kac formula, $u$ is the unique solution of the equation
\begin{equation}
  \label{eqn:fenymankacTwosided}
  \begin{cases}
    \forall t > 0, \forall x \in (0,1), \partial_t u(t,x) = \frac{1}{2} \partial^2_{x} u(t,x) - h x u(t,x)\\
    \forall x \in [0,1], u(0,x) = u_0(x)\\
    \forall t \geq 0, u(t,0) = \lim_{x \to +\infty} u(t,x) = 0.
  \end{cases}
\end{equation}
We define the operator
\[
  \calG^h : \begin{array}{rcl}
  \calC^2_0 & \to & \calC\\
  w & \mapsto & \left( x \mapsto \frac{1}{2} w''(x) - h x w(x), x \in [0,1] \right),
  \end{array}
\]
By Lemma \ref{lem:definephin}, we know that $(\phi_n^h)$ forms an orthogonal basis of $\rmL^2$ consisting of eigenvectors of $\calG^h$. In particular, for all $n \in \N$, $\calG^h \phi_n^h = \frac{h^{2/3}}{2^{1/3}} \lambda^h_n \phi_n^h$.

For any $w \in \calC^2_0$, by integration by parts, we have $\crochet{\calG^h w, \phi_n^h} = \frac{h^{2/3}}{2^{1/3}} \lambda^h_n \crochet{w,\phi_n^h}$. Decomposing $w$ with respect to the basis $(\phi_n^h)$, we obtain
\[
  \crochet{\calG^h w, w} =  \crochet{\calG^h w, \sum_{n = 1}^{+\infty} \crochet{\phi_n^h, w} \phi_n^h} = \sum_{n = 1}^{+\infty} \crochet{w, \phi_n^h}\crochet{\calG^h w,\phi_n^h}
  = \sum_{n = 1}^{+\infty}\frac{h^{2/3}}{2^{1/3}} \lambda^h_n \crochet{w, \phi_n^h}^2.
\]
As $(\lambda^h_n)$ is a decreasing sequence, we get
\begin{equation}
  \label{eqn:prems2}
  \crochet{\calG^h w, w} \leq \sum_{n = 1}^{+\infty}\frac{h^{2/3}}{2^{1/3}} \lambda^h_1 \crochet{w, \phi_n^h}^2 \leq \frac{h^{2/3}}{2^{1/3}} \lambda^h_1 \crochet{w,w}.
\end{equation}
In addition, if $\crochet{w,\phi_n^h}=0$, the inequality can be strengthened in 
\begin{equation}
  \label{eqn:deuz2}
  \crochet{\calG^h w, w} \leq \sum_{n=2}^{+\infty} \frac{h^{2/3}}{2^{1/3}} \lambda^h_2 \crochet{w, \phi_n^h}^2 \leq\frac{h^{2/3}}{2^{1/3}} \lambda^h_2 \crochet{w,w}.
\end{equation}

Using these results, we prove \eqref{eqn:asymptoticTwosided}. For $x \in [0,1]$ and $t \geq 0$, we set
\[
  v(t,x) = e^{- \frac{h^{2/3}}{2^{1/3}} \lambda^h_1 t} u(t,x) - \crochet{u_0,\phi^h_1} \phi_1^h.
\]
By definition, $\crochet{v(0,\cdot),\phi_1^h} = 0$, and for any $t \geq 0$,
\[
  \partial_t \crochet{v(t,\cdot),\phi_1^h} = - \frac{h^{2/3}}{2^{1/3}} \lambda^h_1 e^{- \frac{h^{2/3}}{2^{1/3}} \lambda^h_1 t} \crochet{u(t,x), \phi^h_1} + e^{- \frac{h^{2/3}}{2^{1/3}} \lambda^h_1 t} \crochet{\partial_t u(t,x), \phi_1^h} = 0,
\]
which proves that $\crochet{v(t,\cdot), \phi_1^h} = 0$.

We now prove that the functions $J_1(t) = \int_0^1 |v(t,x)|^2 dx$ and $J_2(t) = \int_0^1 |\partial_x v(t,x)|^2 dx$ are non-negative, decreasing, and converge to 0 as $t \to +\infty$. Note that, similarly to the previous section,
\[
  \partial_t J_1(t) = \int_0^1 2 v(t,x) \partial_t v(t,x) dx = 2 \left[- \frac{h^{2/3}}{2^{1/3}} \lambda^h_1 \crochet{v(t,\cdot),v(t,\cdot)} + \crochet{v(t,\cdot),\calG^h v(t,\cdot)} \right].
\]
As $\crochet{v(t,\cdot),\phi_1^h} = 0$, we have $\partial_t J_1(t) \leq (2h)^{2/3} (\lambda^h_2 - \lambda^h_1) J_1(t)$. Therefore, $J_1(t)$ decreases to $0$ as $t \to +\infty$. With same computations
\[
  \partial_t J_2(t) = - \frac{h^{2/3}}{2^{1/3}} \lambda^h_1 \crochet{\partial_x v(t,\cdot), \partial_x v(t,\cdot)} + 2 \crochet{\partial_x v(t,\cdot), \calG^h \partial_x v(t,\cdot)} - 2 h \int_0^1 v(t,x) \partial_x v(t,x) dx\leq 0.
\]
Thus $J_2$ is non-increasing and non-negative, therefore convergent. Observing that
\[
  \partial_t J_1(t) \leq -J_2(t) - (2h)^{2/3} \lambda^h_1 J_1(t),
\]
we conclude, with the same reasoning as in the previous section that $J_2$ decreases toward $0$.

Finally, by Cauchy-Schwarz inequality, for any $x \in [0,1]$, we have
\[
  |v(t,x)| \leq \int_0^x |\partial_x v(t,x)| dx \leq x^{1/2} \left(\int_0^x |\partial_x v(t,x)|^2 dx \right)^{1/2} \leq J_2(t),
\]
so $\lim_{t \to +\infty} \norme{v(t,\cdot)}_\infty = 0$, which proves \eqref{eqn:asymptoticOnesided}.
\end{proof}

This lemma can be easily extended to authorize more general starting function $u_0$.
\begin{corollary}
\label{cor:asymptoticTwosided}
Let $h>0$ and $u_0$ be a measurable bounded function. Setting, for $x \geq 0$ and $t \geq 0$
\[
  u(t,x) = \E_x\left[ u_0(B_t) e^{-h \int_0^t B_s ds} ; B_s \in [0,1], s \in [0,t] \right],
\]
we have
\begin{equation}
  \label{eqn:asymptoticTwosidedprime}
  \lim_{t \to +\infty} \sup_{x \in \R} \left| e^{-\frac{h^{2/3}}{2^{1/3}} \lambda^h_1 t}u(t,x) - \crochet{u_0,\phi_1^h} \phi^h_1(x) \right| = 0.
\end{equation}
\end{corollary}

\begin{proof}
Let $u_0$ be a measurable bounded function. Using the Ballot theorem, for any $\epsilon>0$, there exists $C>0$ such that for any $x \in [0,1]$, $u(\epsilon,x) \leq C \min(x,1-x)$. Consequently, $u_\epsilon=u(\epsilon,.)$ can be uniformly approached from above and from below by functions $v^{(n)},w^{(n)}$ in $\calC^2_0$. Writing
\begin{align*}
  &v^{(n)}(t,x) = \E_x\left[ v^{(n)}(x)e^{-\int_0^t B_s ds} ; B_s \in [0,1], s \in [0,t] \right]\\
  \text{and } &w^{(n)}(t,x) = \E_x\left[ w^{(n)}(x)e^{-\int_0^t B_s ds} ; B_s \in [0,1], s \in [0,t] \right],
\end{align*}
we have, for any $t \geq 0$ and $x \in [0,1]$, $v^{(n)}(t,x) \geq u(t+\epsilon,x) \geq w^{(n)}(t,x)$. Applying Lemma \ref{lem:asymptoticTwosided} and the dominated convergence theorem, we obtain
\[
  \lim_{t \to +\infty} \sup_{x \in \R} \left| e^{-\frac{h^{2/3}}{2^{1/3}} \lambda^h_1 t}u_\epsilon(t,x) - \crochet{u_\epsilon,\phi_1^h} \phi^h_1(x) \right| = 0.
\]
As for almost every $x \in [0,1]$, $\lim_{\epsilon \to 0} u_\epsilon(x) = u_0(x)$, by dominated convergence theorem again we conclude that
\[
  \lim_{t \to +\infty} \sup_{x \in \R} \left| e^{-\frac{h^{2/3}}{2^{1/3}}\lambda^h_1 t}u(t,x) - \crochet{u_0,\phi_1^h} \phi^h_1(x) \right| = 0.
\]
\end{proof}

\begin{proof}[Proof of Lemma \ref{lem:bmTwosided}]
Let $h>0$, we set $u(t,x) = \E_x\left[ e^{-h\int_0^t B_s ds} ; B_s \in [0,1], s \in [0,t] \right]$ and write $\mu_h = \int_0^{1} \phi_1^h(x)dx<+\infty$. By Corollary \ref{cor:asymptoticOnesided}, we have
\[
  \lim_{t \to +\infty} \sup_{x \geq 0} \left| e^{-\frac{h^{2/3}}{2^{1/3}} \lambda^h_1 t}u(t,x) - \mu_h \phi^h_1(x) \right| = 0.
\]
As $\phi^h_1$ is bounded,
\begin{equation}
  \label{eqn:majoBmTwosided}
  \limsup_{t \to +\infty} \sup_{x \geq 0} \frac{1}{t} \log u(t,x) = 2^{-1/3} h^{2/3} \lambda^h_1.
\end{equation}

Similarly, for $0<a<b<1$ and $0<a'<b'<1$, we set
\[
  \tilde{u}(t,x) = \E_x\left[ \ind{B_t \in [a',b']} e^{-h\int_0^t B_s ds} ; B_s \in [0,1], s \in [0,t] \right],
\]
and $\tilde{\mu}_h = \int_{a'}^{b'} \phi_1^h(x) dx > 0$. Corollary \ref{cor:asymptoticOnesided} implies that
\[
  \lim_{t \to +\infty} \sup_{x \geq 0} \left| e^{-\frac{h^{2/3}}{2^{1/3}} \lambda^h_1 t}\tilde{u}(t,x) - \tilde{\mu}_h \phi^h_1(x) \right| = 0.
\]
In particular, as $\inf_{x \in [a,b]} \phi^h_1>0$, we have
\begin{equation}
  \label{eqn:minoBmTwosided}
  \liminf_{t \to +\infty} \inf_{x \in [a,b]} \frac{1}{t} \log \tilde{u}(t,x) = \frac{h^{2/3}}{2^{1/3}} \lambda^h_1.
\end{equation}
Using the fact that $\tilde{u} \leq u$, \eqref{eqn:majoBmTwosided} and \eqref{eqn:minoBmTwosided} lead to
\[
  \lim_{t \to +\infty} \sup_{x \geq 0} \frac{1}{t} \log u(t,x) = \lim_{t \to +\infty} \inf_{x \in [a,b]} \frac{1}{t} \log \tilde{u}(t,x) = \frac{h^{2/3}}{2^{1/3}} \lambda^h_1.
\]

Moreover, by definition of $\Psi$, for all $h>0$ we have $\Psi(h) = \frac{h^{2/3}}{2^{1/3}} \lambda^h_1$, and \eqref{eqn:alternativeDefinition} is a consequence of the definition of $\lambda^h_1$. By the implicit function theorem, we observe immediately that $\Psi$ is differentiable on $(0,+\infty)$. Moreover,
\[
  \frac{\Psi(h)}{h^{2/3}} = 2^{-1/3} \lambda^h_1 = 2^{-1/3}\sup\left\{ x \in \R : \mathrm{Bi}\left(\lambda\right)\mathrm{Ai}\left(\lambda + (2h)^{1/3}\right) = \mathrm{Ai}\left( \lambda \right) \mathrm{Bi}\left( \lambda + (2h)^{1/3} \right) \right\}.
\]
Observe that $\lim_{h \to +\infty} \sup_{\lambda \geq \alpha_2} \frac{\mathrm{Bi}\left(\lambda\right)\mathrm{Ai}\left(\lambda + (2h)^{1/3}\right)}{\mathrm{Bi}\left( \lambda + (2h)^{1/3} \right)} = 0$. As $\mathrm{Ai}(\lambda^h_1) = \frac{\mathrm{Bi}(\lambda^h_1) \mathrm{Ai}(\lambda^h_1 + (2h)^{1/3})}{\mathrm{Bi}(\lambda^h_1 + (2h)^{1/3})}$, we have $\lim_{h \to +\infty} \frac{\Psi(h)}{h^{2/3}} = 2^{-1/3} \alpha_1$.

We now observe that if $h<0$, then
\begin{multline*}
  \E_x\left[ u(B_t) e^{-h\int_0^t B_s ds} ; B_s \in [0,1], s \in [0,t] \right]\\
  = e^{-ht}\E_{1-x} \left[ u(1-B_t) e^{h \int_0^t B_s ds} ; B_s \in [0,1], s \in [0,t] \right],
\end{multline*}
yielding, for any $0<a<b<1$ and $0<a'<b'<1$,
\begin{multline*}
  \lim_{t \to +\infty} \sup_{x \in [0,1]} \frac{1}{t} \log \E_x\left[ e^{-h\int_0^t B_s ds} ; B_s \in [0,1], s \in [0,t] \right]\\
  = \lim_{t \to +\infty} \inf_{x \in [a,b]} \frac{1}{t} \log \E_x\left[ \ind{B_t \in [a',b']} e^{-h\int_0^t B_s ds} ; B_s \in [0,1], s \in [0,t] \right] = -h + \Psi(-h).
\end{multline*}
Moreover, for $h<0$, $\Psi(h) = \Psi(-h)-h$.

Finally, we take care of the case $h=0$. By \cite{ItK74},
\[
 \P_x\left[ B_t \in [a,b], B_s \in [0,1], s \in [0,t] \right]
  = \int_a^b 2 \sum_{n=1}^{+\infty} e^{-n^2\frac{\pi^2}{2}t}\sin(n \pi x) \sin(n\pi z)dz
\]
As a consequence,
\begin{multline*}
 \lim_{t \to +\infty} \sup_{x \in [0,1]} \frac{1}{t} \log \P_x\left[B_s \in [0,1], s \in [0,t] \right]\\
  = \lim_{t \to +\infty} \inf_{x \in [a,b]} \frac{1}{t} \log \P_x\left[B_t \in [a',b'], B_s \in [0,1], s \in [0,t] \right] = \Psi(0) = -\frac{\pi^2}{2}.
\end{multline*}
\end{proof}

\section{Notation}
\label{app:notation}

\begin{itemize}
  \item {\em Point processes}
  \begin{itemize}
    \item $\calL_t$: law of a point process;
    \item $L_t$: point process with law $\calL_t$;
    \item $\kappa_t$: log-Laplace transform of $\calL_t$;
    \item $\kappa^*_t$: Fenchel-Legendre transform of $\calL_t$.
  \end{itemize}
  \item {\em Paths}
  \begin{itemize}
    \item $\calC$: set of continuous functions on $[0,1]$;
    \item $\calD$: set of càdlàg functions on $[0,1]$, continuous at point 1;
    \item $\bar{b}^{(n)}_k=\sum_{j=1}^k b_{j/n}$: path of speed profile $b \in \calD$;
    \item $K^*(b)_t =  \int_0^t \kappa^*_s(b_s) ds$: energy associated to the path of speed profile $b$;
    \item $\phi_t = \partial_a \kappa^*_t (b_t)$: parameter associated to the path of speed profile $b$;
    \item $E(\phi)_t = \int_0^t \phi_s \partial_\theta \kappa_s(\phi_s) - \kappa_s(\phi_s)ds$: quantity equal to $K^*(b)_t$, energy associated to the path of parameter function $\phi$;
    \item $\calR = \left\{ b \in \calD : \forall t \in [0,1], K^*(b)_t \leq 0 \right\}$: set of speed profiles $b$ such that $\bar{b}^{(n)}$ is followed until time $n$ by at least one individual with positive probability.
  \end{itemize}
  \item {\em Branching random walk}
  \begin{itemize}
    \item $\T$: genealogical tree of the process;
    \item $u \in \T$: individual in the process;
    \item $V(u)$: position of the individual $u$;
    \item $|u|$: generation at which $u$ belongs;
    \item $u_k$: ancestor of $u$ at generation $k$;
    \item $\emptyset$: initial ancestor of the process;
    \item if $u \neq \emptyset$, $\pi u$: parent of $u$;
    \item $\Omega(u)$: set of children of $u$;
    \item $L^u = (V(v)-V(u), v \in \Omega(u))$: point process of the displacement of the children;
    \item $M_n = \max_{|u|=n} V(u)$ maximal displacement at the $n^\text{th}$ generation in $(\T,V)$;
    \item $\Lambda_n = \min_{|u|=n} \max_{k \leq n} \bar{b}^{(n)}_k - V(u_k)$: consistent maximal displacement with respect to the path $\bar{b}^{(n)}$;
    \item $\calW^\phi_n = \{ u \in \T : \forall k \in F_n, V(u_k) \geq \bar{b}^{(n)}_k + f_{k/n} n^{1/3} \}$: tree of a BRWtie with selection above the curve $\bar{b}^{(n)}_. + n^{1/3}f_{./n}$ at times in $F_n$.
  \end{itemize}
  \item {\em The optimal path}
  \begin{itemize}
    \item $v^* = \sup_{b \in \calR} \int_0^1 b_s ds$: speed of the BRWtie;
    \item $a \in \calR$ such that $\int_0^1 a_s ds = v^*$: optimal speed profile;
    \item $\theta_t = \partial_a \kappa^*_s(a_s)$: parameter of the optimal path;
    \item $\sigma^2_t = \partial^2_\theta \kappa_t(\theta_t)$: variance of individuals following the optimal path;
    \item $\dot{\theta}$: Radon-Nikod\'ym derivative of $d\theta$ with respect to the Lebesgue measure;
    \item $l^*= \alpha_1 2^{-1/3}\int_0^1 \frac{\left(\dot{\theta}_s \sigma_s\right)^{2/3}}{\theta_s} ds$: $n^{1/3}$ correction of the maximal displacement;
    \item $v_t = \inf_{\theta > 0} \frac{\kappa_t(\theta)}{\theta}$: natural speed profile.
  \end{itemize}
  \item {\em Airy functions}
  \begin{itemize}
    \item $\mathrm{Ai}(x) = \frac{1}{\pi} \int_0^{+\infty} \cos\left( \tfrac{s^3}{3} + x s \right) ds$: Airy function of the first kind;
    \item $\mathrm{Bi}(x) = \frac{1}{\pi} \int_0^{+\infty} \exp\left( - \tfrac{s^3}{3} + x s \right) + \sin\left( \tfrac{s^3}{3} + x s \right) ds$: Airy function of the second kind;
    \item $(\alpha_n)$: zeros of $\mathrm{Ai}$, listed in the decreasing order;
    \item $\Psi(h) = \lim_{t \to +\infty} \frac{1}{t} \log \sup_{x \in [0,1]} \E_x\left[ e^{-h\int_0^t B_s ds} ; B_s \in [0,1], s \in [0,t] \right]$.
  \end{itemize}
  \item {\em Random walk estimates}
  \begin{itemize}
    \item $(X_{n,k}, n \in \N, k \leq n)$: array of independent random variables;
    \item $S^{(n)}_k = S^{(n)}_0 + \sum_{j=1}^k X_{n,j}$: time-inhomogeneous random walk, with $\P_x(S^{(n)}_0 = x) = 1$;
    \item Given $f,g \in \calC$, and $0 \leq j \leq n$,
\[
  I_n(j) =
  \begin{cases}
    \left[f_{j/n}n^{1/3}, g_{j/n}n^{1/3}\right] & \text{if } j \in F_n \cap G_n,\\
    \left[f_{j/n},+\infty\right[ & \text{if } j \in F_n \cap G_n^c,\\
    \left]-\infty, g_{j/n}n^{1/3}\right] & \text{if } j \in F_n^c\cap G_n,\\
    \R & \text{otherwise;}
  \end{cases}
\]
    \item $\tilde{I}^{(n)}_j = I^{(n)}_j \cap [-n^{2/3}, n^{2/3}]$.
  \end{itemize}
  \item {\em Many-to-one lemma}
  \begin{itemize}
    \item $\P_{k,x}$: law of the BRWtie of length $n-k$ with environment $(\calL_{(k+j)/n},j \leq n-k)$;
    \item $\calF_k = \sigma(u,V(u), |u| \leq k)$: filtration of the branching random walk;
    \item Many-to-one lemma: Lemma \ref{lem:manytoone}.
  \end{itemize}
  \item {\em Random walk estimates}
  \begin{itemize}
    \item $\calA_n^{F,G}(f,g) = \left\{ u \in \T, |u| \in G_n : V(u) - \bar{b}^{(n)}_{|u|} > g_{|u|/n}n^{1/3} , V(u_j) - \bar{b}^{(n)}_j \in I^{(n)}_j , j < |u|\right\}$: individuals staying in the path $\bar{b}^{(n)} + I^{(n)}$ until some time then exiting by the upper boundary;
    \item $\calB_n^{F,G} (f,g,x) = \left\{ |u|=n : V(u_j) - \bar{b}^{(n)}_j \in \tilde{I}^{(n)}_j, j \leq n, V(u)-\bar{b}^{(n)}_n\geq (g_1-x)n^{1/3}\right\}$: individuals staying in the path $\bar{b}^{(n)} + I^{(n)}$, and being above $\bar{b}^{(n)}_n + (g_1-x)n^{1/3}$ at time $n$.
  \end{itemize}
\end{itemize}

\nocite{*}

\bibliographystyle{plain}
\def\cprime{$'$}

\end{document}